\documentclass{article}

\usepackage{graphicx}
\usepackage{amsfonts}
\usepackage{amsmath}
\usepackage{amssymb}
\usepackage{fancyhdr}
\usepackage{titlesec}
\usepackage{indentfirst}
\usepackage{booktabs}
\usepackage{verbatim}
\usepackage{color}
\usepackage{latexsym}
\usepackage{amscd}
\usepackage{esvect}
\usepackage{graphicx}
\usepackage{upgreek}
\usepackage{subcaption}
\usepackage{caption}
\usepackage[page,header]{appendix}
\usepackage{titletoc}
\usepackage{mathrsfs}
\usepackage[numbers]{natbib}
\usepackage{float}
\usepackage{lipsum}
\usepackage[export]{adjustbox}
\usepackage{amsthm}
\usepackage{appendix}

\usepackage{amssymb}
\usepackage{appendix}
\usepackage{xspace}
\usepackage{bold-extra}
\usepackage[most]{tcolorbox}

\colorlet{texcscolor}{blue!50!black}
\colorlet{texemcolor}{red!70!black}
\colorlet{texpreamble}{red!70!black}
\colorlet{codebackground}{black!25!white!25}

\numberwithin{equation}{section}
\newtheorem{theorem}{Theorem}[section]

\newtheorem{remark}[theorem]{Remark}

\usepackage{amsmath}

\begin{document}

\title{Micro-macro decomposition based asymptotic-preserving numerical schemes and numerical moments conservation for collisional nonlinear kinetic equations 
 \footnote{The first and the third author was supported by the funding DOE--Simulation Center for Runaway Electron Avoidance and Mitigation, 
project No. DE-SC0016283. The second author was supported by NSF grants DMS-1522184 and DMS-1107291: RNMS KI-Net, and NSFC grants No. 31571071 and No. 11871297. 
}}

\date{}
\author{Irene M. Gamba \footnote{Department of Mathematics and The Institute for Computational Engineering and Sciences (ICES), University of Texas at Austin, Austin, TX 78712, USA (gamba@math.utexas.edu).}, 
Shi Jin \footnote{School of Mathematical Sciences, Institute of Natural Sciences,
MOE-LSC and SHL-MAC, Shanghai Jiao Tong University, Shanghai, China (shijin-m@sjtu.edu.cn).},
and Liu Liu\footnote{The Institute for Computational Engineering and Sciences (ICES), University of Texas at Austin, Austin, TX 78712, USA (lliu@ices.utexas.edu). }}
\maketitle

\abstract{
In this paper, we first extend the micro-macro decomposition method for multiscale kinetic equations from the BGK model to general collisional kinetic equations, including the Boltzmann and the Fokker-Planck Landau equations. 
The main idea is to use a relation between the (numerically stiff) linearized collision operator with the nonlinear quadratic ones, the latter's stiffness can be overcome using the BGK penalization method of Filbet and Jin for the Boltzmann, or the linear Fokker-Planck penalization method of Jin and Yan for the Fokker-Planck Landau equations. Such a scheme allows the computation of 
multiscale collisional kinetic equations efficiently in all regimes, including the fluid regime in which the fluid dynamic behavior can be correctly computed even without resolving the small Knudsen number. A distinguished feature of these schemes is that although they contain implicit terms, they can be implemented {\it explicitly}.
These schemes preserve the moments (mass, momentum and energy) {\it exactly} thanks to the use of the macroscopic system which is naturally in a conservative form. We further utilize this conservation property for more general kinetic systems, 
using the Vlasov-Amp\`{e}re and Vlasov-Amp\`{e}re-Boltzmann systems as examples. 
The main idea is to evolve both the kinetic equation for the probability density distribution and the moment system, the later naturally induces a scheme that conserves exactly the moments numerically if they are physically conserved. 
}

{\bf keywords: } Boltzmann equation, Landau equation, micro-macro decomposition, asymptotic preserving scheme, conservative scheme,
Vlasov-Amp\`{e}re-Boltzmann

\section{Introduction}

The Boltzmann equation and the Fokker-Planck-Landau equation are among the most important
kinetic equations, arising in describing the dynamics of probability density
distribution of particles in rarified gas and plasma, respectively. 
One of the main computational challenges for these kinetic equations is that the problem may often encounter multiple time and spatial scales, 
characterized by the Knudsen number (denoted by $\varepsilon$), the dimensionless mean free
path, that may vary in orders of magnitude in the computational domain,
covering the regimes from fluid, transition, rarefied to even free streaming
regimes. Asymptotic-Preserving (AP) schemes, which mimics the asymptotic
transition from one scale to another at the discrete level, have been shown to be an effective
computational paradigm in the last two decades \cite{jin1999efficient,
jin2010asymptotic}.  Such schemes allow efficient numerical approximations
in {\it all} regimes, and coarse mesh and large time steps can be used even
in the fluid dynamic regime, without numerically resolving the small
Knudsen number. For the space inhomogeneous Boltzmann equation, AP schemes
were first designed using BGK-operator based penalty \cite{Filbet-Jin}.
Other approaches include the exponential
integrator based methods \cite{dimarco2011exponential, li2014exponential},
or micro-macro (MM) decomposition \cite{MM-Lemou}. We also mention relevant
works \cite{xu2010unified, liu2016unified, Luc}. One should note that
\cite{MM-Lemou, xu2010unified} only dealt with the BGK model, rather than
the full Boltzmann equation. For AP schemes to deal with the stiff Landau collision operator, 
the BGK-penalization method was extended to the Fokker-Planck-Landau
equation in \cite{JinYan}, using the linear Fokker-Planck operator as the
penalty.

The aim of this paper is not on the comparison of all these different approaches,
rather we will focus on the micro-macro decomposition method, which was
formulated in \cite{MM-Lemou} for the Boltzmann but numerically realized only for the BGK model. 
One of the difficulties in this formulation is that one encounters a stiff linearized collision operator whose inversion could be computationally inefficient.
In \cite{Lemou-Note}, a linear penalty for the microscopic deviation equation was used to remove the stiffness. 
This idea is suitable for the Boltzmann equation but not for the Fokker-Planck-Landau equation which has second derivative terms 
in the collision operators.
One of the goals of the current paper is to show how the micro-macro decomposition method in \cite{MM-Lemou}
can be extended to the general collision operators,  
include the Boltzmann and Landau collision operators.
Having its theoretical origin in \cite{liu2004boltzmann} (see also \cite{liu2006nonlinear}),
the micro-macro decomposition has also found its advantage in designing
AP schemes for radiative heat transfer \cite{klar2001numerical}, linear
transport equation \cite{Lemou-BC}, among others. 
For the Boltzmann or the Fokker-Planck-Landau equation, 
the MM method is based on a decomposition of the kinetic equation under study 
into a coupled system composed of a kinetic equation on the microscopic part and a fluid equation on the macroscopic part. 
By using an implicit-explicit (IMEX) temporal discretization, it naturally leads to an AP scheme at the level of the compressible Navier-Stokes asymptotics 
\cite{MM-Lemou}. Moreover, the MM formulation guarantees the preservation of numerical moments (mass, momentum and energy) {\it exactly}
thanks to the macroscopic system which is naturally in a conservative form.
Another advantage of the MM approach is that one can obtain good uniform numerical
stability result \cite{liu2010analysis, LFY-DG}. 

Our main idea for the MM method is the usage of a simple relation between a linearized
collision operator (a numerically stiff term) and the quadratically
nonlinear collision operator. For the latter (stiff) nonlinear collision operators, we then use the BGK-penalty
method of Filbet-Jin \cite{Filbet-Jin} for the Boltzmann collision or
the Fokker-Planck penalty of Jin-Yan \cite{JinYan} for the Fokker-Planck-Landau collision.
This allows us to extend the MM method of \cite{MM-Lemou} from the BGK
model to the more physical Boltzmann and Fokker-Planck-Landau equations in a rather simple fashion.

We would like to emphasize that in the MM formalism (as well as in the
penalty methods in \cite{Filbet-Jin, JinYan}), 
one needs to solve the macroscopic system, which is in a conservation
form, giving rise the conservation of mass, momentum and total energy. 
When discretizing the macroscopic system with a standard spatially conservative 
scheme, these physically conserved quantities are naturally conserved
numerically. This is not the case if one uses the microscopic equation for
the particle density distribution $f$ and then takes moments from the discrete $f$,
since many collision solvers, for example the spectral methods
\cite{gamba2017fast, pareschi2000numerical,gamba2009spectral, mouhot2006fast}, do not
have the {\it exact} conservation properties, and extra efforts are needed for
the exact conservation, see \cite{mieussens2000discrete, zhang2017conservative, gamba2014conservative}. 
The advantage of the conservation of moments made from the {\it macro} system was noted and emphasized in
\cite{JinYan}. 

Numerically conserving the physically conserved quantities is a highly desirable property for a numerical scheme.
In this paper we realized this using the moment systems which are already in conservation forms at the continuous level. 
Note that although not all AP schemes use the moment system, some of the popular ones, like those in Filbet-Jin \cite{Filbet-Jin}, Jin-Yan \cite{JinYan}
and micro-macro decomposition based method \cite{MM-Lemou}, indeed use it thus naturally induce the exactly conservative schemes. This fact was pointed out and utilized in \cite{JinYan}. In Section \ref{sec:7} we further extend this idea to design {\it conservative} schemes
for {\it general} (collisional or non-collisional) kinetic systems, using
the Vlasov-Poisson and Vlasov-Poisson-Boltzmann systems as examples. 
The general principle favored here is that one should solve the original kinetic equation and the moment system {\it simultaneously}. 
One first obtains the moment system
analytically and then the discrete moment system, when using spatially
conservative discretizations, {\it automatically} yields the exact conservations
of moments, if they are conserved physically. Since the total energy also includes the electric energy, 
another idea introduced here is to replace the Poisson equation for the electric field by the Amp\`{e}re equation, 
and then the coupled system is discretized in time by a carefully designed explicit-implicit scheme. 

This paper is organized as follows. Section \ref{sec:Intro} gives an introduction of two kinetic equations: the Boltzmann and the Fokker-Planck-Landau equations. In Section \ref{sec:micro}, the basic idea of the micro-macro decomposition method is reviewed. Section \ref{sec:NA} studies the fully discretized AP numerical scheme, especially on how to embed the penalization method in the micro-macro decomposition framework to solve the full nonlinear Boltzmann and Fokker-Planck-Landau equations. We also emphasize that our scheme conserves the moments (mass, momentum, energy) if these moment variables are obtained from the macroscopic system instead of from 
the particle density distribution $f$. 
Section \ref{Sec:Num} provides some implementation details, while in Section \ref{sec:NE} a number of numerical examples are used to study the conservation property as well as the performance of the new schemes in different regimes. 
In Section \ref{sec:7} we introduce conservation schemes for the Vlasov-Amp\'ere system and Vlasov-Amp\'re-Boltzmann system, with the conservations obtained through solving the moment systems and a specially designed time discretization. 
Finally, we conclude and list some future work in Section \ref{sec:FW}.

\section{Introduction of two kinetic equations}
\label{sec:Intro}

\subsection{The Boltzmann equation}

One of the most celebrated kinetic equations for rarefied gas is the Boltzmann equation, 
which describes the time evolution of the density distribution of a dilute gas of particles 
when the only interactions considered are binary elastic collisions. A dimensionless form reads
\begin{equation}\label{Boltz}
\partial_t f + v\cdot \nabla_x f= \frac{1}{\varepsilon}\, \mathcal Q_{\text{B}}(f, f), \qquad t>0, \, (x,v)\in\Omega\times\mathbb R^d, 
\end{equation}
where $f(t,x,v)$ is the probability density distribution (p.d.f) function, modeling the probability of finding a particle at time $t$, at position $x\in\Omega$, 
with velocity $v\in\mathbb R^d$. The parameter $\varepsilon$ is the Knudsen number 
defined as the ratio of the mean free path over a typical length scale such as the size of the spatial domain, which 
characterizes the degree of rarefaction of the gas. 
The Boltzmann collision operator is denoted by $\mathcal Q_{\text{B}}$, which is a bilinear functional 
and only acts on the velocity dependence of $f$, 
\begin{equation} \mathcal Q_B(f, g)(t,x,v) = \int_{\mathbb R^d}\int_{\mathbb S^{d-1}}\, 
B(|v-v_{\ast}|, \cos\theta) \left(f(t,x,v^{\prime})g(t,x,v_{\ast}^{\prime}) - f(t,x,v)g(t,x,v_{\ast})\right) d\sigma dv_{\ast}\,. 
\end{equation}
We consider the elastic interaction. The velocity pairs before and after the collision $(v, v_{\ast})$ and $(v^{\prime}, v_{\ast}^{\prime})$ have the relation, 
\begin{align}
\begin{cases}
&\displaystyle v^{\prime}= \frac{v+v_{\ast}}{2} + \frac{|v-v_{\ast}|}{2}\, \sigma, \\[6pt]
&\displaystyle v_{\ast}^{\prime} = \frac{v+v_{\ast}}{2} - \frac{|v-v_{\ast}|}{2}\, \sigma. 
\end{cases}
\end{align}
Here $\sigma$ is the scattering direction varying in the unit sphere $\mathbb S^{d-1}$, and is defined by 
$$\sigma = \frac{u^{\prime}}{|u^{\prime}|} = \frac{u^{\prime}}{|u|}, $$
where the pre- and post-collisional relative velocities $u=v-v_{\ast}$ 
and $u^{\prime}=v^{\prime}-v_{\ast}^{\prime}$ have the same magnitude, i.e., $|u^{\prime}|=|u|$. 

Cosine of the deviation angle is given by 
$$\cos\theta = \frac{u\cdot u^{\prime}}{|u|^2} = \frac{u\cdot \sigma}{|u|} := \hat u \cdot\sigma\,. $$
The collision kernel $B$ is a non-negative function, which is usually written in a form of a product 
of a power function of the relative velocity $u$ and a scattering angular function $b$ depending on $\cos\theta$, that is, 
\begin{equation} B(|v-v_{\ast}|, \cos\theta) = B(|u|, \hat u \cdot\sigma) = 
C_{\lambda}\, |u|^{\lambda}\, b(\hat u \cdot\sigma), \qquad -d\leq \lambda\leq 1. 
\end{equation} 
Here $\lambda>0$ corresponds to the hard potentials, $\lambda<0$ the soft potentials, 
and $\lambda=0$ refers to the Maxwell pseudo-molecules model. 

It is not hard to find that 
\begin{equation}\label{weak}
\int_{\mathbb R^d}\, \mathcal Q_B(f, f)(v)\phi(v)\, dv = 
\frac{1}{2} \int_{\mathbb R^d}\, f f_{\ast}\left(\phi + \phi_{\ast} - \phi^{\prime} - \phi_{\ast}^{\prime}\right) B(|v-v_{\ast}|, \cos\theta)\, 
d\sigma dv_{\ast} \end{equation}
equals to zeros if
\begin{equation}\label{phi} \phi + \phi_{\ast} = \phi^{\prime} + \phi_{\ast}^{\prime}. 
\end{equation}
One can prove that (\ref{phi}) holds if and only if $\phi(v)$ lies in the space 
spanned by the moments of mass, momentum and kinetic energy. 
We call the $d+2$ test functions $1, \, v, \, \frac{|v|^2}{2}$
{\it collision invariants} associated to $\mathcal Q_{B}$. 
Denote $$m(v) = \left(1, v, \frac{|v|^2}{2}\right)^{T}, $$ 
then \begin{equation}\label{Q_cons} \int_{\mathbb R^d}\, \mathcal Q_{B}(f, f)m(v)\, dv = 0, 
\end{equation}
which correspond to the conservation of mass, momentum and kinetic energy of $\mathcal Q_{B}$.

Define $U=(\rho, \rho u, E)^{T}$ as the velocity averages of $f$ multiplying by the collision invariants $m$, which is a
vector composing of $d+2$ conserved moments of density, momentum and energy, 
\begin{equation}\label{U_eqn} \langle m M(U)\rangle = U 
= \int_{\mathbb R^d}\begin{pmatrix}1 \\ v \\ \frac{1}{2}|v|^2 \end{pmatrix}f(v)dv = \begin{pmatrix}\rho \\ \rho u \\
\frac{1}{2}\rho\, |u|^2 + \frac{d}{2}\rho\, T \end{pmatrix} =
\begin{pmatrix} \rho \\ \rho u \\ E \end{pmatrix}.  \end{equation}

If setting $\phi(v) = \ln f(v)$ in (\ref{weak}), one can prove the following dissipation of entropy
\begin{equation}\label{H-thm1}\int_{\mathbb R^d}\, \mathcal Q_{B}(f, f)\ln f\, dv \leq 0, 
\end{equation}
which is known as the celebrated Boltzmann's H-theorem. 
Furthermore, the Boltzmann theorem for elastic interaction is given by
\begin{equation}\label{H-thm2} \int_{\mathbb R^d}\, \mathcal Q_{B}(f, f)\ln f\, dv=0  \, \Leftrightarrow\,  
\mathcal Q_{B}(f, f)=0  \, \Leftrightarrow\, f = M, 
\end{equation}
where $M$ is the equilibrium state given by a {\it Maxwellian distribution}
\begin{equation}\label{Max} M(U)(v) = \frac{\rho}{(2\pi T)^{\frac{d}{2}}}\exp\left(-\frac{|v-u|^2}{2T}\right)
:= M_{U(x,t)}(v)\,.  \end{equation}
Here $\rho$, $u$ and $T$ are respectively the density, bulk velocity, and temperature defined by 
$$\rho = \int_{\mathbb R^d}\, f(v)\, dv, \qquad u = \frac{1}{\rho}\, \int_{\mathbb R^d}\, f(v)v\, dv, 
\qquad T = \frac{1}{d \rho}\, \int_{\mathbb R^d}\, f(v)|v-u|^2\, dv. $$
\\[2pt]

{\bf The fluid limit}\, 
We introduce the notation $\langle\, \cdot\, \rangle$ as the velocity averages of the argument, i.e., 
$$ \langle f \rangle = \int_{\mathbb R^d}\, f(v)\, dv. $$
Multiplying (\ref{Boltz}) by $m(v)$ and integrating with respect to $v$, by using the conservation property of 
$\mathcal Q_{\text{B}}$ given by (\ref{Q_cons}), one has
$$ \partial_t \langle m f \rangle + \nabla_x\cdot \langle v m f \rangle = 0. $$
This gives a non-closed system of conservation laws
\begin{equation}\label{INS} \partial_t \begin{pmatrix} \rho \\ \rho u \\ E \end{pmatrix}
+ \nabla_x \cdot  \begin{pmatrix} \rho u \\ \rho u \otimes u + \mathbb P \\
E u + \mathbb P u + \mathbb Q  \end{pmatrix} = 0, 
\end{equation}
where $E$ is the energy defined in (\ref{U_eqn}), 
$\mathbb P = \langle (v-u)\otimes (v-u)f \rangle$ is the pressure tensor, 
and $\mathbb Q = \frac{1}{2}\langle (v-u) |v-u|^2 f \rangle$ is the heat flux vector. 
When $\varepsilon\to 0$, $f\to M(U)$. Replacing $f$ by $M(U)$ and using expression 
(\ref{Max}), $\mathbb P$ and $\mathbb Q$ are given by 
$$ \mathbb P = p\, I, \qquad \mathbb Q = 0, $$
where $p = \rho T$ is the pressure, $I$ is the identity matrix. Then (\ref{INS}) reduces to 
the usual compressible Euler equations
\begin{equation}\label{Euler}
\partial_t \begin{pmatrix} \rho \\ \rho u \\ E \end{pmatrix} + 
\nabla_x \cdot \begin{pmatrix} \rho u \\ \rho u\otimes u + p\, I \\ (E+p)\, u \end{pmatrix} = 0. 
\end{equation}
\subsection{The Fokker-Planck-Landau equation}
The nonlinear Fokker-Planck-Landau (nFPL) equation is widely used in plasma physics. The rescaled nFPL equation reads
\begin{equation}\label{LD}
\partial_t f + v\cdot \nabla_x f= \frac{1}{\varepsilon}\, \mathcal Q_{\text{L}}(f, f), \qquad t>0, \, (x,v)\in\Omega\times\mathbb R^d, 
\end{equation}
with the nFPL operator 
\begin{equation} \mathcal Q_{L}(f, f) = \nabla_v \cdot \int_{\mathbb R^d}\, A(v-v_{\ast})
\left(f(v_{\ast})\, \nabla_v f(v) - f(v)\, \nabla_v f(v_{\ast})\right) dv_{\ast}\,,
\end{equation}
where the semi-positive definite matrix $A(z)$ is 
$$ A(z) = \Psi(z)\, \left( I -\frac{z\otimes z}{|z|^2} \right), \qquad \Psi(z) = |z|^{\gamma+2}\,. $$
The parameter $\gamma$ characterizes the type of interaction between particles. 
The inverse power law gives $\gamma\geq -3$. 
Similar to Boltzmann collision operator,  $\gamma>0$ categorizes hard potentials, 
$\gamma=0$ for Maxwellian molecules and $\gamma<0$ for soft potentials. 
The case $\gamma=-3$ corresponding to Coulomb interactions. 

The nFPL equation is derived as a limit of the Boltzmann equation when all the collisions become grazing. 
Therefore, the nFPL operator possesses similar conservation laws and decay of entropy (H-theorem) as the Boltzmann collision operator, which are
given in (\ref{H-thm1})-(\ref{H-thm2}). 

\section{The micro-macro decomposition method}
\label{sec:micro}
When no confusion is possible, we set $M_{U(x,t)}(v)=M$ in the following. 
Consider the Hilbert space $L^2_{M}=\left\{\phi \, \big|\,  \phi\, M^{-\frac{1}{2}}\in L^2 (\mathbb R^d)\right\}$ endowed with the weighted scalar product
$$ (\phi, \, \psi)_{M}= \langle \phi\, \psi\, M^{-1}\rangle. $$
It is well-known that the linearized operator $\mathcal L_{M}$ is a non-positive self-adjoint operator on $L^2_{M}$ and that its 
null space is 
$$\mathcal N(\mathcal L_{M}) =\text{Span}\left\{M, \, |v| M,\, |v|^2 M\right\}, $$
whose orthogonal basis is
$$\mathcal B = \left\{\frac{M}{\rho}, \, \frac{(v-u)}{\sqrt{T}}\frac{M}{\rho}, \, \left(\frac{|v-u|^2}{2T}-\frac{d}{2}\right)\frac{M}{\rho}\right\}. $$
The orthogonal projection of $\phi\in L^2_{M}$ onto $\mathcal N(\mathcal L_{M})$ is given by $\Pi_{M}(\phi)$: 
$$\Pi_{M}(\phi)= \frac{1}{\rho}\left[\langle\phi\rangle + \frac{(v-u)\cdot \langle(v-u)\phi\rangle}{T} + \left(\frac{|v-u|^2}{2T}-\frac{d}{2}\right)\frac{2}{d}
\left\langle\left(\frac{|v-u|^2}{2T}-\frac{d}{2}\right)\phi\right\rangle\right] M. $$ 

We explain the main idea of the micro-macro decomposition, which mostly follows that in \cite{MM-Lemou}, where the BGK  equation, with 
$\mathcal Q_{BGK}(f, f) = \frac{1}{\tau}(M - f)$, is numerically implemented ($\tau$ is the relaxation time). 
Let $f$ be the solution of the Boltzmann equation (\ref{Boltz}). We decompose $f = f(t,x,v)$ as
\begin{equation}\label{ansatz}
f = M + \varepsilon g(x,t,v)
\end{equation}
where $U$ and $M$ are given in (\ref{U_eqn}) and (\ref{Max}) respectively. 
Inserting (\ref{ansatz}) into (\ref{Boltz}), one obtains 
$$\partial_t M + v\cdot\nabla_x M + \varepsilon (\partial_t g + v\cdot\nabla_x g)=\frac{1}{\varepsilon}\mathcal Q(M+\varepsilon g, M+ \varepsilon g). $$
Denote the linearized collision operator \begin{equation}\mathcal L_{M}(g)=2\mathcal Q(M, g). \label{LL}\end{equation}
Since $\mathcal Q$ is bilinear and $\mathcal Q(M,M)=0$, then 
$$\mathcal Q(M+\varepsilon g, M+ \varepsilon g)=\mathcal Q(M,M) + 2\varepsilon\mathcal Q(M,g) + \varepsilon^2 \mathcal Q(g,g)= \varepsilon\mathcal L_{M}(g) + \varepsilon^2 \mathcal Q(g,g), $$
thus 
\begin{equation}\label{M_1}\partial_t M + v\cdot\nabla_x M + \varepsilon (\partial_t g + v\cdot\nabla_x g) = \mathcal L_{M}(g) + \varepsilon \mathcal Q(g,g). 
\end{equation}
Applying the operator $\mathbb I - \Pi_{M}$ to (\ref{M_1}), one gets
\begin{equation}\label{g}\partial_t g + (\mathbb I - \Pi_{M})(v\cdot\nabla_x g)- \mathcal Q(g,g)= \frac{1}{\varepsilon}\left[\mathcal L_{M}(g) - (\mathbb I -\Pi_{M})(v\cdot\nabla_x M)\right]. \end{equation}
On the other hand, if we take the moments of equation (\ref{M_1}), then 
\begin{equation}\label{M_2}\partial_t \langle mM \rangle + \nabla_x \cdot \langle v m M\rangle + \varepsilon \nabla_x \cdot \langle vmg\rangle =0. \end{equation}
Denote the flux vector of $U$ by 
$$F(U)=\langle v m M\rangle = \begin{pmatrix} \rho u \\ \rho u \otimes u + \rho T \\ E u + \rho T u \end{pmatrix}, $$ then (\ref{M_2}) becomes 
\begin{equation}\label{U}\partial_t U  + \nabla_x \cdot F(U) +  \varepsilon \nabla_x \cdot \langle vmg\rangle =0. \end{equation}
Therefore, the coupled system (\ref{g}) and (\ref{U}) gives  a kinetic/fluid formulation of the Boltzmann equation. 
It has been shown in \cite{MM-Lemou} that this coupled system is equivalent to the Boltzmann equation (\ref{Boltz}).
\\[4pt]

{\bf Initial and boundary conditions} \\
For the initial condition, we set $$ f(t=0, x, v) = f^{0}(x, v). $$
$x$ is in a bounded set $\Omega$ with boundary $\Gamma$. 
For the numerical implementation purpose, we only consider the periodic boundary condition (BC) in $x$ in this paper. Nevertheless, 
we briefly mention other types of BC. 

For points $x$ on the boundary $\Gamma$, the distribution function of incoming velocities 
(i.e., $v$ with $v\cdot n(x)<0$, where $n(x)$ is the outer normal vector of $\Gamma$ at $x$) should be specified. 
The Dirichlet BC reads
\begin{equation}\label{BC1} f(t,x,v) = f_{\Gamma}(t,x,v) \qquad \forall x\in\Gamma, \,  \forall v,  \,  \text{s.t.}\, v\cdot n(x)<0. 
\end{equation}
The reflecting BC is given by 
\begin{equation}\label{BC2} f(t,x,v) = \int_{v^{\prime}\cdot n(x)>0}\, K(x,v,v^{\prime})f(t,x,v^{\prime})\, dv^{\prime} \qquad 
\forall x\in\Gamma, \, \forall v, \, \text{s.t.}\, v\cdot n(x)<0, 
\end{equation}
where the kernel $K$ satisfies the zero normal mass flux condition
across the boundary: 
$$ \int_{\Gamma}\, v\cdot n(x) f(t,x,v)\, dv = 0. $$
The periodic BC can be used when the shape of $\Omega$ is symmetric, 
$$ f(t,x,v) = f(t, Sx, v), \qquad x\in\Gamma_1, \, \forall v, $$ 
where S is a one-to-one mapping from a part $\Gamma_1$ of $\Gamma$ onto another part $\Gamma_2$ of $\Gamma$. 

In general, using the micro-macro decomposition into boundary conditions (\ref{BC1})-(\ref{BC2})
provides relations for $M + \varepsilon g$, but do not provide the values for $M$ and $g$ separately. 
Moreover, $f$ is generally known only for incoming velocities at boundary points, which may induce difficulties to define the macroscopic
moments $U$. Note that various numerical boundary conditions based on micro-macro formulation for {\it linear} kinetic equations in the diffusion limit 
is studied in \cite{Lemou-BC}. 

\section{Numerical Approximation}
\label{sec:NA}

\subsection{Time discretization}
We denote $\Delta t$ a fixed time step, $t_n$ a discrete time with $t_n=n\Delta t$, $n\in \mathbb N$. 
Let $U^n(x)\approx U(t_n, x)$, $g^n(x,v)\approx g(t_n, x, v)$. Note that in equation (\ref{g}), $\varepsilon^{-1}\mathcal L_{M}(g)$ is the only collision term
that presents the stiffness, thus one needs to take an implicit discretization for this term, while the term $(I -\Pi_{M})(v\cdot \nabla_x M)$ is still explicit. 
The time discretization for (\ref{g}) is given by 
\begin{equation}\label{g_discrete}\frac{g^{n+1}-g^n}{\Delta t} +
 (\mathbb I - \Pi_{M^n})(v\cdot\nabla_x g^n) - \mathcal Q(g^n,g^n)=
 \frac{1}{\varepsilon}\left[\mathcal L_{M^n}(g^{n+1}) - (\mathbb I -\Pi_{M^n})(v\cdot\nabla_x M^n)\right]. 
 \end{equation}

For the time discretization of the fluid part (\ref{U}), the flux $F(U)$ at time $t_n$ is approximated by $F(U^n)=\langle v m M^n \rangle$, 
and the convection term $\nabla_x \cdot \langle v m g\rangle$ is discretized by $\nabla_x \cdot \langle v m g^{n+1}\rangle$, 
\begin{equation}\label{U_discrete}\frac{U^{n+1}-U^n}{\Delta t}  + \nabla_x \cdot F(U^n) +  
\varepsilon\nabla_x \cdot \langle v m g^{n+1}\rangle =0. 
\end{equation}

In \cite{MM-Lemou} only BGK collision operator was considered, thus avoided
the difficulty of inverting the $\mathcal L_{M^n}(g^{n+1})$ term in (\ref{g_discrete}),  
since the implicit BGK operator can be inverted explicitly, thanks
to the conservation property of the operator due to (\ref{U_eqn}). For general
collision operator this is no longer true. In the next subsection, we propose
an efficient method to deal with the term $\mathcal L_{M^n}(g^{n+1})$, which
is one of the main ideas of this paper.

\subsection{AP schemes by penalization}

To avoid the complication of inverting the stiff, implicit linearized collision operator $\mathcal L_{M^n}(g^{n+1})$ in (\ref{g_discrete}), 
our proposed method is to use the relation
$$ \mathcal Q(M, g)=\frac{1}{4}\left[ \mathcal Q(M+g, M+g)  - \mathcal Q(M-g, M-g)\right], $$
and by (\ref{LL}), namely $\mathcal L_{M}(g)=2\mathcal Q(M, g)$, then 
\begin{equation}\label{L-Mg}\mathcal L_{M^n}(g^{n+1})=\frac{1}{2}\left[ \mathcal Q(M^n + g^{n+1}, M^n + g^{n+1}) - 
\mathcal Q(M^n - g^{n+1}, M^n - g^{n+1})\right]. 
\end{equation}
To deal with the implicit collision operator $\mathcal Q$, we adopt the penalization method developed in \cite{Filbet-Jin} for the Boltzmann equation, 
and that in \cite{JinYan} for the Fokker-Planck-Landau equation. 

We briefly recall the spirit of the penalization for the collision operators used in \cite{Filbet-Jin, JinYan}. 
They introduced some dissipative penalization operator $\mathcal P$ for the Boltzmann or the FPL collision operator. 
The collision operators $\mathcal Q_B$ and $\mathcal Q_L$ in (\ref{Boltz}) or (\ref{LD}), when divided by a small Knudsen number $\varepsilon$,
become numerically stiff. 
Since explicit schemes require severe stability constraints and are computationally expensive, while 
implicit schemes, though allow larger time step, are difficult to seek numerical solution of a fully nonlinear problem at each time step, thus one desires to 
combine both advantages of implicit and explicit schemes for solving the stiff problem: large time step and low computational complexity. 

The idea is to split the RHS of (\ref{Boltz}) or (\ref{LD}) as the sum of a stiff part and a less stiff part as
$$\frac{\mathcal Q(f,f)}{\varepsilon} = \underbrace{\frac{\mathcal Q(f^n,f^n) - \mathcal P(f^n)}{\varepsilon}}_{\text{less stiff}} + \underbrace{\frac{\mathcal P(f^{n+1})}{\varepsilon}}_{\text{stiff}}, $$
where $\mathcal Q$ represents $\mathcal Q_B$ or $\mathcal Q_L$, 
$\mathcal P(f)$ is a well balanced, linear operator and is asymptotically close to the source term $\mathcal Q(f, f)$. 
We adopt a first order implicit-explicit (IMEX) scheme for the time discretization here. 

The following gives an explicit explanation on how the linearized operator is implemented by combining the formulation (\ref{L-Mg}) and the 
penalization strategies for the Boltzmann and the FPL equations. 
The advantage of using BGK for the Boltzmann and linear Fokker-Planck for the FPL equation is that these penalty operators are much easier to invert than the original kinetic operators when discretized implicitly. In particular, the implicit BGK operator can be inverted explicitly, 
while the Fokker-Planck operator can be inverted as a linear symmetric operator. See \cite{Filbet-Jin, JinYan}. 
\\[6pt]
{\bf I}. For the Boltzmann equation, the linear BGK collision operator \cite{Filbet-Jin}
\begin{equation} P(f) = P_{BGK}^{M}f = \beta(M - f) \end{equation}
is used as the penalty operator. Now we replace $\mathcal L_{M^n}(g^{n+1})$ in (\ref{U_discrete}) by $\mathcal L_{M^n}^P(g^{n+1})$, given by 
\begin{align}
&\displaystyle\mathcal L_{M^n}^P(g^{n+1}) = \frac{1}{2}\bigg[\mathcal Q_{\text{B}}(M^n+g^n, M^n+g^n) - \beta_1^n (M^n - (M^n+ g^n)) 
+ \beta_1^{n+1} (M^{n+1}-(M^{n+1} + g^{n+1})) \notag\\[2pt]
&\displaystyle \qquad\qquad\qquad - \left\{\mathcal Q_{\text{B}}(M^n-g^n, M^n-g^n) - \beta_2^n (M^n - (M^n - g^n)) + \beta_2^{n+1} (M^{n+1}- (M^{n+1} - g^{n+1})) \right\}\bigg]\notag\\[2pt]
&\displaystyle \qquad\qquad\quad = \frac{1}{2}\bigg[\mathcal Q_{\text{B}}(M^n+g^n, M^n+g^n)  + \beta_1^n g^n -  \beta_1^{n+1} g^{n+1} \notag\\[2pt]
&\displaystyle \qquad\qquad\qquad  - \mathcal Q_{\text{B}}(M^n-g^n, M^n-g^n)  + \beta_2^n g^n - \beta_2^{n+1} g^{n+1})\bigg] \notag\\[2pt]
&\displaystyle\label{LG1}
\qquad\qquad\quad =  \frac{1}{2}\left[\mathcal Q_{\text{B}}(M^n+g^n, M^n+g^n) - \mathcal Q_{\text{B}}(M^n-g^n, M^n-g^n)\right]\notag\\[2pt]
&\displaystyle\qquad\qquad\qquad + \frac{1}{2}(\beta_1^n + \beta_2^n) g^n - \frac{1}{2}(\beta_1^{n+1} + \beta_2^{n+1})g^{n+1}\,. 
\end{align}

In the Boltzmann equation, the parameter $\beta>0$ is chosen as an upper bound of $||\nabla \mathcal Q(M)||$ or some approximation of it,
for example, 
\begin{align}
\begin{split}
\label{penalty1} 
\displaystyle\beta_1^n &= \sup_{v}\left|\frac{\mathcal Q(M^n+g^n, M^n+g^n) -\mathcal Q(M^n,M^n)}{g^n}\right|
=  \sup_{v}\left|\frac{\mathcal Q(M^n+g^n, M^n+g^n)}{g^n}\right|, \\[4pt]
\displaystyle
\beta_2^n &= \sup_{v} \left|\frac{\mathcal Q(M^n-g^n, M^n-g^n) - \mathcal Q(M^n,M^n)}{g^n}\right| = 
\sup_{v} \left|\frac{\mathcal Q(M^n-g^n, M^n-g^n)}{g^n}\right|. 
\end{split}
\end{align}
\\[6pt]
{\bf II}. For the nFPL equation, the linear Fokker-Planck (FP) operator 
\begin{equation}\label{FP}P(f) = P_{FP}^{M}f = \nabla_{v}\cdot \left(M \nabla_{v}\left(\frac{f}{M}\right)\right) 
\end{equation}
is chosen as the suitable penalty operator \cite{JinYan}. We now replace $\mathcal L_{M^n}(g^{n+1})$ in (\ref{U_discrete}) by $\mathcal L_{M^n}^P(g^{n+1})$
(and use the bracket notation $( \cdot )$ to denote $P$ imposed on the argument), 
\begin{align}
&\displaystyle\mathcal L_{M^n}^P(g^{n+1}) = \frac{1}{2}\bigg[\mathcal Q_{\text{L}}(M^n+g^n, M^n+g^n) - \beta_1^n P^n (M^n+g^n) 
+ \beta_1^n P^{n+1}(M^{n+1}+g^{n+1})\notag\\[2pt]
&\displaystyle \qquad\qquad\qquad - \left\{\mathcal Q_{\text{L}}(M^n-g^n, M^n-g^n) - \beta_2^n P^n (M^n-g^n) + \beta_2^n P^{n+1} (M^{n+1}-g^{n+1}) 
\right\}\bigg] \notag\\[2pt]
&\displaystyle \qquad\qquad\quad = \frac{1}{2}\bigg[\mathcal Q_{\text{L}}(M^n+g^n, M^n+g^n) - \beta_1^n P^n (g^n) + \beta_1^n P^{n+1} (g^{n+1})\notag\\[2pt]
&\displaystyle \qquad\qquad\qquad - \mathcal Q_{\text{L}}(M^n-g^n, M^n-g^n)  - \beta_2^n P^n (g^n) + \beta_2^n P^{n+1} (g^{n+1}) \bigg] \notag\\[2pt]
&\displaystyle\label{LG2} 
\qquad\qquad\quad =  \frac{1}{2}\left[\mathcal Q_{\text{L}}(M^n+g^n, M^n+g^n) - \mathcal Q_{\text{L}}(M^n-g^n, M^n-g^n)\right] \notag\\[2pt]
&\displaystyle\qquad\qquad\qquad - \frac{1}{2}(\beta_1^n + \beta_2^n)P^n (g^n) + \frac{1}{2}(\beta_1^n + \beta_2^n)P^{n+1} (g^{n+1})\,, 
\end{align}
where the well-balanced property of $P$, i.e., $P^n (M^n) =P^{n+1} (M^{n+1}) = 0$ is used. 
 
In (\ref{LG2}), $\beta_1^n$ and $\beta_2^n$ are chosen as
\begin{align*}
&\displaystyle\beta_1^n = \beta_0 \max_{v}\lambda(D_{A}(g^n+M^n)), \\[4pt]
&\displaystyle \beta_2^n = \beta_0 \max_{v}\lambda(D_{A}(g^n-M^n)). 
\end{align*}
$\beta_0$ is a constant satisfying $\beta_0>\frac{1}{2}$, and a simple choice is $\beta_0=1$. 
$\lambda(D_{A})$ is the spectral radius of the positive symmetric matrix $D_{A}$, 
$$ D_{A}(f)= \int_{\mathbb R^d} A(v-v_{\ast})f_{\ast}\, dv_{\ast}, $$

\subsection{Space and velocity discretizations}

{\bf Space discretization}\, 
For simplicity and clarity of notations, we only consider $x\in\mathbb R$.
As done in \cite{MM-Lemou}, a finite volume discretization is used for the transport term in the left-hand-side of (\ref{g_discrete}); 
a central difference scheme is used to discretize the term $(\mathbb I -\Pi_{M^n})(v\cdot\nabla_x M^n)$ via (\ref{g_discrete}), 
and the term $\varepsilon\nabla_x \cdot\langle v m g^{n+1}\rangle$ via (\ref{U_discrete}). 

Consider spatial grid points $x_{i+\frac{1}{2}}$ and $x_i$ the center of the cell $[x_{i-\frac{1}{2}}, x_{i+\frac{1}{2}}]$, for $i=0, \cdots N_x$. 
A uniform space step is $\Delta x=x_{i+\frac{1}{2}}-x_{i-\frac{1}{2}}=x_{i}-x_{i-1}$. 
Let $U_i^n \approx U(t_n, x_i)$ and $g_{i+\frac{1}{2}}^n \approx g(t_n, x_{i+\frac{1}{2}})$. 
Now we define the following notations for the finite difference operators. 
For every grid function $\phi=(\phi_{i+\frac{1}{2}})$, define the one-sided difference operators: 
$$ D^{-}\phi_{i+\frac{1}{2}}=\frac{\phi_{i+\frac{1}{2}}-\phi_{i-\frac{1}{2}}}{\Delta x}, \qquad 
D^{+}\phi_{i+\frac{1}{2}}=\frac{\phi_{i+\frac{3}{2}}-\phi_{i+\frac{1}{2}}}{\Delta x}. $$
For every grid function $\mu=(\mu_{i})$, we define the following centered operator: 
$$ \delta^{0}\mu_{i+\frac{1}{2}}=\frac{\mu_{i+1}-\mu_{i}}{\Delta x}. $$

{\bf Velocity discretization}\, We adopt the simple trapezoidal rule to compute the numerical integral in velocity space. 
For example, we write the one-dimensional trapezoidal rule, 
$$ \int_{\mathbb R} f \, dv \approx \Delta v \left(\frac{1}{2}f(v_0) + f(v_1) + \cdots + f(v_{N_{v}-1}) + \frac{1}{2}f(v_{N_v})\right) := 
\sum_{j=0}^{N_v} f(v_j) w_j\, \Delta v, $$
where $w=(\frac{1}{2}, 1, \cdots, 1, \frac{1}{2})$. 
\\[2pt]

{\bf Macroscopic equations} \\
The fluid equation (\ref{U_discrete}) is approximated at points $x_i$. The flux $\partial_{x}F(U^n)$ at $x_i$ is discretized by 
\begin{equation} \partial_{x}F(U^n)\big|_{x_i} \approx\frac{F_{i+\frac{1}{2}}(U^n)- F_{i-\frac{1}{2}}(U^n)}{\Delta x},  \end{equation}
where upwind-based discretization is used to approximate
$F(U^n)=\langle vm M^n \rangle$ at points $x_{i+\frac{1}{2}}$. The first order approximation is given by
\begin{equation} F_{i+\frac{1}{2}}(U^n) = \langle m(v^{+}M_{i}^n + v^{-}M_{i+1}^n)\rangle. \end{equation}
A second order approximation of $\partial_x F(U)$ term will be discussed in section \ref{Sec:Num}.

The flux term $\partial_{x}\langle vm g^{n+1}\rangle$ at $x_i$ on the right-hand-side of (\ref{U_discrete}) is approximated by central differences, 
\begin{equation}
\partial_{x}\langle vmg^{n+1}\rangle \big|_{x_i} \approx\bigg\langle v m \frac{g_{i+\frac{1}{2}}^{n+1}- g_{i-\frac{1}{2}}^{n+1}}{\Delta x}\bigg\rangle. \end{equation}
The fully discretized scheme of the equation (\ref{U_discrete}) then reads
\begin{equation}\label{full_U}
\frac{U_i^{n+1}-U_i^n}{\Delta t} + \frac{F_{i+\frac{1}{2}}(U^n)-F_{i-\frac{1}{2}}(U^n)}{\Delta x} = 
- \varepsilon\sum_{j=0}^{N_v}\, v_j\, m(v_j)\, \frac{g_{i+\frac{1}{2}, j}^{n+1} - g_{i-\frac{1}{2}, j}^{n+1}}{\Delta x}\, w_j\, \Delta v, 
\end{equation}
where $g_{i+\frac{1}{2}, j}^n \approx g(t_n, x_{i+\frac{1}{2}}, v_j)$.

Next, we prove that the discrete macroscopic equations (\ref{full_U}) conserve mass, momentum and total energy.

\begin{theorem} {\bf (Conservation of moments $U$)}\\
For periodic or zero flux boundary condition, one has
  \begin{equation}\label{cons-U}
    \sum_{i=0}^{N_x}\, U_i^{n+1} = \sum_{i=0}^{N_x}\, U_i^n\,.
    \end{equation}
Namely, the total mass, momentum and energy are all numerically conserved. 
\end{theorem}

\begin{proof} 
Summing up all $i=0, \cdots, N_x$ on (\ref{full_U}),  one has
\begin{equation}\label{U_Cons} \frac{\sum_{i} U_i^{n+1} - \sum_{i} U_i^n}{\Delta t} + \sum_{i}\left(\frac{F_{i+\frac{1}{2}}(U^n) - F_{i-\frac{1}{2}}(U^n)}{\Delta x}\right)  
= -\varepsilon \sum_{i}\sum_{j=0}^{N_v}\, v_j\, m(v_j)\, \frac{g_{i+\frac{1}{2}, j}^{n+1} - g_{i-\frac{1}{2}, j}^{n+1}}{\Delta x}\, w_j\, \Delta v.  
\end{equation}
By the assumption on the boundary condition, the telescoping summation terms vanish, and then one has (\ref{cons-U}).
\end{proof}

\begin{remark} If $\varepsilon$ is spatially dependent, then (\ref{U_discrete}) is written by
$$\frac{U^{n+1}-U^n}{\Delta t}  + \nabla_x \cdot F(U^n) +  
\nabla_x \cdot \langle \varepsilon v m g^{n+1}\rangle =0, $$ 
and (\ref{U_Cons}) correspondingly becomes
\begin{align}
&\displaystyle\quad\frac{\sum_{i} U_i^{n+1} - \sum_{i} U_i^n}{\Delta t} + \sum_{i}\left(\frac{F_{i+\frac{1}{2}}(U^n) - F_{i-\frac{1}{2}}(U^n)}{\Delta x}\right) \notag
\\[4pt]
&\label{U_Cons1}\displaystyle = - \sum_{i}\sum_{j=0}^{N_v}\, v_j\, m(v_j)\, \frac{\varepsilon_{i+\frac{1}{2}}\, g_{i+\frac{1}{2}, j}^{n+1} - \varepsilon_{i-\frac{1}{2}}\, g_{i-\frac{1}{2}, j}^{n+1}}{\Delta x}\, w_j\, \Delta v, 
\end{align}
with $\varepsilon_{i+\frac{1}{2}}=\varepsilon(x_{i+\frac{1}{2}})$, $\varepsilon_{i-\frac{1}{2}}=\varepsilon(x_{i-\frac{1}{2}})$. 
This again has the conservation property (\ref{cons-U}). 
\end{remark}

\begin{remark}\label{rmk-cons}
  Typically, a discrete collision operator, particularly those based on
  spectral approximations in velocity space \cite{gamba2017fast, pareschi2000numerical,gamba2009spectral, mouhot2006fast}, does not
  conserve {\it exactly} the moments $U$ (\ref{cons-U}), which
  needs to be taken care of with extra efforts \cite{mieussens2000discrete, zhang2017conservative, gamba2014conservative}.
  What differs here is that the conserved variables $U$ are obtained from
  the macroscopic system (\ref{U}), which has  the zero right hand side,
  thus the conservation property (\ref{cons-U}) can be easily guaranteed
  by {\it any} conservative discretization of the spatial derivative in
  (\ref{U}). What differs here from those typical kinetic solvers in \cite{zhang2017conservative, gamba2014conservative} is that in the latter
  cases the moments were obtained by taking the discrete moments from $f$, computed from
  the original kinetic equation for $f$, with the collision operator
  discretized not in an exactly conserved way! This observation is not new,
  and in fact was already pointed out in \cite{JinYan}. In section \ref{sec:7}
  this point will be further explored for general kinetic systems and this offers 
  a generic recipe for obtaining (exactly) conservative schemes through solving the
  moment systems.
\end{remark}

{\bf Microscopic equation}\\
Equation (\ref{g_discrete}) is approximated at grid point $x_{i+\frac{1}{2}}$; the term $(\mathbb I - \Pi_{M^n})(v\cdot\nabla_x g^n)$ in the left-hand-side is approximated by a first order upwind scheme 
\begin{equation} \label{1st-order}
  (\mathbb I - \Pi_{M^n})(v\, \partial_{x}g^n) \big|_{x_{i+\frac{1}{2}}} 
\approx \left(\mathbb I - \Pi_{i+\frac{1}{2}}^n \right)\left(v^{+}D^{-}+v^{-}D^{+}\right)g_{i+\frac{1}{2}}^n. \end{equation}
The transport term $(\mathbb I -\Pi_{M^n})(v\cdot\nabla_x M^n)$ in the right-hand-side of (\ref{g_discrete}) is discretized by a central difference
scheme 
\begin{equation}  (\mathbb I -\Pi_{M^n})(v \, \partial_{x}M^n) \big|_{x_{i+\frac{1}{2}}} \approx \left(\mathbb I - \Pi_{i+\frac{1}{2}}^n \right)\left(v\, \delta^{0}M_{i+\frac{1}{2}}^n\right), 
\end{equation}
where $\Pi_{i+\frac{1}{2}}^n$ is an approximation of $\Pi_{M(U(t_n, x_{i+\frac{1}{2}}))}$. 
A suitable choice of $\Pi_{i+\frac{1}{2}}^n$ is given by (\cite{MM-Lemou})
\begin{equation}\label{Pie}
\Pi_{i+\frac{1}{2}}^n = \frac{\Pi_{i}^n +\Pi_{i+1}^n}{2}=\frac{\Pi(U_i^n)+\Pi(U_{i+1}^n)}{2},  \qquad \text{   or    }\, 
\Pi_{i+\frac{1}{2}}=\Pi  \left(\frac{U_{i}+U_{i+1}}{2}\right),
\end{equation}
and $M_{i+\frac{1}{2}}^n \approx\frac{M_{i}^n+M_{i+1}^n}{2}$. 
\\[2pt]

{\bf I.}\,  For the Boltzmann equation, the discretized scheme of the microscopic equations (\ref{g_discrete}) is given by
\begin{align}
&\displaystyle \quad\frac{g_{i+\frac{1}{2}}^{n+1}-g_{i+\frac{1}{2}}^{n}}{\Delta t} + 
\left(\mathbb I - \Pi_{i+\frac{1}{2}}^n \right)\left(v^{+}\, \frac{g_{i+\frac{1}{2}}^{n}-g_{i-\frac{1}{2}}^{n}}{\Delta x}+ v^{-}\, \frac{g_{i+\frac{3}{2}}^{n}-g_{i+\frac{1}{2}}^{n}}{\Delta x}\right) - \mathcal Q_{B}(g^n_{i+\frac{1}{2}}, g^n_{i+\frac{1}{2}}) \notag  \\[8pt]
&\displaystyle = \frac{1}{\varepsilon}\bigg[\frac{1}{2}\left(\mathcal Q_{B}(M_{i+\frac{1}{2}}^n+g_{i+\frac{1}{2}}^n, M_{i+\frac{1}{2}}^n+g_{i+\frac{1}{2}}^n) 
- \mathcal Q_{B}(M_{i+\frac{1}{2}}^n-g_{i+\frac{1}{2}}^n, M_{i+\frac{1}{2}}^n-g_{i+\frac{1}{2}}^n)\right) + \frac{1}{2}(\beta_1^n + \beta_2^n) g_{i+\frac{1}{2}}^n  \notag\\[8pt]
&\displaystyle\label{full_g0} \qquad  - \frac{1}{2}(\beta_1^{n+1} + \beta_2^{n+1}) g_{i+\frac{1}{2}}^{n+1}
- \left(\mathbb I - \Pi_{i+\frac{1}{2}}^n \right)\left(v\, \frac{M_{i+1}^n-M_{i}^n}{\Delta x}\right)\bigg],  
\end{align}
thus 
\begin{align}
&\displaystyle g_{i+\frac{1}{2}}^{n+1}=\frac{1}{1+\frac{\Delta t}{2\varepsilon}(\beta_1^{n+1}+\beta_2^{n+1})}\, \bigg[g_{i+\frac{1}{2}}^{n} - 
\Delta t  \left(\mathbb I - \Pi_{i+\frac{1}{2}}^n \right)\left(v^{+}\, \frac{g_{i+\frac{1}{2}}^{n}-g_{i-\frac{1}{2}}^{n}}{\Delta x}+ v^{-}\, \frac{g_{i+\frac{3}{2}}^{n}-g_{i+\frac{1}{2}}^{n}}{\Delta x}\right)  \notag\\[8pt]
&\displaystyle\qquad + \Delta t\, \mathcal Q_{B}(g^n_{i+\frac{1}{2}}, g^n_{i+\frac{1}{2}})  + \frac{\Delta t}{\varepsilon}
\bigg(\frac{1}{2}\left(\mathcal Q_{B}(M_{i+\frac{1}{2}}^n+g_{i+\frac{1}{2}}^n, M_{i+\frac{1}{2}}^n+g_{i+\frac{1}{2}}^n)
- \mathcal Q_{B}(M_{i+\frac{1}{2}}^n-g^n_{i+\frac{1}{2}}, M_{i+\frac{1}{2}}^n-g_{i+\frac{1}{2}}^n)\right)\notag\\[8pt]
&\displaystyle \label{full_g}\qquad\qquad\qquad\qquad\qquad\qquad\quad + \frac{1}{2}(\beta_1^n + \beta_2^n) g_{i+\frac{1}{2}}^n
- \left(\mathbb I - \Pi_{i+\frac{1}{2}}^n \right)\left(v\, \frac{M_{i+1}^n-M_{i}^n}{\Delta x}\right)\bigg)\bigg]\,.  
\end{align}
\begin{remark}
To improve to second-order spatial discretization of $v\, \partial_x g$ in the above equation, one uses a second order upwind (MUSCL) discretization in
(\ref{1st-order}), and then (\ref{full_g}) is replaced by
\begin{align}
&\displaystyle g_{i+\frac{1}{2}}^{n+1}=\frac{1}{1+\frac{\Delta t}{2\varepsilon}(\beta_1^{n+1}+\beta_2^{n+1})}\, \bigg[g_{i+\frac{1}{2}}^{n} - 
\Delta t  \left(\mathbb I - \Pi_{i+\frac{1}{2}}^n\right)\left(\frac{G_{i+1}^n - G_i^n}{\Delta x}\right)  \notag\\[8pt]
&\displaystyle\qquad + \Delta t\, \mathcal Q_{B}(g^n_{i+\frac{1}{2}}, g^n_{i+\frac{1}{2}})  + \frac{\Delta t}{\varepsilon}
\bigg(\frac{1}{2}\left(\mathcal Q_{B}(M_{i+\frac{1}{2}}^n+g_{i+\frac{1}{2}}^n, M_{i+\frac{1}{2}}^n+g_{i+\frac{1}{2}}^n)
- \mathcal Q_{B}(M_{i+\frac{1}{2}}^n-g^n_{i+\frac{1}{2}}, M_{i+\frac{1}{2}}^n-g_{i+\frac{1}{2}}^n)\right)\notag\\[8pt]
&\displaystyle \label{full_g2}\qquad\qquad\qquad\qquad\qquad\qquad\quad + \frac{1}{2}(\beta_1^n + \beta_2^n) g_{i+\frac{1}{2}}^n
- \left(\mathbb I - \Pi_{i+\frac{1}{2}}^n \right)\left(v\, \frac{M_{i+1}^n-M_{i}^n}{\Delta x}\right)\bigg)\bigg]\,, 
\end{align}
where
\begin{equation} \label{419}
   G_i^n = v^{+} g_i^{+, n} + v^{-} g_i^{-, n} = v^{+} \left(g_{i-\frac{1}{2}}^n + \frac{\Delta x}{2}\, \delta g_{i-\frac{1}{2}}^n\right) 
   + v^{-}\left(g_{i+\frac{1}{2}}^n - \frac{\Delta x}{2}\, \delta g_{i+\frac{1}{2}}^n\right), \qquad i = 0, \cdots, N,
\end{equation}
and $\delta g$ represents a slope with a slope limiter, given by \cite{LeVeque}
$$\delta g_{j-\frac{1}{2}}^n = \frac{1}{\Delta x}\, \text{minmod}\left\{g_{j+\frac{1}{2}}^n - g_{j-\frac{1}{2}}^n, \, g_{j-\frac{1}{2}}^n - g_{j-\frac{3}{2}}^n\right\}, \qquad j = 0, \cdots, N+1.
$$
\end{remark}

{\bf II. }\,  For the nFPL equation, we first introduce the symmetrized operator in \cite{JinYan}
$$ \widetilde P h = \frac{1}{\sqrt{M}}\nabla_v \cdot\left(M\nabla_v \left(\frac{h}{\sqrt{M}}\right)\right). $$
Thus the penalty operator given in (\ref{FP}) can be rewritten as
$$ P_{FP}^{M} f = \sqrt{M}\widetilde P \frac{f}{\sqrt{M}}. $$
Use (\ref{LG2}), (\ref{full_g0}) correspondingly becomes 
\begin{align}
&\displaystyle g_{i+\frac{1}{2}}^{n+1}= \left(\mathbb I  - 
\frac{\Delta t}{2\varepsilon}(\beta_1^{n} + \beta_2^{n}) P^{n+1}\right)^{-1}
\bigg[g_{i+\frac{1}{2}}^{n} - \Delta t \left(\mathbb I - \Pi_{i+\frac{1}{2}}^n \right)\left(v^{+}\, \frac{g_{i+\frac{1}{2}}^{n}-g_{i-\frac{1}{2}}^{n}}{\Delta x}+ v^{-}\, \frac{g_{i+\frac{3}{2}}^{n}-g_{i+\frac{1}{2}}^{n}}{\Delta x}\right)  \notag\\[8pt]
&\displaystyle \qquad  + \Delta t\, \mathcal Q_{L}(g^n_{i+\frac{1}{2}}, g^n_{i+\frac{1}{2}})  + \frac{\Delta t}{\varepsilon}
\bigg(\frac{1}{2}\left(\mathcal Q_{L}(M_{i+\frac{1}{2}}^n+g_{i+\frac{1}{2}}^n, M_{i+\frac{1}{2}}^n+g_{i+\frac{1}{2}}^n) - \mathcal Q_{L}(M_{i+\frac{1}{2}}^n-g_{i+\frac{1}{2}}^n, M_{i+\frac{1}{2}}^n-g_{i+\frac{1}{2}}^n)\right) \notag\\[8pt]
&\displaystyle \label{full_gL}\qquad\qquad\qquad\qquad\qquad\qquad\quad -\frac{1}{2}(\beta_1^n + \beta_2^n)P^n g_{i+\frac{1}{2}}^n 
- \left(\mathbb I - \Pi_{i+\frac{1}{2}}^n \right)\left(v\, \frac{M_{i+1}^n-M_{i}^n}{\Delta x}\right)\bigg)\bigg]\,.  
\end{align}
Rewrite the above equation (\ref{full_gL}) as
\begin{align}
&\displaystyle \left(\frac{g_{i+\frac{1}{2}}}{\sqrt{M}}\right)^{n+1} = 
\left(\mathbb I  - \frac{\Delta t}{2\varepsilon}(\beta_1^{n} + \beta_2^{n}) \widetilde P^{n+1}\right)^{-1}\bigg\{\frac{1}{\sqrt{M^{n+1}}}
\bigg[g_{i+\frac{1}{2}}^{n} - \Delta t  \left(\mathbb I - \Pi_{i+\frac{1}{2}}^n \right) \notag\\[8pt]
&\displaystyle \qquad\qquad\qquad \cdot\left(v^{+}\, \frac{g_{i+\frac{1}{2}}^{n}-g_{i-\frac{1}{2}}^{n}}{\Delta x}+ v^{-}\, \frac{g_{i+\frac{3}{2}}^{n}-g_{i+\frac{1}{2}}^{n}}{\Delta x}\right) + \Delta t\, \mathcal Q_{L}(g^n_{i+\frac{1}{2}}, g^n_{i+\frac{1}{2}})  \notag\\[8pt]
&\displaystyle\qquad\qquad\qquad + \frac{\Delta t}{\varepsilon}
\bigg(\frac{1}{2}\left(\mathcal Q_{L}(M_{i+\frac{1}{2}}^n+g_{i+\frac{1}{2}}^n, M_{i+\frac{1}{2}}^n+g_{i+\frac{1}{2}}^n) - \mathcal Q_{L}(M_{i+\frac{1}{2}}^n-g_{i+\frac{1}{2}}^n, M_{i+\frac{1}{2}}^n-g_{i+\frac{1}{2}}^n)\right)  \notag\\[8pt] 
&\displaystyle\qquad\qquad\qquad\qquad\label{full_g1} -\frac{1}{2}(\beta_1^n + \beta_2^n)\sqrt{M^n}
\widetilde P\, \frac{g_{i+\frac{1}{2}}^n}{\sqrt{M^n}}
- \left(\mathbb I - \Pi_{i+\frac{1}{2}}^n \right)\left(v\, \frac{M_{i+1}^n-M_{i}^n}{\Delta x}\right)\bigg)\bigg]\bigg\}\,. 
\end{align}
One can apply the Conjugate Gradient (CG) method to get $\left(\frac{g_{i+\frac{1}{2}}}{\sqrt{M}}\right)^{n+1}$, which is used in \cite{JinYan}. 
\\[2pt]

A second order discretization of $(\mathbb I - \Pi_{M^n})(v\, \partial_{x}g^n)$ can also be used as in (\ref{419}).

{\bf Velocity discretization of $\widetilde P$ }\, 
As was done in \cite{JinYan}, the discretization of $\tilde P$ in one dimension is given by
\begin{align}
&\displaystyle (\tilde P h)_{j}= \frac{1}{(\Delta v)^2}\frac{1}{\sqrt{M_j}}\bigg\{\sqrt{M_j M_{j+1}} \left(\left(\frac{h}{\sqrt{M}}\right)_{j+1}
- \left(\frac{h}{\sqrt{M}}\right)_{j}\right) - \sqrt{M_j M_{j-1}}\left(\left(\frac{h}{\sqrt{M}}\right)_{j} - \left(\frac{h}{\sqrt{M}}\right)_{j-1}\right)\bigg\} \notag\\[6pt]
&\displaystyle\qquad\quad = \frac{1}{(\Delta v)^2}\left(h_{j+1} -\frac{\sqrt{M_{j+1}} + \sqrt{M_{j-1}}}{\sqrt{M_j}}h_j + h_{j-1}\right). 
\end{align}
It is obvious that $\tilde P$ is symmetric. We discretize dimension-by-dimension in velocity space. 

\subsection{The Asymptotic-Preserving property of the scheme}

In this section, we investigate the formal fluid dynamics behavior
(for $\varepsilon \ll 1$) of the discretized numerical scheme given by (\ref{full_g}) and (\ref{full_U}) for the Boltzmann equation, in order to show that
the scheme is Asymptotic-Preserving (AP)\cite{jin1999efficient, jin2010asymptotic} in the fluid dynamic regime. For notation simplicity, rewrite the term 
$$ \mathcal L_{M_{i+\frac{1}{2}}^n}(g_{i+\frac{1}{2}}^n) = \frac{1}{2}\left(\mathcal Q(M_{i+\frac{1}{2}}^n+g_{i+\frac{1}{2}}^n, M_{i+\frac{1}{2}}^n+g_{i+\frac{1}{2}}^n) - \mathcal Q(M_{i+\frac{1}{2}}^n-g_{i+\frac{1}{2}}^n, M_{i+\frac{1}{2}}^n-g_{i+\frac{1}{2}}^n)\right). $$ 

From the right hand side of (\ref{full_g0}), one can see 
\begin{align}
&\displaystyle \mathcal L_{M_{i+\frac{1}{2}}^n}(g_{i+\frac{1}{2}}^n) + \frac{1}{2}(\beta_1^n + \beta_2^n)g_{i+\frac{1}{2}}^n 
 - \frac{1}{2}(\beta_1^{n+1}+\beta_2^{n+1})g_{i+\frac{1}{2}}^{n+1}  \notag \\[6pt]
&\displaystyle \label{AP0}  \qquad\qquad\qquad  - \left(\mathbb I - \Pi_{i+\frac{1}{2}}^n \right)\left(v\, \frac{M_{i+1}^n-M_{i}^n}{\Delta x}\right) = O(\varepsilon). 
\end{align}

We make the following assumptions similar to that in \cite{Filbet-Jin}: there exists a constant $C>0$ such that
\begin{equation}\label{assump1}
 |g^n| + \left|\frac{g^{n+1}-g^n}{\Delta t}\right|\leq C, 
\end{equation}
and 
\begin{equation}\label{assump2} 
|U^n| + \left|\frac{U^{n+1}-U^n}{\Delta t}\right| \leq C. 
\end{equation}
These are typical assumptions for AP schemes, since the fluid dynamic limit of the Boltzmann or FPL equation is not rigorously justified even in the continuous case when solutions admit singularities such as shocks. \\
Denote $\beta=\frac{1}{2}(\beta_1 + \beta_2)$, we have in (\ref{AP0}) 
\begin{align*}
&\displaystyle \text{term }I: =\frac{1}{2}(\beta_1^n + \beta_2^n)g_{i+\frac{1}{2}}^n 
- \frac{1}{2}(\beta_1^{n+1}+\beta_2^{n+1})g_{i+\frac{1}{2}}^{n+1}  = 
 \beta^n g_{i+\frac{1}{2}}^n - \beta^{n+1}g_{i+\frac{1}{2}}^{n+1}  \notag\\[4pt]
&\displaystyle\qquad\quad = \beta^{n+1}(g_{i+\frac{1}{2}}^n - g_{i+\frac{1}{2}}^{n+1}) + 
(\beta^n - \beta^{n+1})g_{i+\frac{1}{2}}^n. 
\end{align*}
Under the assumption (\ref{assump1}) and (\ref{assump2}), and since $\beta^n$ only depends on $U^n$, one gets
$$ \left| \text{term }I\right| = O(\Delta t). $$
From (\ref{AP0}), $g_{i+\frac{1}{2}}^n$ is approximated by
\begin{equation}\label{g_AP_temp} g_{i+\frac{1}{2}}^n = \mathcal L_{M_{i+\frac{1}{2}}^n}^{-1}\bigg\{\left(\mathbb I - \Pi_{i+\frac{1}{2}}^n \right)\left(v\, \frac{M_{i+1}^n-M_{i}^n}{\Delta x}\right)\bigg\}
+ O(\varepsilon) + O(\Delta t). \end{equation}
$g_{i+\frac{1}{2}}^{n+1}$ can be approximated by $g_{i+\frac{1}{2}}^n + O(\Delta t)$, 
thus 
\begin{equation}\label{g_AP} g_{i+\frac{1}{2}}^{n+1} = \mathcal L_{M_{i+\frac{1}{2}}^n}^{-1}\bigg\{\left(\mathbb I - \Pi_{i+\frac{1}{2}}^n \right)\left(v\, \frac{M_{i+1}^n-M_{i}^n}{\Delta x}\right)\bigg\}
+ O(\varepsilon) + O(\Delta t).
\end{equation}
Plug (\ref{g_AP}) into (\ref{full_U}), 
\begin{align}
&\displaystyle  \frac{U_i^{n+1}-U_i^n}{\Delta t} + \frac{F_{i+\frac{1}{2}}(U^n)-F_{i-\frac{1}{2}}(U^n)}{\Delta x}
= \frac{\varepsilon}{\Delta x}\bigg\langle  v m \bigg\{ \mathcal L_{M_{i+\frac{1}{2}}^n}^{-1}\left[\left(\mathbb I - \Pi_{i+\frac{1}{2}}^n\right)\left(v\, \frac{M_{i+1}^n - M_i^n}{\Delta x}\right)\right]  \notag\\[6pt]
&\displaystyle \label{U_AP} - \mathcal L_{M_{i -\frac{1}{2}}^n}^{-1}\left[\left(\mathbb I - \Pi_{i -\frac{1}{2}}^n\right)\left(v\, \frac{M_i^n - M_{i-1}^n}{\Delta x}\right)\right]\bigg\} \bigg\rangle + O(\varepsilon\Delta t + \varepsilon^2).
\end{align}
Following the same calculation as \cite{MM-Lemou, Filbet-Jin}, one obtains 
$$ (\mathbb I - \Pi_{M})(v\cdot\nabla_x M) = 
\left(B: \left(\nabla_x u + (\nabla_x u)^{T}-\frac{2}{d}(\nabla_x \cdot u)\mathbb I\right) + A \cdot \frac{\nabla_x T}{\sqrt{T}}\right) M + O(\varepsilon), $$ 
where 
$$ A = \left(\frac{|v-u|^2}{2T}-\frac{d+2}{2}\right)\frac{v-u}{\sqrt{T}}, \qquad
B= \frac{1}{2}\left(\frac{(v-u)\otimes(v-u)}{2T}-\frac{|v-u|^2}{dT}\mathbb I\right). $$

Therefore, 
$$ \mathcal L_{M^n}^{-1}\bigg((\mathbb I - \Pi_{M^n})(v\cdot\nabla_x M^n)\bigg)
= \mathcal L_{M^n}^{-1}(BM): \left(\nabla_x u + (\nabla_x u)^{T}-\frac{2}{d}(\nabla_x\cdot u)\mathbb I\right) 
+ \mathcal L_{M^n}^{-1}(AM)\cdot\frac{\nabla_x T}{\sqrt{T}}. $$
Thus (\ref{U_AP}) is a consistent time discretization scheme to the compressible Navier--Stokes 
system, with the order of $\varepsilon$ term given by 
$$\varepsilon \nabla_x\cdot \begin{pmatrix} 0 \\ \mu \sigma(u) \\ \mu \sigma(u)u + \kappa\nabla_x T \end{pmatrix}, $$
with 
$$ \sigma(u)=\nabla_x u + (\nabla_x u)^{T} - \frac{2}{d}\nabla_x \cdot u\mathbb I, $$
where the viscosity $\mu=\mu(T)$ and the thermal conductivity 
$\kappa=\kappa(T)$ only depend on the temperature and whose general expressions can be found in 
\cite{Bardos}. 

We summarize the conclusions in the following theorem. Compared to Proposition 4.3 in \cite{MM-Lemou},  the result here is valid for 
the full Boltzmann instead of the BGK equation. 
\begin{theorem}

Consider the time and space discretizations of the Boltzmann equation, given by equation (\ref{full_g}) and (\ref{full_U}), then 

(i) In the limit $\varepsilon\to 0$, the moments $U^n$ satisfy the following discretization of the Euler equations
$$ \frac{U_i^{n+1}-U_i^n}{\Delta t} + \frac{F_{i+\frac{1}{2}}(U^n)-F_{i-\frac{1}{2}}(U^n)}{\Delta x} = 0. $$

(ii) The scheme (\ref{full_g}) and (\ref{full_U}) is asymptotically equivalent, with an error of $O(\varepsilon^2)$, to the following scheme, 
 \begin{align*}
&\displaystyle  \frac{U_i^{n+1}-U_i^n}{\Delta t} + \frac{F_{i+\frac{1}{2}}(U^n)-F_{i-\frac{1}{2}}(U^n)}{\Delta x}
= \frac{\varepsilon}{\Delta x}\bigg\langle  v m \bigg\{ \mathcal L_{M_{i+\frac{1}{2}}^n}^{-1}\left[\left(\mathbb I - \Pi_{i+\frac{1}{2}}^n\right)\left(v\, \frac{M_{i+1}^n - M_i^n}{\Delta x}\right)\right]  \notag\\[6pt]
&\displaystyle \label{U_AP} \qquad\qquad\qquad\qquad\qquad\qquad\qquad\qquad\qquad\qquad - \mathcal L_{M_{i -\frac{1}{2}}^n}^{-1}\left[\left(\mathbb I - \Pi_{i -\frac{1}{2}}^n\right)\left(v\, \frac{M_i^n - M_{i-1}^n}{\Delta x}\right)\right]\bigg\} \bigg\rangle, 
 \end{align*}
which is a consistent approximation of the compressible Navier-Stokes equation, provided that the viscous terms 
are resolved numerically.  
\end{theorem}

From (ii), it shows that one needs the mesh size and time step to be $O(\varepsilon)$ in order to capture the
Navier-Stokes approximation. This is necessary for {\it any} scheme since the viscosity and heat conductivity are of $O(\varepsilon)$. 

\section{Numerical Implementation}
\label{Sec:Num}

We mention some details in the numerical implementation. 
Assume we have all the values of $U$ and $g$ at time $t^n$, namely
$g_{-\frac{1}{2}}^n, \cdots, g_{N+\frac{1}{2}}^n$, and $U_{-1}^n, \, U_{0}^n, \cdots, U_{N+1}^n, \, U_{N+2}^n$. 
\\[2pt]

(i) {\it Step 1}. $g$ is calculated at staggered grids  $x_{\frac{1}{2}}, \cdots, x_{N-\frac{1}{2}}$. 

We use equation (\ref{full_g}) for the Boltzmann or (\ref{full_gL}) for the Landau equation (with a rewritten form of (\ref{full_g1})). 
The projection operator is given in (\ref{Pie}). Here
the second choice is used. 
Denote $$M^{\ast} = \frac{M(U_i) + M(U_{i+1})}{2}, $$ by definition of $\Pi$, one has
$$ \Pi_{M^{\ast}}(\psi) = \frac{1}{\rho}\left[ \langle\psi\rangle + \frac{(v-u)\cdot \langle (v-u)\psi \rangle}{T} + 
\left(\frac{|v-u|^2}{2T} - \frac{d}{2}\right)\frac{2}{d}\bigg\langle\left(\frac{|v-u|^2}{2T} - \frac{d}{2}\right)\psi \bigg\rangle\right] M^{\ast},   $$
where $\rho$, $u$, $T$ are associated with $M^{\ast}$ as in (\ref{U_eqn}). 
If one assumes the periodic in $x$ boundary condition, then
\begin{equation}\label{g_BC} g_{-\frac{1}{2}} = g_{N-\frac{1}{2}}, \qquad 
g_{N+\frac{1}{2}} = g_{\frac{1}{2}}\,. \end{equation}
Free-flow boundary condition is used in the shock-tube tests, that is, 
\begin{equation}\label{g_BC1}  g_{-\frac{1}{2}} = g_{\frac{1}{2}},  \qquad 
g_{N+\frac{1}{2}} = g_{N-\frac{1}{2}}, \end{equation}
and similarly for $U$. 

(ii) {\it Step 2}. $U$ is calculated at $i=1, \cdots, N$, by using (\ref{full_U}), where values of $g_{\frac{1}{2}}^{n+1}, \cdots, g_{N+\frac{1}{2}}^{n+1}$ are used. 
The numerical flux $F$ is calculated by first or second order splitting with slope limiters. We apply a second-order TVD method. 
Following \cite{MM-Lemou}, we use a simple reconstruction of the upwind flux
$F_{i+\frac{1}{2}}(U^n)$ ($i=0, \cdots, N$) from the flux splitting that is naturally derived from its kinetic formulation:
\begin{equation}\label{flux_split} F(U) = \langle v^{+}m M(U)\rangle + \langle v^{-}m M(U)\rangle := F^{+}(U) + F^{-}(U). \end{equation}
A second order approximation of the positive and negative flux is obtained by a linear piecewise polynomial $\hat F_i$ for $i=0, \cdots, N+1$. 
Then we reconstruct the numerical flux $F_{i+\frac{1}{2}}(U)$ ($i=0, \cdots, N$) in a split form, 
\begin{equation}\label{FS_2} F_{i+\frac{1}{2}}^n =
F^{+}(U_i^n) + s_i^{+, n}\, \frac{\Delta x}{2} + F^{-}(U_{i+1}^n) - s_{i+1}^{-, n}\, \frac{\Delta x}{2}\,, \end{equation}
where a slope limiter $s_i^{\pm, n}$ is introduced to suppress possible spurious oscillations near discontinuities. 
We use a second order TVD minmod slope limiter \cite{LeVeque},
\begin{equation}\label{slope} s_i^{\pm, n} =\frac{1}{\Delta x}\, \text{minmod}\left\{F^{\pm}(U_{i+1}^n)-F^{\pm}(U_i^n), \, 
F^{\pm}(U_i^n)-F^{\pm}(U_{i-1}^n)\right\}. 
\end{equation}
Note that we need $F(U_{-1}), \, F(U_0), \, F(U_{N+1}),\,  F(U_{N+2})$ when computing
$s_0^{+}, \, s_{N+1}^{-}$, thus two ghost cells are needed. 
For periodic BC, we let
$$ U_0 = U_N, \qquad U_{-1}=U_{N-1}, \qquad U_{N+1}=U_1, \qquad U_{N+2}=U_2. $$
Implementation details of solving (\ref{U_discrete}) are shown in the Appendix. 

\section{Numerical Examples}
\label{sec:NE}

{\bf Test I: The micro-macro scheme for the Boltzmann equation}

Consider the spatial variable $x\in[0,1]$. Periodic boundary condition is used except for the shock tube tests. 
The velocity variable $v\in[-L_v,L_v]^2$ with $L_v=8.4$. $N_x=100$, $\Delta t = \Delta x/20$. 
Note that the velocity domain should be chosen large enough so that the numerical solution $f$ is essentially zero at its boundary. 
The fast spectral method in \cite{pareschi2000numerical} is applied to evaluate the collision operator $\mathcal Q$ and $32$ points are used in each velocity dimension. 

In order to compare different schemes, we denote by `FJ'  the Filbet-Jin AP method with penalty proposed in \cite{Filbet-Jin} for the Boltzmann 
equation; by `JY'  the Jin-Yan AP method with penalty in \cite{JinYan} for the Landau equation. 
`MM' stands for the micro-macro scheme for the full Boltzmann and Landau equations we propose in the current paper. 
`DS' represents a direct, explicit 4th order Runge-Kutta time discretization solver for the Boltzmann or Landau equations. 
\\[2pt]

{\bf Test I (a) } \\
The initial data is given by
\begin{equation}\label{Ia_IC} \rho^{0}(x)=\frac{2+\sin(2\pi x)}{3}, \qquad u^{0} = (0.2, 0), \qquad
T^{0}(x)=\frac{3+\cos(2\pi x)}{4}\,. \end{equation}
The following non-equilibrium double-peak initial distribution is considered, 
\begin{equation}\label{dp} f^{0}(x,v)=\frac{\rho^{0}}{4\pi T^{0}}\left(e^{-\frac{|v-u^{0}|^2}{2T^{0}}} 
+ e^{-\frac{|v+u^{0}|^2}{2T^{0}}}\right). \end{equation}

{\bf Test I (b) } \\
In this example, we consider a mixed regime with the Knudsen number 
$\varepsilon$ varying in space, where $x\in[0,1]$, $v\in[-6,6]^2$, 
\begin{align}
&\displaystyle
\varepsilon(x) =
\begin{cases}   
10^{-2} + \frac{1}{2}\left(\tanh(25-20x)+\tanh(-5+20x)\right), \qquad x\leq 0.65, \\[4pt]
10^{-2}, \qquad x>0.65. 
\end{cases}
\end{align}
The initial data is given by (\ref{Ia_IC})--(\ref{dp}). 
\\[2pt]

{\bf Test I (c).}
We study a Sod shock tube test problem for the Boltzmann equation. The equilibrium initial distribution is given by
$$ f^{0}(x,v) = \frac{\rho^{0}}{2\pi T^{0}} e^{-\frac{|v-u^{0}|^2}{2 T^{0}}},$$
where the initial data for $\rho^{0}$, $u^{0}$ and $T^{0}$ are given by
\begin{align}
\begin{cases}
&\displaystyle\rho_{l}=1, \qquad  u_{l}=(0,0), \qquad T_{l}=1, \qquad x\leq 0.5, \\[2pt]
&\displaystyle\rho_{r}=0.125, \qquad u_{r}=(0,0), \qquad T_{r}=0.25,  \qquad x>0.5. 
\end{cases}
\end{align}
\\[2pt]
There are different choices of the free parameter $\beta$ in the penalty operators. We list below: 

{\it Choice 1. }
In the BGK penalty operator, $P = \beta(M-f)$, where $\beta$ is a positive constant chosen for stability. 
One can split the collision operator $\mathcal Q$ into the gaining part and the losing part 
$\mathcal Q(f,f) = \mathcal Q^{+} - f \mathcal Q^{-}$. 
In order to obtain positivity, it is sufficient to require $\beta > \mathcal Q^{-}$ (\cite{QinJin}).  
In our case, 
$$ \beta_1^n > \mathcal Q_1^{-}, \qquad \beta_2^n > \mathcal Q_2^{-}, $$
where $\mathcal Q_1$, $\mathcal Q_2$ represent the collision operator 
$\mathcal Q(M^n+g^n,M^n+g^n)$ and $\mathcal Q(M^n-g^n,M^n-g^n)$ respectively. 
Here $\beta_1^n$, $\beta_2^n$ are space and time dependent. 

{\it Choice 2. } Another choice is given in \cite{Filbet-Jin}, recall (\ref{penalty1}), we let 
\begin{align}
&\displaystyle \beta_1^n =\sup_{v}\left|\frac{\mathcal Q(M^n+g^n, M^n+g^n)}{g^n}\right|,  \\[4pt]
&\displaystyle \beta_2^n =\sup_{v} \left|\frac{\mathcal Q(M^n-g^n, M^n-g^n)}{g^n}\right|. 
\end{align}
\\[6pt]
\indent Now we present and compare numerical results using difference schemes. Property of conservation of moments will also be verified. 
In the figure titles, $P_0$, $P_1$, $P_2$ represent the mass, momentum (in $v_1$ direction) and the total energy respectively. 
For Test I (a), Figure \ref{TestI-1} (for $\varepsilon=1$)
and Figure \ref{TestI-2} (for $\varepsilon=10^{-4}$) show the time evolution of mass, momentum and energy obtained from $f$ (using $f=\mathcal M + \varepsilon g$),
denoted by `Mf' (see Remark \ref{rmk-cons}), and from solving the macroscopic equations, denoted by `ME' (moment equations) below. 
Figure \ref{TestI-1} uses `DS' and Figure~\ref{TestI-2} uses `MM' for small $\varepsilon$. 
One can observe that the moments calculated from `ME' are perfectly conserved with values unchanged as time propagates, while the conservation is not guaranteed if the moments are obtained from $f$ itself, however Figure~\ref{VP_Fig} later shows that the error in total energy conservation is bounded for long time. This phenomenon verifies the proof that moments solved from `ME' are conserved as shown 
in (\ref{U_Cons1}). Moments computed from $f$, although not exactly conserved, however owes an spectral accuracy due to the numerical error of the spectral method used for the collision operators. 
One observes that if {\it not} using the moments systems to obtain the conserved quantities, 
third moments usually own a larger error than lower (first and second) moments by using the same discretization,
a phenomenon that is also observed in several other tests in the following sections. The reason might be due to that the error in $f$ is enlarged more when multiplying by $|v|^2$ (instead of $1$ or $v$) in the integration to 
get third moments. 

In Section \ref{sec:7} we will use this idea to obtain the conservative solvers
for more general kinetic equations and for general numerical schemes, not just
the micro-macro decomposition based \cite{MM-Lemou} or penalty based 
\cite{Filbet-Jin, JinYan} approaches.

The density $\rho$, bulk velocity $u_1$ and temperature $T$ defined as the following: 
$$ \rho = \int_{\mathbb R^2} f dv, \, u_i = \frac{1}{\rho}\int_{\mathbb R^2} v_i f dv  (i = 1, 2), \, 
T = \frac{1}{2\rho}\int_{\mathbb R^2} (v-u)^2 f dv. $$
The numerical solutions are shown in Figure \ref{TestI-a-sol} for Test I (a). Here $u_2=0$ and we omit plotting it. MM uses the penalty parameters in {\it Choice 1}. One can observe that the two different approaches `FJ' and `MM' are consistent and produce the same results. 

For Test I (b), the function $\varepsilon$ is plotted in Figure 
\ref{fig_eps1}, whose values range from $10^{-2}$ to 1, and is discontinuous at $x=0.65$.
Figure \ref{TestI-b2} shows that by comparing with the `DS' solutions as a reference, 
`MM' is able to capture the macroscopic behavior efficiently with coarse mesh size and time steps when 
$\varepsilon$ is discontinuous, by using the penalty parameter $\beta$ in {\it Choice 1}. 

One can observe from Figure \ref{TestI-c}, for Test I (c), that the macroscopic quantities are well approximated although the mesh size and time steps are larger than 
$\varepsilon$, by using both the `FJ' and `MM' schemes, which give similar numerical results for the Sod problem.

\begin{figure}[H]
\begin{subfigure}{1\textwidth}
\includegraphics[width=1\textwidth, height=0.59\textwidth]{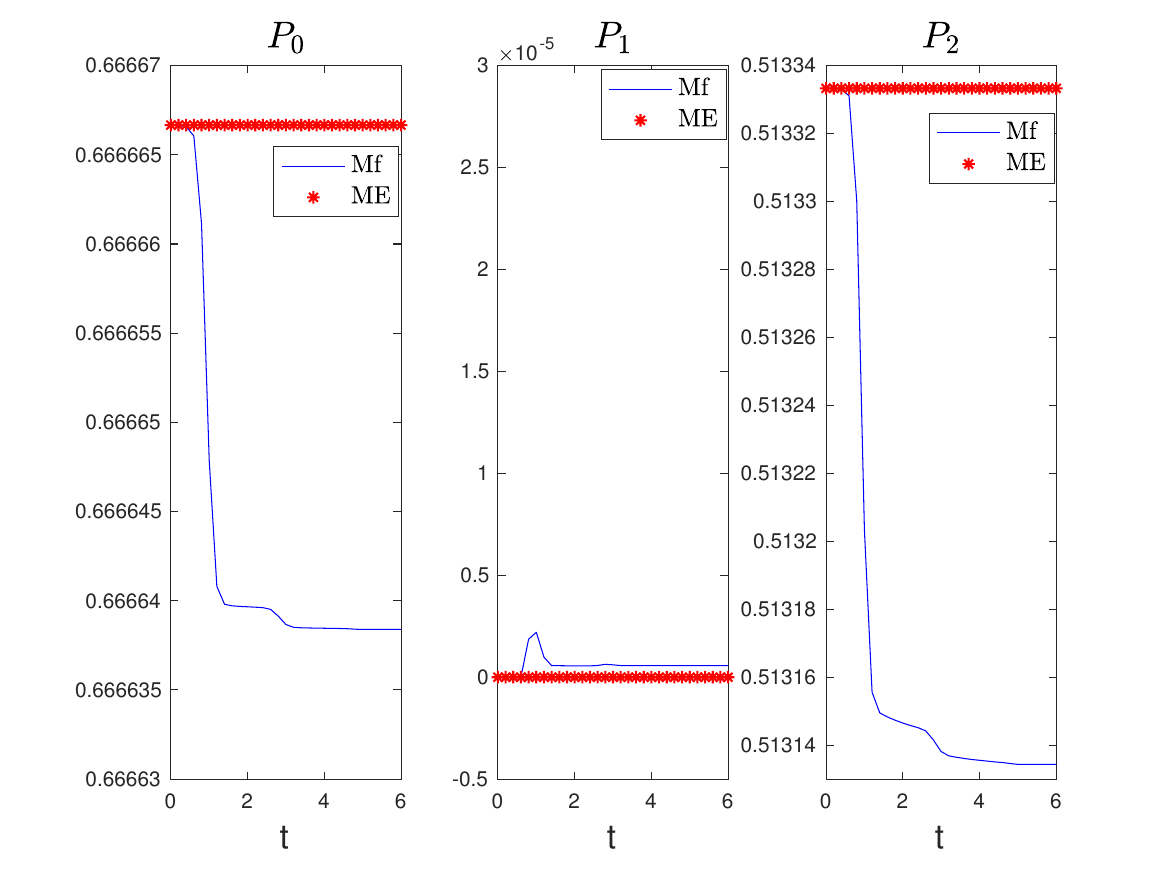}
 \end{subfigure}
  \caption{Test I (a). Time evolution of mass, momentum and energy by DS. 
 `Mf' versus `ME'. $\varepsilon=1$. }
\label{TestI-1}
\end{figure}

\begin{figure}[H]
\centering
\begin{subfigure}{1\textwidth}
\includegraphics[width=1\textwidth, height=0.59\textwidth]{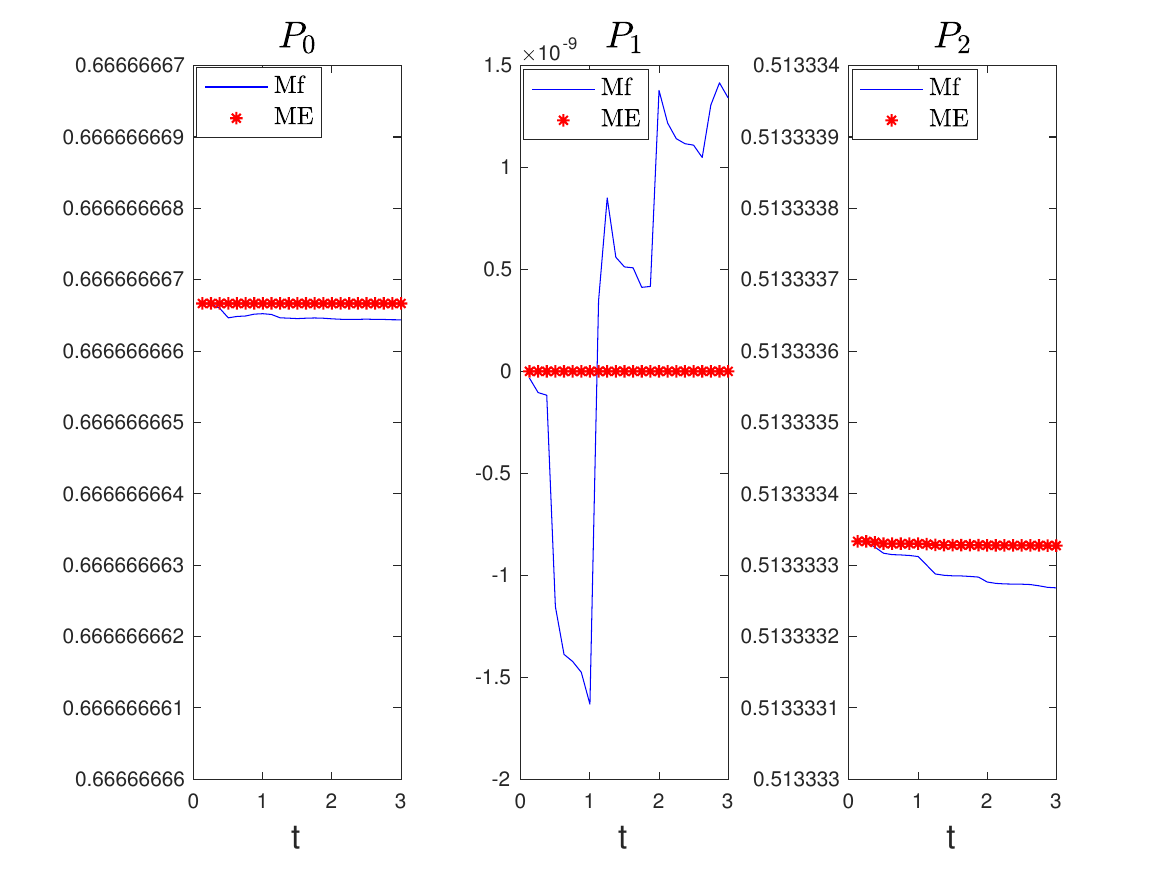}
\end{subfigure}
 \caption{Test I (a). Time evolution of mass, momentum and energy by MM:
`Mf' versus `ME'. $\varepsilon=10^{-4}$. }
\label{TestI-2}
\end{figure}

\begin{figure}[H]
\centering
\includegraphics[width=0.45\linewidth]{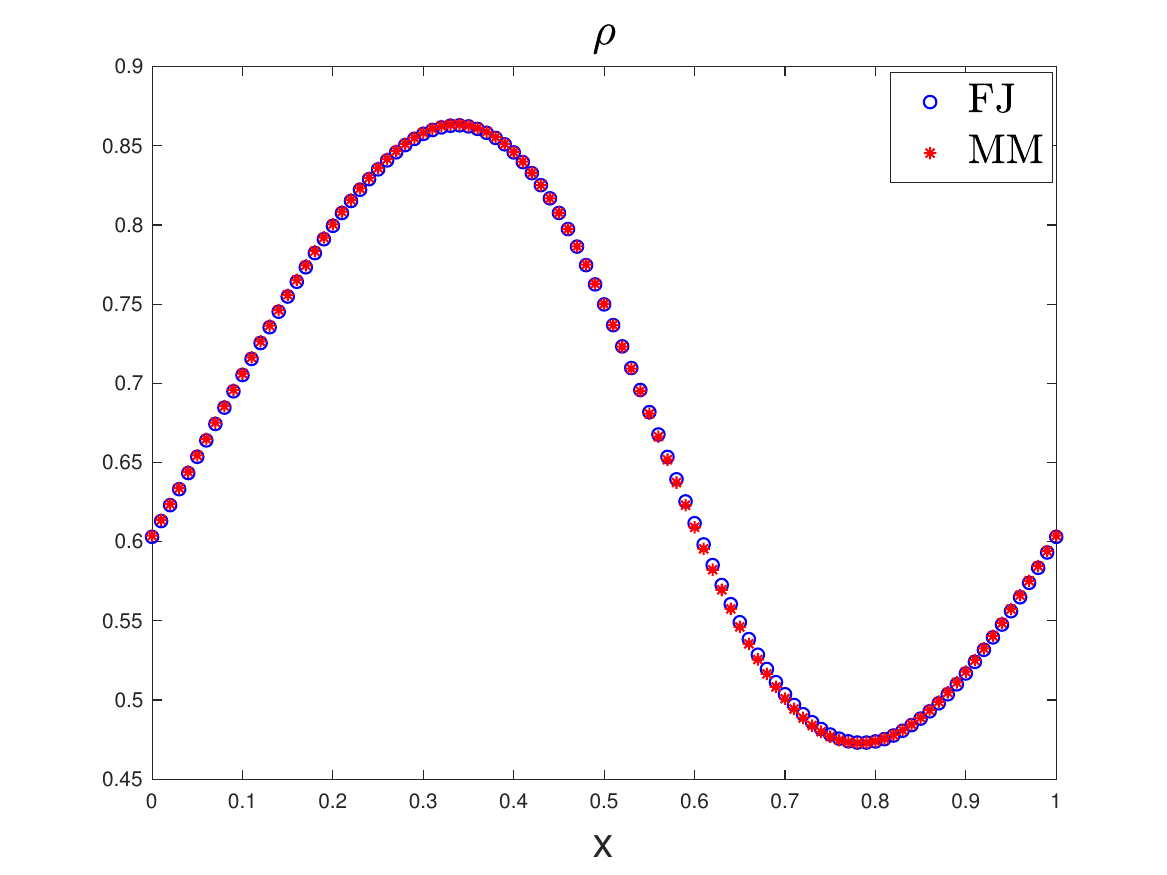}
\includegraphics[width=0.45\linewidth]{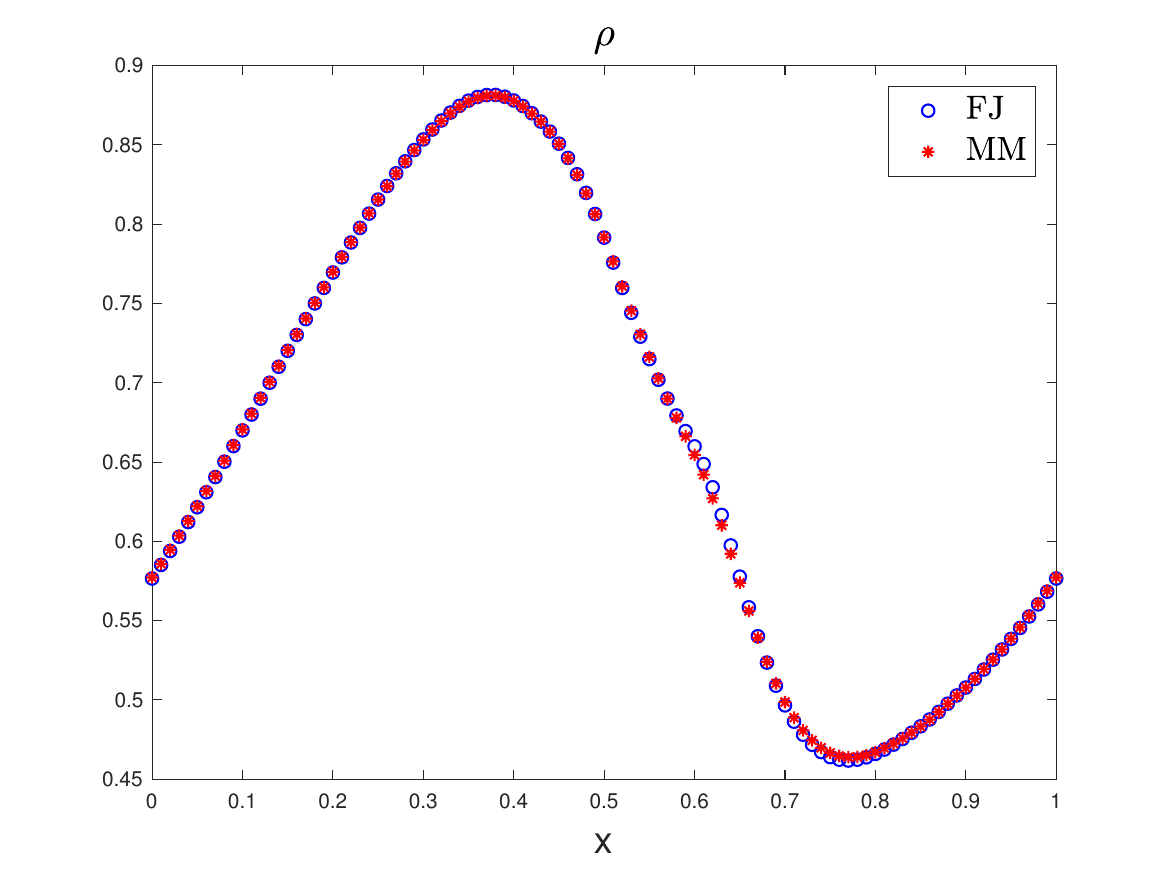}
\centering
\includegraphics[width=0.45\linewidth]{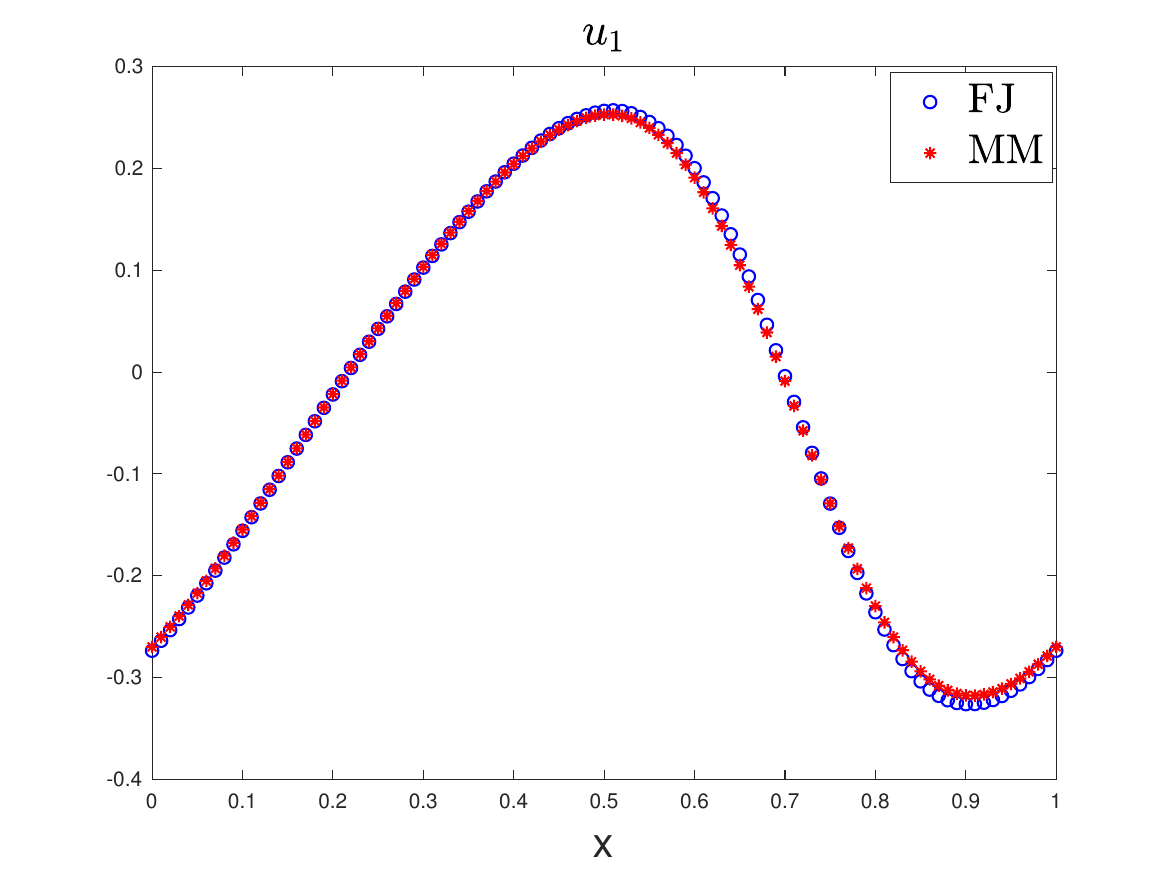}
\includegraphics[width=0.45\linewidth]{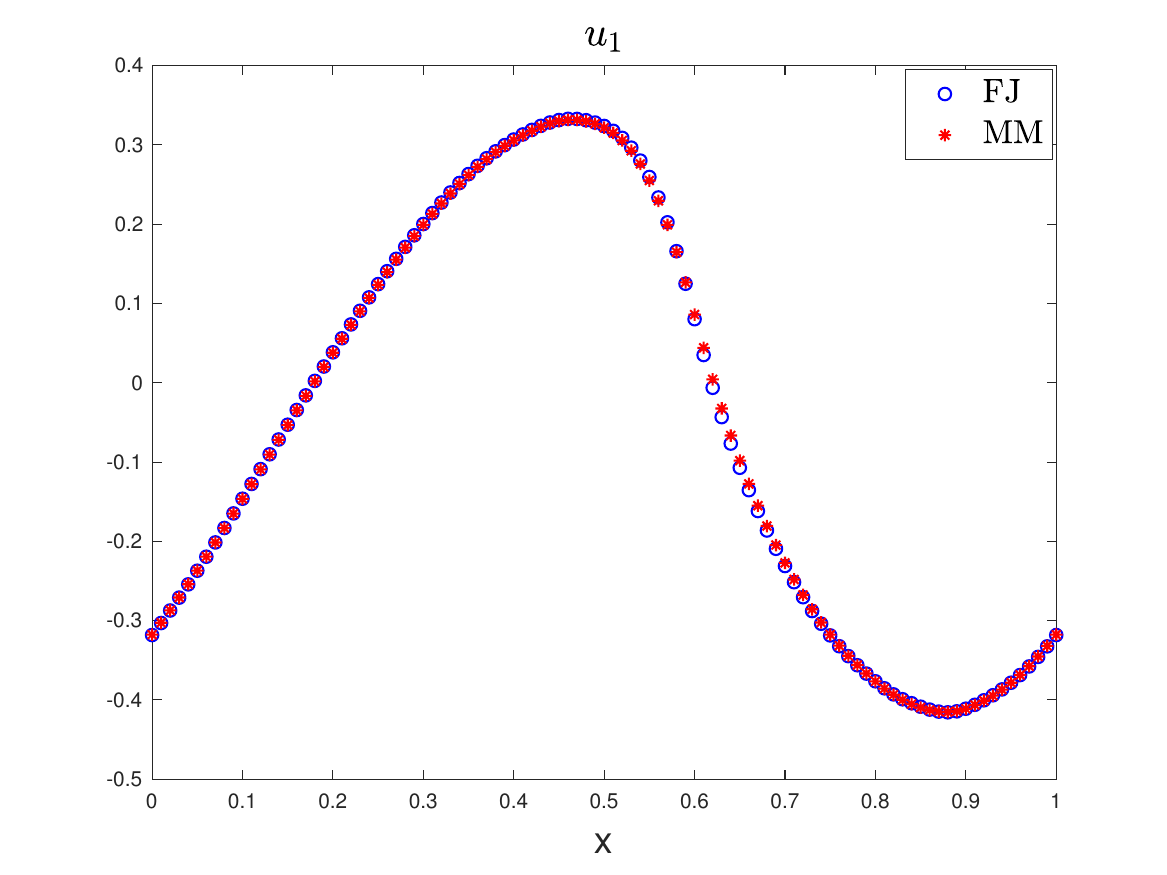}
\centering
\includegraphics[width=0.45\linewidth]{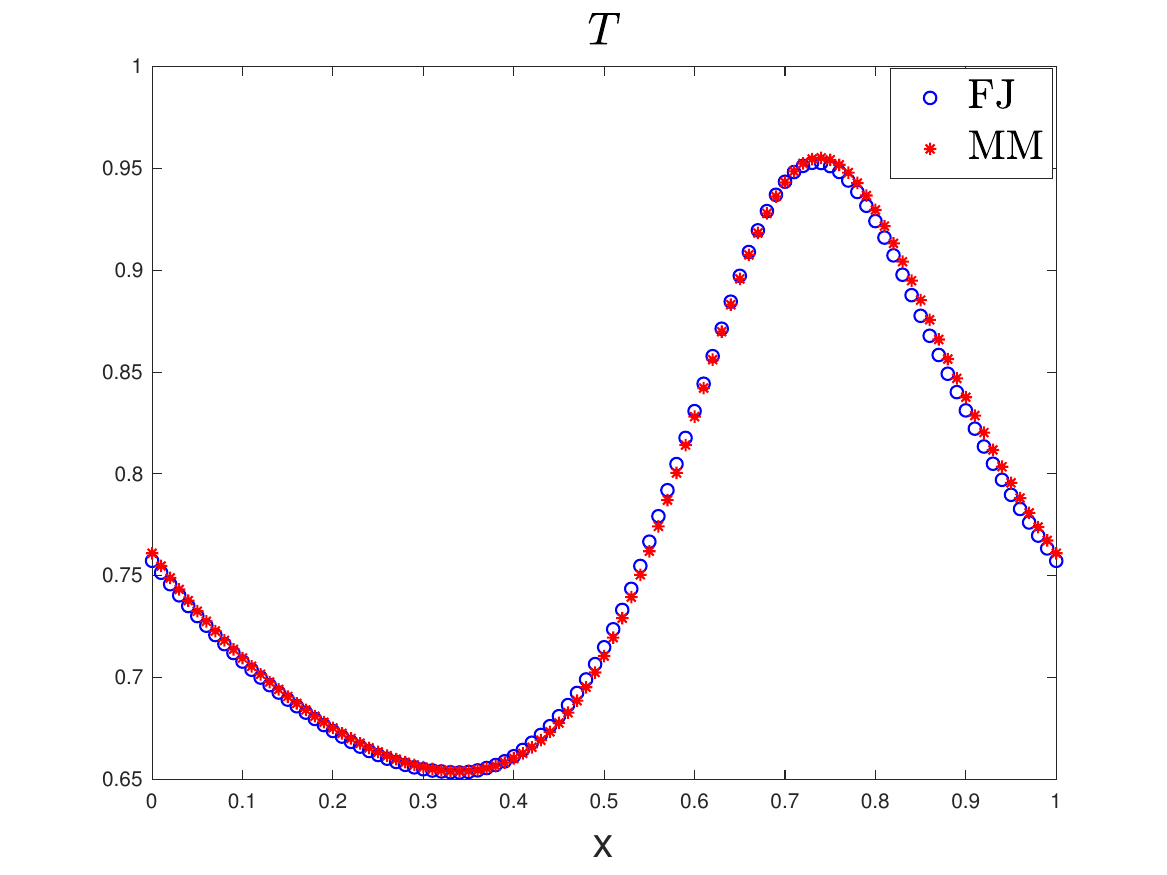}
\includegraphics[width=0.45\linewidth]{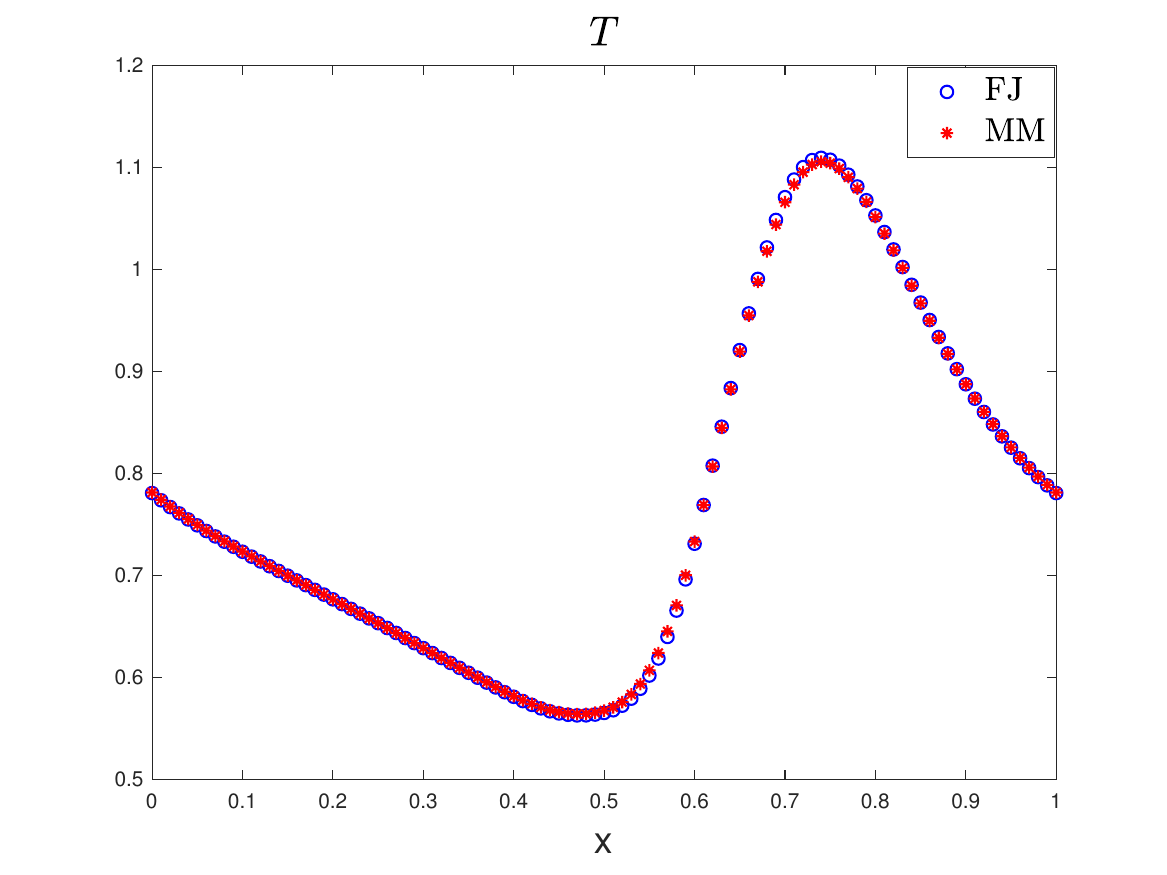}
\caption{Test I (a). $t=0.2$. 
First column: $\varepsilon=1$. Second column: $\varepsilon=10^{-3}$. }
\label{TestI-a-sol}
\end{figure}


\begin{figure}[H]
\centering
\includegraphics[width=0.49\linewidth]{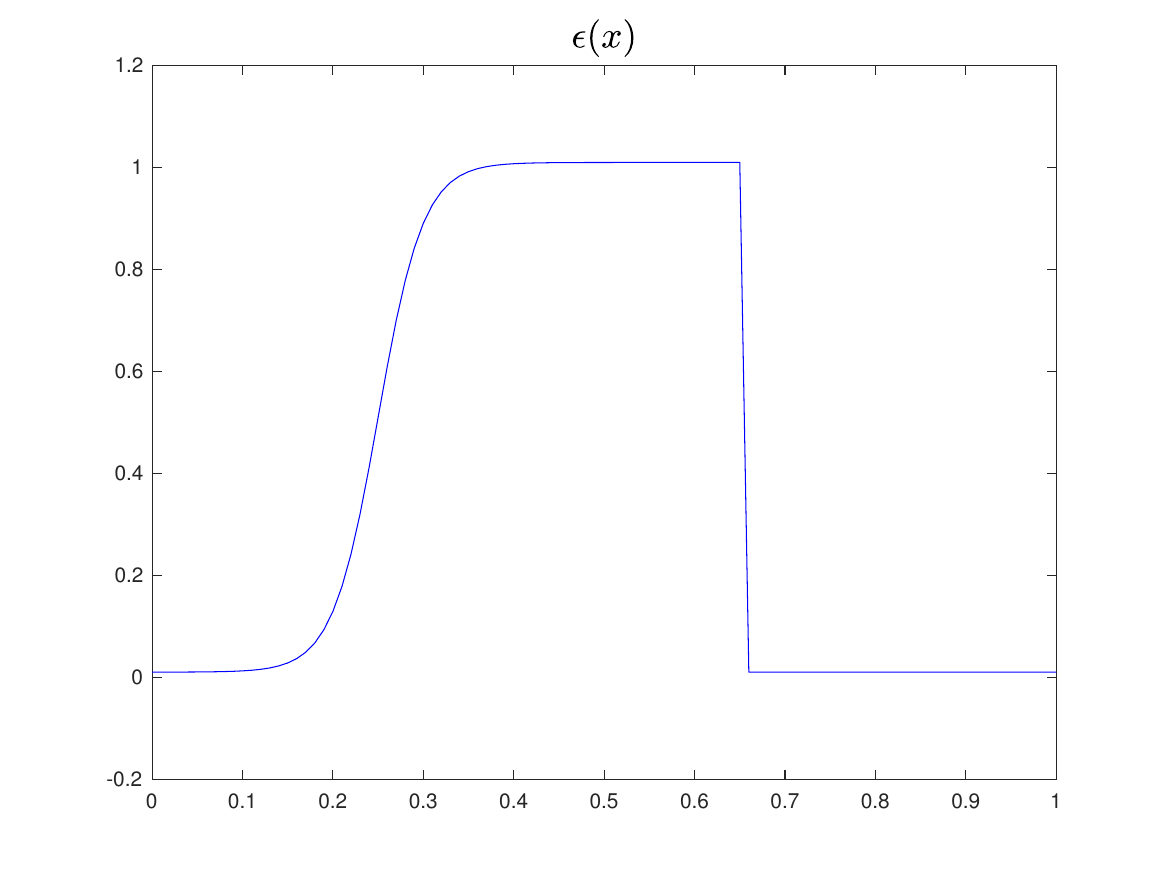}
\caption{A spatially varying Knudsen number $\varepsilon(x)$ for Test I (b). }
\label{fig_eps1}
\end{figure}

\begin{figure}[H]
\centering
\includegraphics[width=0.45\linewidth]{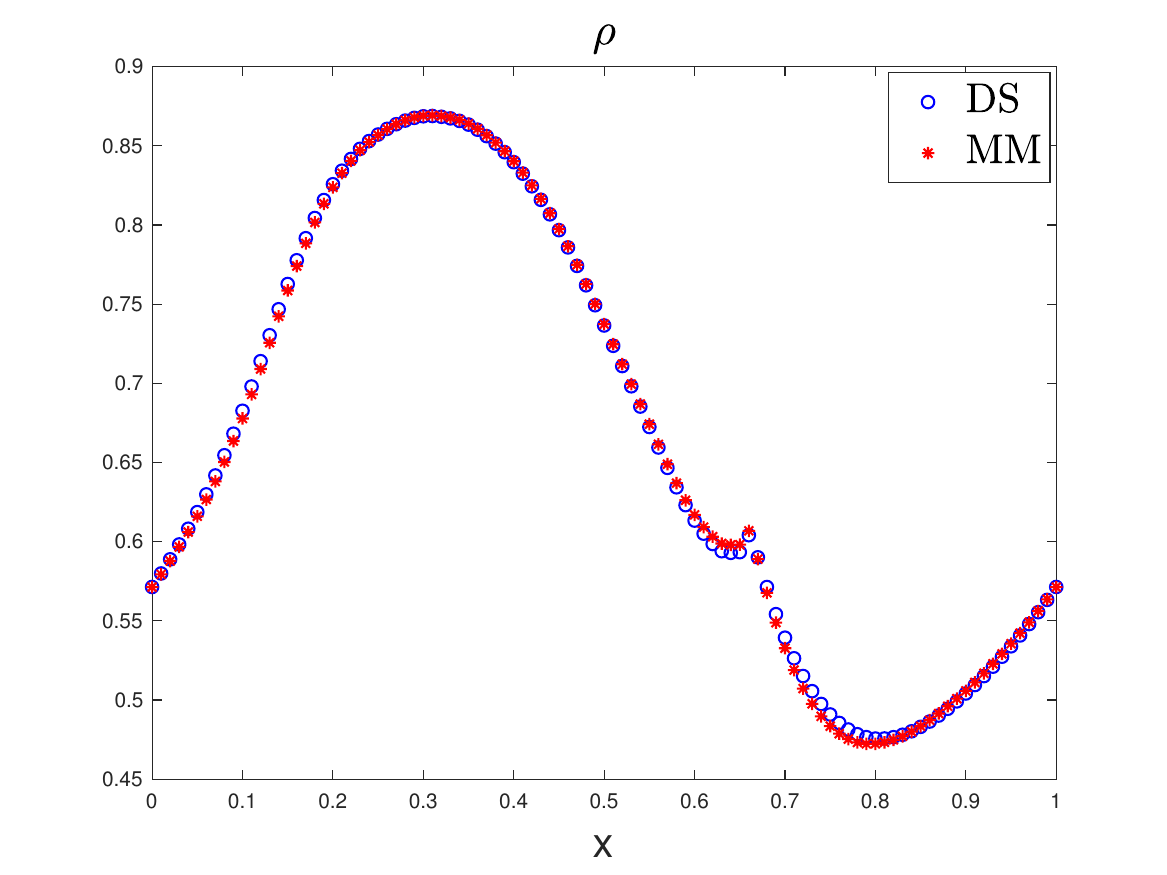}
\centering
\includegraphics[width=0.45\linewidth]{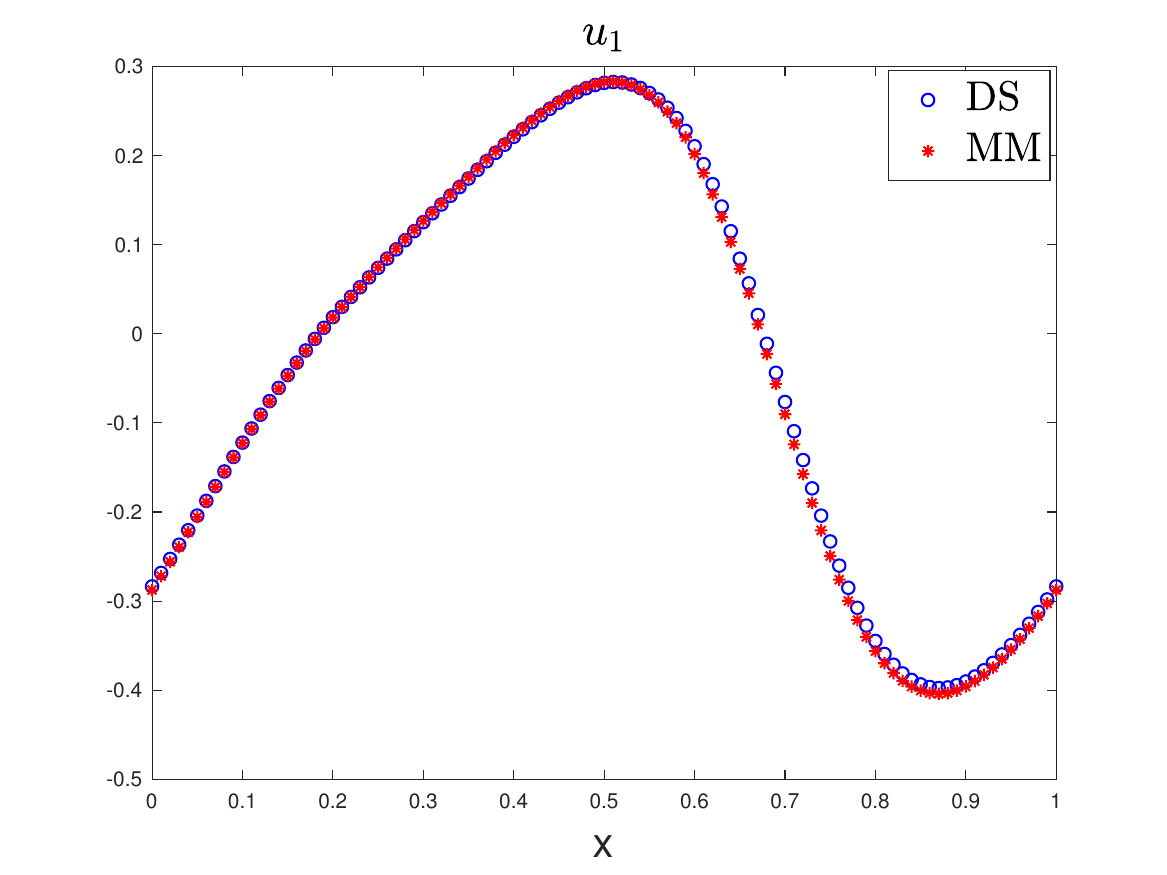}
\centering
\includegraphics[width=0.45\linewidth]{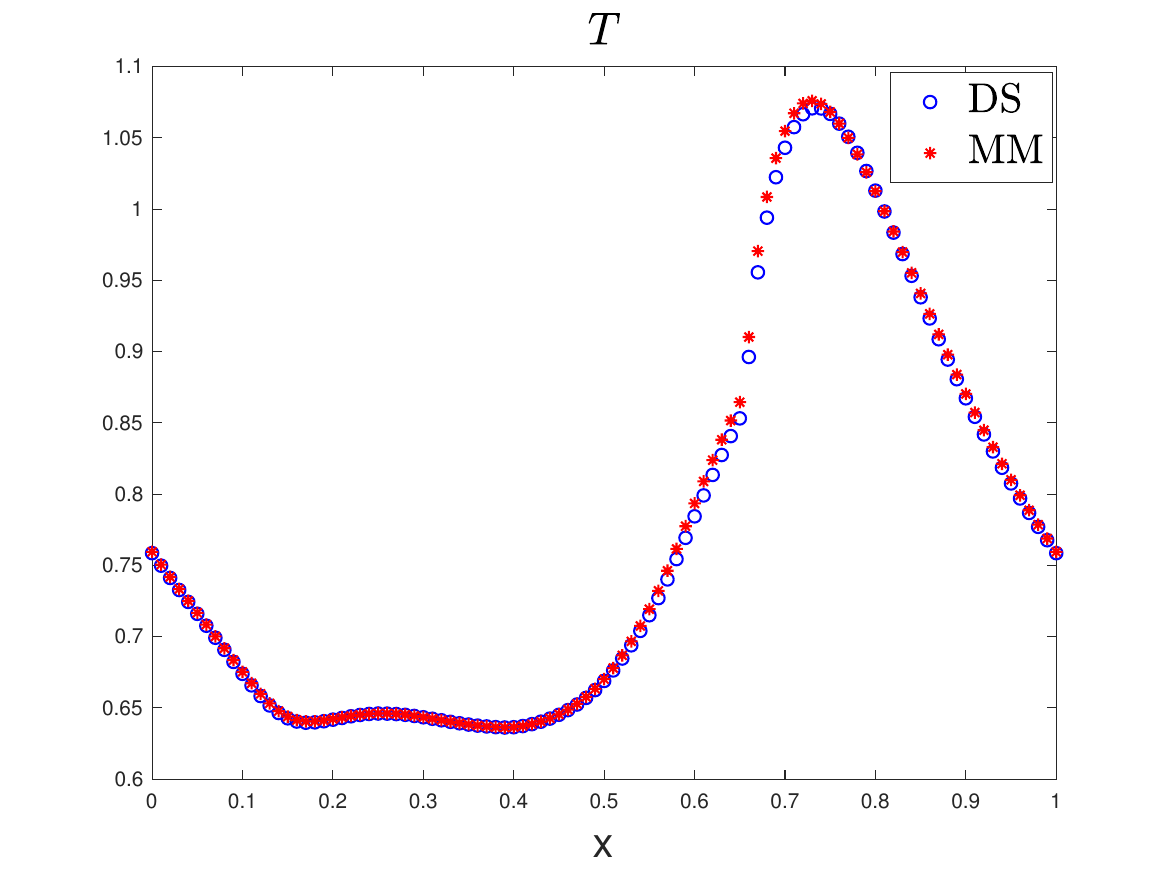}
\caption{Test I (b). Numerical solutions at $t=0.2$ by `DS' and `MM'. }
\label{TestI-b2}
\end{figure}


\begin{figure}[H]
\centering
\includegraphics[width=0.45\linewidth]{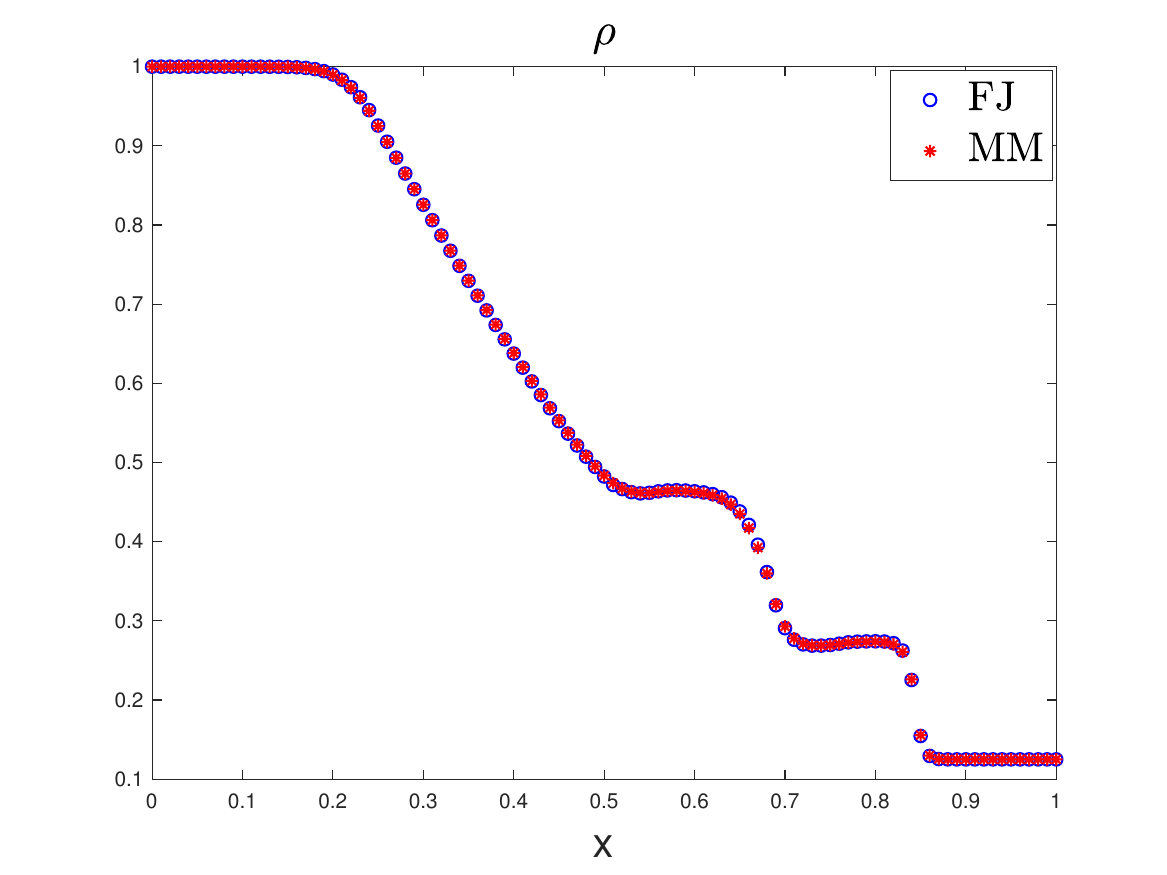}
\centering
\includegraphics[width=0.45\linewidth]{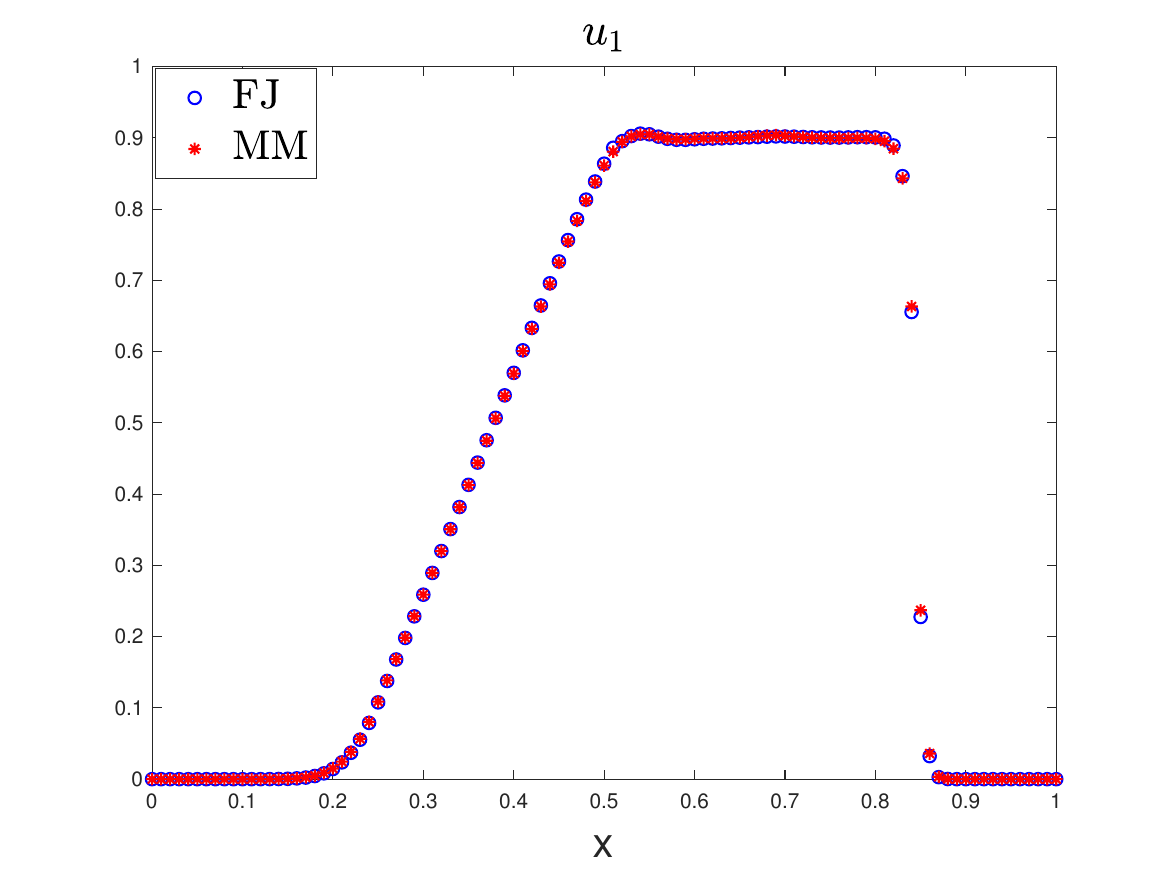}
\centering
\includegraphics[width=0.45\linewidth]{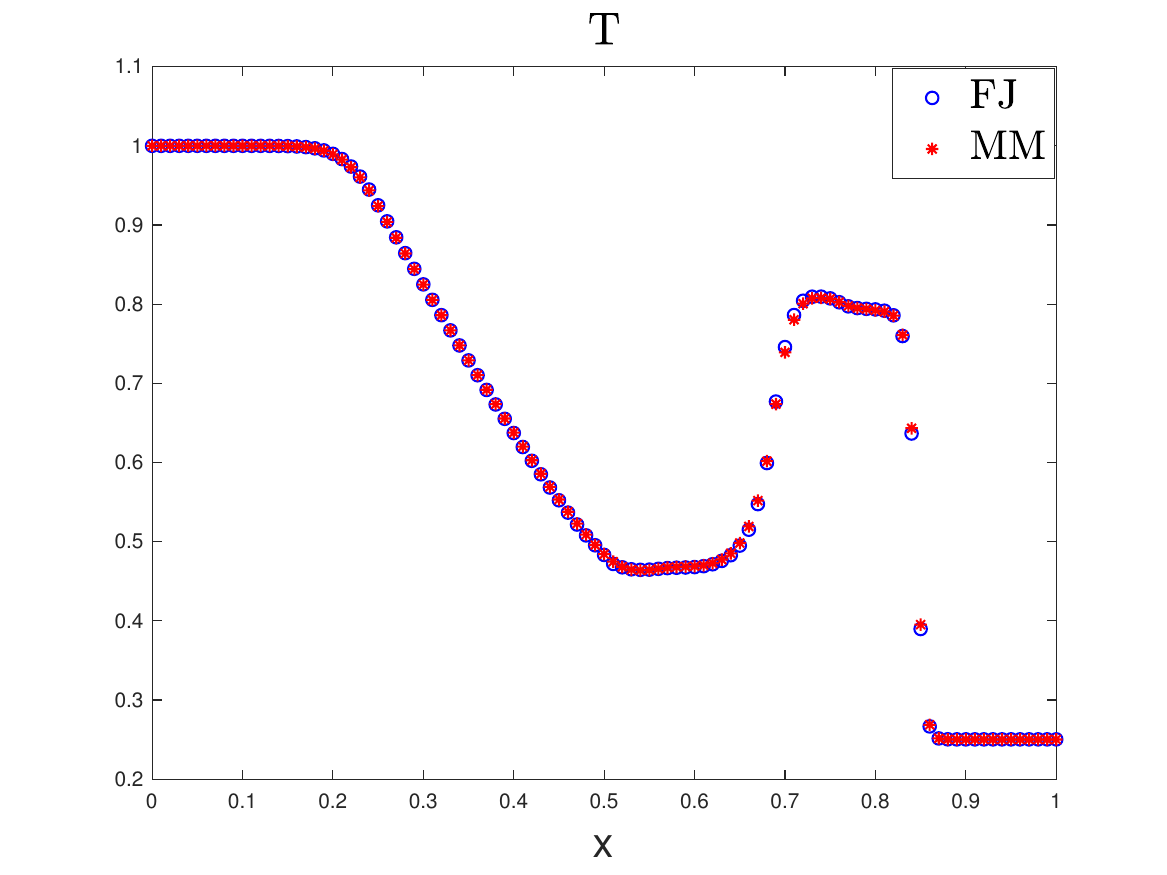}
\caption{Test I (c). `FJ' versus `MM'. Numerical solutions at $t=0.2$. $\varepsilon=10^{-4}$. }
\label{TestI-c}
\end{figure}


{\bf Test II: The micro-macro scheme for the nFPL equation}

{\bf Test II (a). }
The initial data is given by
$$ \rho^{0}(x)=\frac{2+\sin(\pi x)}{3}, \qquad u^{0} = (0.2, 0), \qquad
T^{0}(x)=\frac{3+\cos(\pi x)}{4}\,. $$
Consider the double-peak initial distribution (\ref{dp}). Let $x\in[-1,1]$, $v\in [-6,6]^2$ and $N_x=100$, $N_v=32$, $\Delta t = \Delta x/20$. 
$\varepsilon=1$. 

{\bf Test II (b). } \\
We consider a Sod shock tube test for the nFPL equation with an equilibrium initial distribution:
$$ f^{0}(x,v) = \frac{\rho^{0}}{2\pi T^{0}} e^{-\frac{|v-u^{0}|^2}{2 T^{0}}},$$
where the initial data for $\rho^{0}$, $u^{0}$ and $T^{0}$ are given by
\begin{align*}
&\displaystyle (\rho, u_1, T) = (1, 0, 1), \qquad\qquad \text{if  } -0.5 \leq x <0, \\[4pt]
&\displaystyle (\rho, u_1, T) = (0.125, 0, 0.25), \qquad \text{if  } 0 \leq x \leq 0.5. 
\end{align*}
Let $x\in[-0.5, 0.5]$, $v\in [-6,6]^2$ and $N_x=100$, $N_v=32$, $\Delta t = \Delta x/20$. $\varepsilon=10^{-3}$. 

Figure \ref{TestII-a} shows the numerical solutions of Test II (a) by `MM' compared with `DS', for both $\mathcal O(1)$ 
or moderately small $\varepsilon$, in good agreement. In Figure \ref{TestII-b} for Test II (b), one can see that the macroscopic quantities are well approximated 
for the Sod shock tube test for the nFPL equation, although
the mesh size and time steps are coarse, thus it verifies the AP property. 

\begin{figure}[H]
\centering
\includegraphics[width=0.4\linewidth]{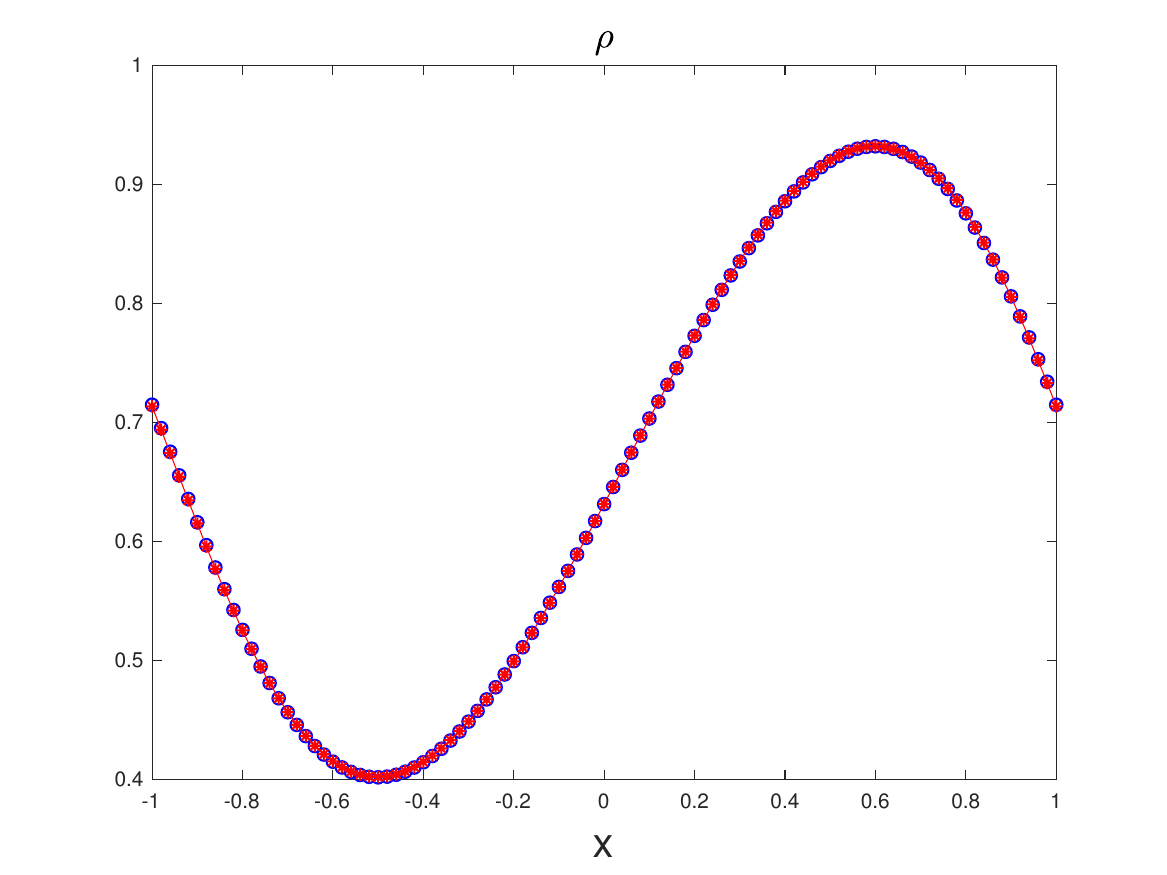}
\centering
\includegraphics[width=0.4\linewidth]{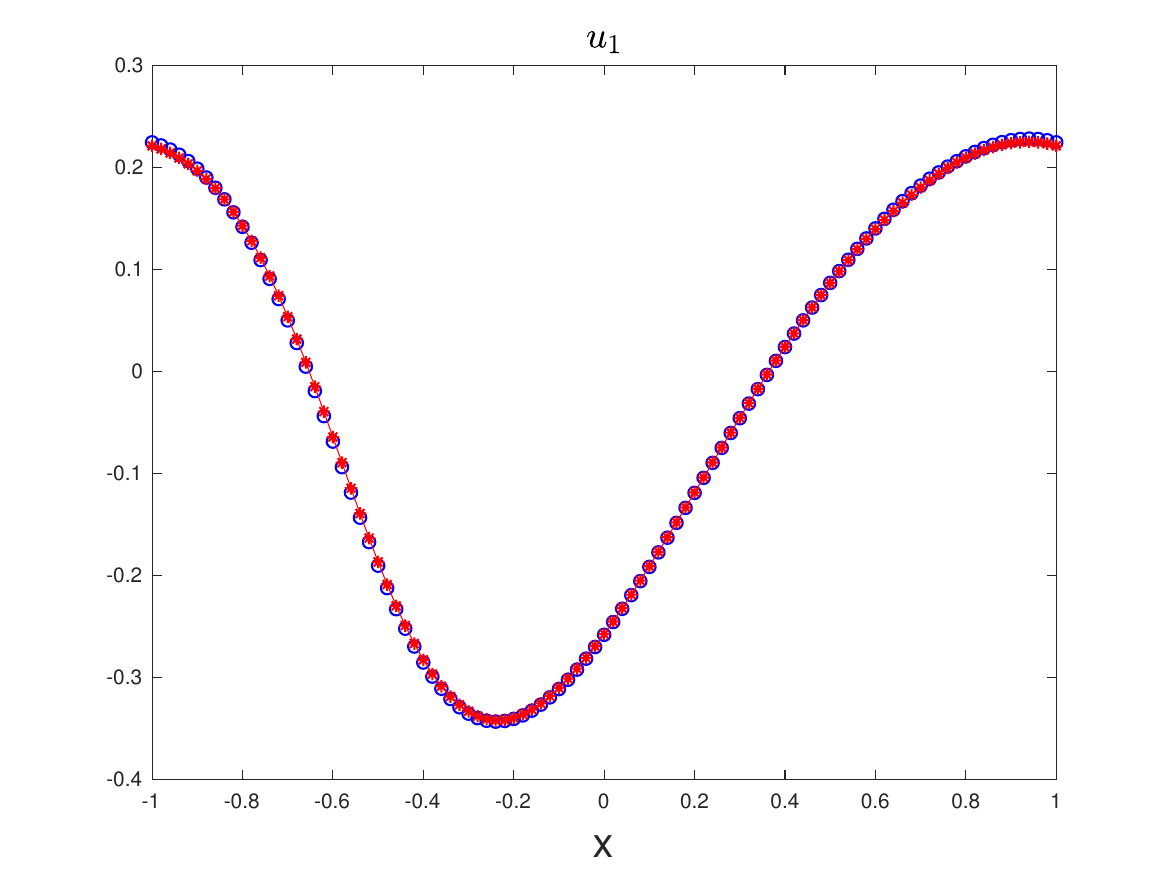}
\centering
\includegraphics[width=0.4\linewidth]{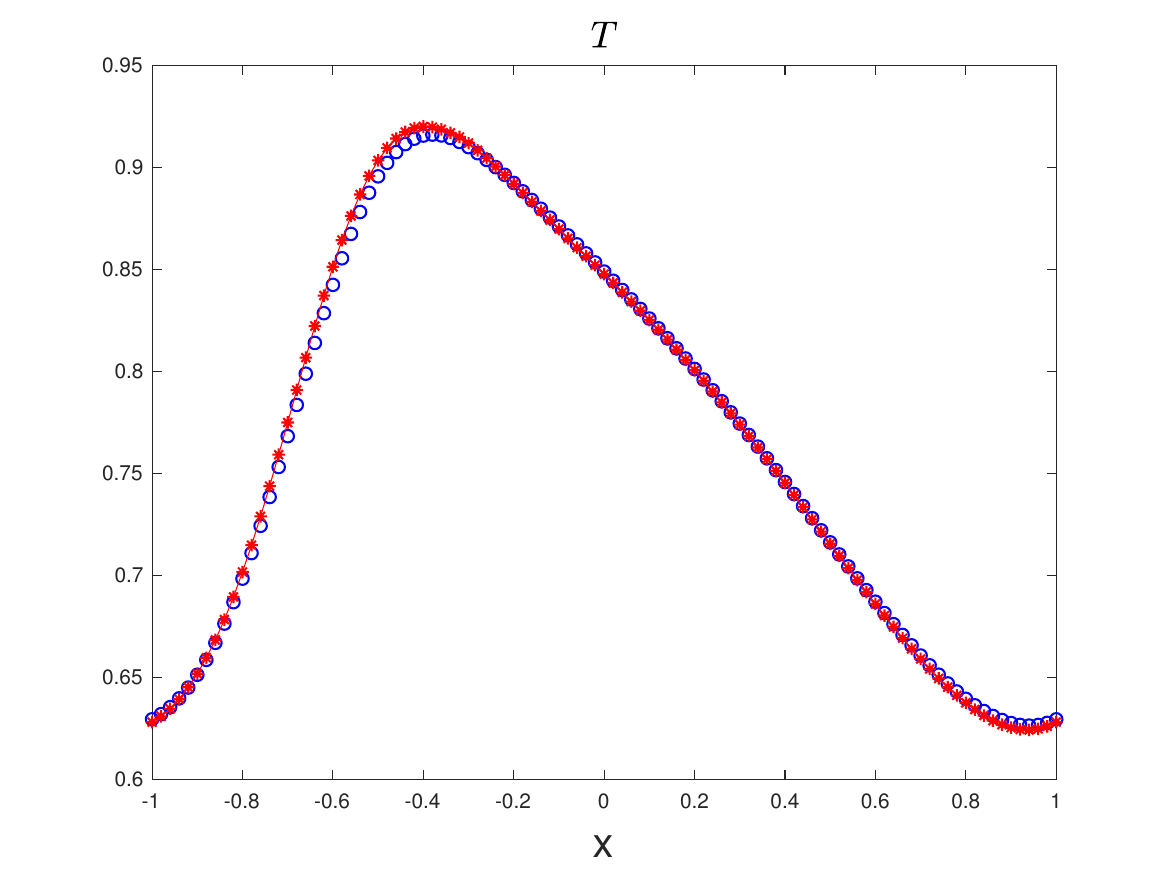}
\caption{Test II (a). Numerical solutions by `DS' (circle) and `MM' (asteroid). $\varepsilon=1$, $t=0.25$. 
}
\label{TestII-a}
\end{figure}

\begin{figure}[H]
\centering
\includegraphics[width=0.4\linewidth]{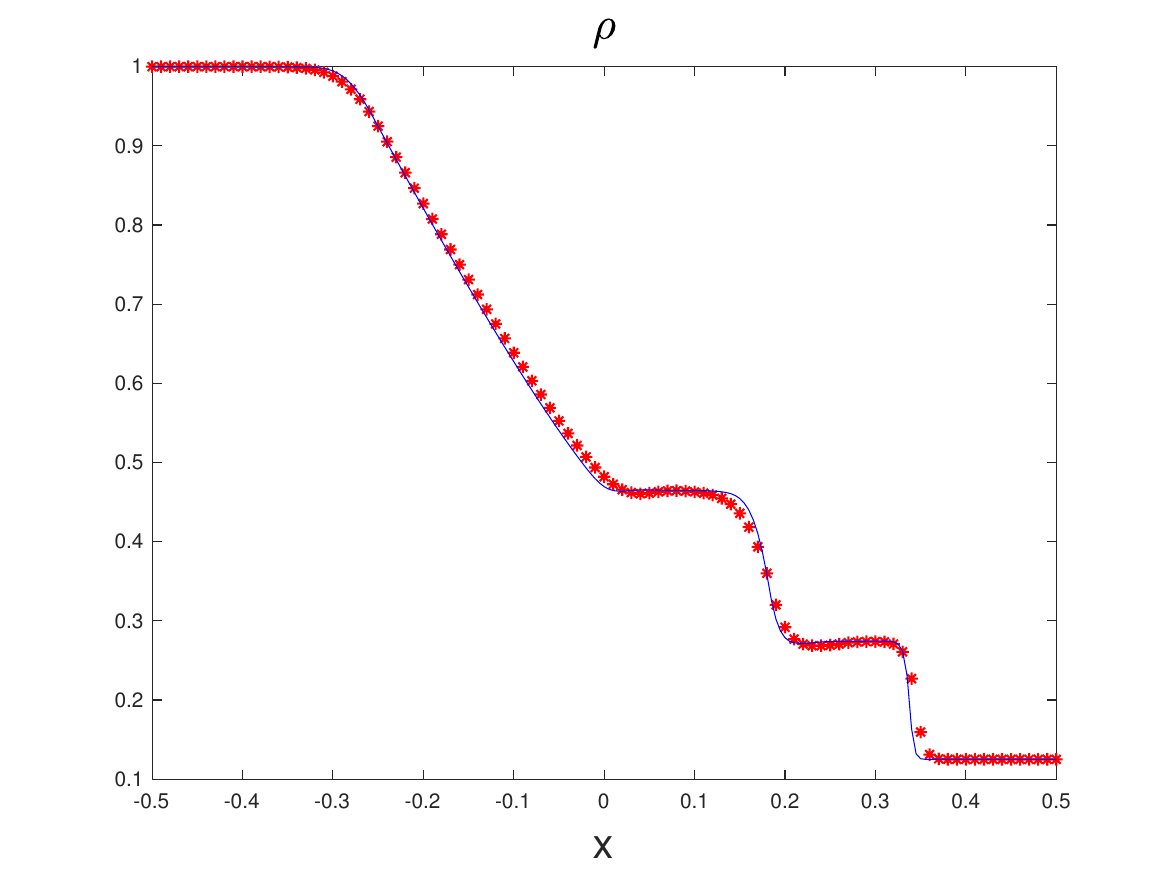}
\centering
\includegraphics[width=0.4\linewidth]{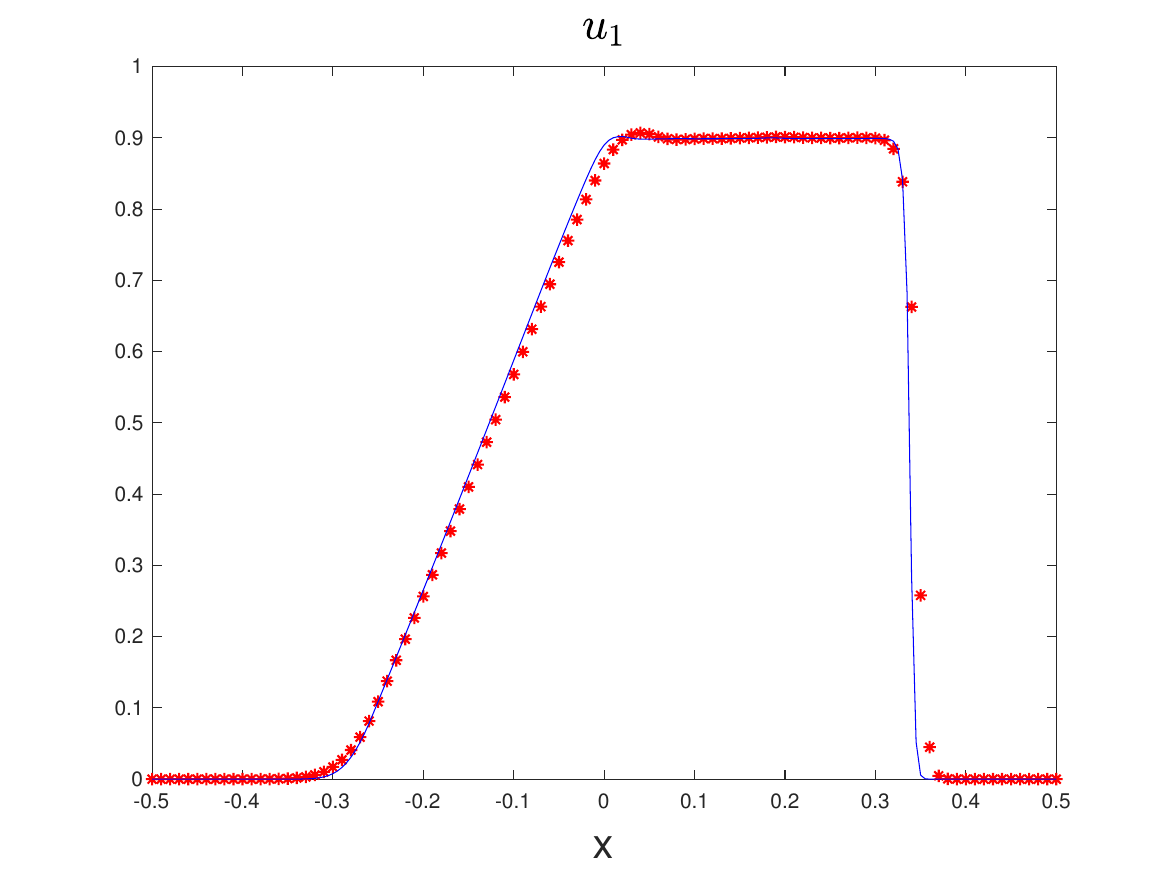}
\centering
\includegraphics[width=0.4\linewidth]{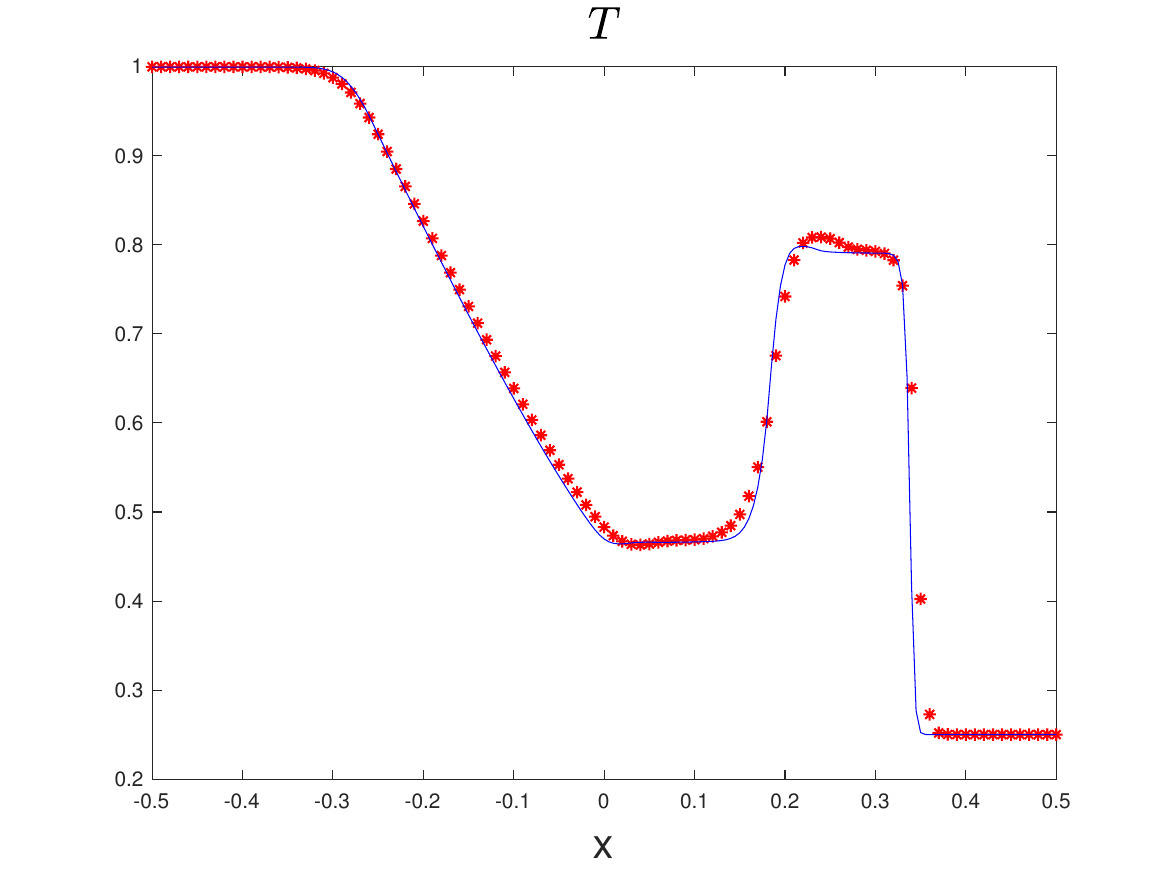}
\caption{Test II (b). Numerical solutions by `MM' (asteroid) and reference solutions by using a fine mesh of `MM' (solid line). 
$\varepsilon=10^{-3}$, $t=0.2$. }
\label{TestII-b}
\end{figure}


\section{A Conservative Scheme for the Vlasov-Amp\`{e}re-Boltzmann system}
\label{sec:7}

In order to further elaborate the issue of numerical conservative of moments
in kinetic solvers, in this section, we develop a conservative scheme for the Vlasov-Amp\`{e}re-Boltzmann system with or without the collisional term, 
which is not only of interest for the systems under study, but also gives a general guidance on how to obtain
numerically the exact conservation of moments for a general kinetic solver.

\subsection{The collisionless Vlasov-Poisson and Vlasov-Amp\`{e}re systems}
\label{Sec:VP}
First, we consider the Vlasov-Poisson (VP) system without collisions between particles, 
\begin{align}
\begin{cases}
&\displaystyle \partial_t f + v\cdot\nabla_x f - E\cdot\nabla_v f = 0\,, \\[4pt]
&\displaystyle\label{VP} \nabla_x\cdot E = c(x) - \int_{\mathbb R^d} f\, dv\,. \end{cases}
\end{align}
Here $E$ is the electric field, while $c(x)$ is the background density. The domain is given by $\Omega_{x,v} = \Omega\times\mathbb R^d$. 
This system arises in modeling collisionless plasmas \cite{Liboff}.
For simplicity, we will always assume periodic boundary condition in $x$ for $f$. 
Denote the moments as
$$ \rho = \int_{\mathbb R^d} f\, dv, \qquad 
\rho u =  \int_{\mathbb R^d} v f\, dv, \qquad E_{\text{Kin}} = \int_{\mathbb R^d} \frac{1}{2}|v|^2\, f\, dv\,. $$ 
Moment equations for (\ref{VP}) are given by 

\begin{align}
\begin{cases}
&\displaystyle \partial_t \rho + \nabla_x \cdot(\rho u)=0\,,  \\[10pt]
&\displaystyle \partial_t (\rho u) + \nabla_x  \cdot \int_{\mathbb R^d}\, v\otimes v f\, dv + E\, \rho = 0\,,    \\[10pt]
&\displaystyle \label{moment2}\partial_t \int_{\mathbb R^d}\, \frac{1}{2} |v|^2 f\, dv +\nabla_x \cdot \int_{\mathbb R^d}\,  \frac{|v|^2}{2}\, v f\, dv 
+ E \cdot (\rho u) =0\,. 
\end{cases}
\end{align}

It is easy to check that the system (\ref{VP}) conserves the total energy defined 
$$E_{\text{Total}} = \frac{1}{2}\int_{\Omega}\int_{\mathbb R^d}\, |v|^2 f\, dv dx + \frac{1}{2}\int_{\Omega}|E|^2\, dx. $$
While there have been previous works to develop schemes that conserve this
total energy, for example see \cite{YD-Cheng1, YD-Cheng}, our strategy
is different, and it serves the purpose for a generic strategy to
develop energy conserving schemes for collisional system, see the next section. 
We also refer to \cite{DG-BP} for Discontinuous Galerkin 
solvers for the Boltzmann-Poisson system. 

In order to construct a scheme that conserves $E_{\text{Total}}$, 
we solve the following Vlasov-Amp\`{e\`{e}}re (VA) system by adopting the Amp\`{e}re's law, instead of solving the Vlasov-Poisson system (\ref{VP}), 
\begin{align}
&\displaystyle \label{VP_f}\partial_t f + v\cdot\nabla_x f - E\cdot\nabla_v f = 0\,, \\[4pt]
&\displaystyle \label{Amp}\partial_t E = \rho u. 
\end{align}
Note that the VA and VP systems are equivalent when the charge solves the continuity equation 
$$ \partial_t \rho + \nabla_x\cdot \rho u =0. $$

{\bf Step 1. }\,  Update $f^{n+1}$ by solving (\ref{VP_f}) explicitly, that is, 
\begin{equation} f^{n+1} = f^{n} - \Delta t\, (v\cdot\nabla_x f^n + E^n\cdot\nabla_v f^n). \end{equation}
Here the transport term $v\cdot\nabla_x f$ is approximated by a 
non-oscillatory high resolution shock-capturing method, and a spectral 
discretization in the velocity space is used for the term $E\cdot\nabla_v f$. 
\\[2pt]

{\bf Step 2. }\, Update $E^{n+1}$ by using a forward Euler solver of (\ref{Amp}), 
\begin{equation}\label{E_Amp1} E^{n+1} = E^n + \Delta t\, (\rho u)^n. \end{equation}

{\bf Step 3. }\, Update the moments at $t^{n+1}$ by solving equations (\ref{moment2}) and using $f^n$. 
\begin{align}
\label{moment3}
\begin{cases}
&\displaystyle \frac{\rho^{n+1}-\rho^n}{\Delta t} + \nabla_x \cdot  \int_{\mathbb R^d} v f^n \, dv =0\,, \\[10pt]
&\displaystyle \frac{(\rho u)^{n+1}-(\rho u)^n}{\Delta t} + \nabla_x \cdot \int_{\mathbb R^d} v\otimes v f^n\, dv + E^n \cdot\rho^{n} = 0\,,  \\[10pt]
&\displaystyle \frac{E_{\text{Kin}}^{n+1}-E_{\text{Kin}}^n}{\Delta t} +\nabla_x \cdot \int_{\mathbb R^d} \frac{|v|^2}{2}\, v f^n\, dv 
+ \frac{E^n + E^{n+1}}{2}\cdot (\rho u)^{n} =0\,. 
\end{cases}
\end{align}

\begin{theorem}
  Let $(\rho, u, E_{kin}, E)_i$ be the numerical approximation of the
  corresponding quantities at grid point $x_i$. If one discretizes the
  divergence term in (\ref{moment3}) by a conservative spatial discretization,
  then one has the conservations of total mass and energy
  \begin{equation}\label{E_Kin_Dis}
     \sum_{i=0}^{N_x}\, \rho_i^{n+1}\,  = \\
 \sum_{i=0}^{N_x}\, \rho_i^{n}, \quad
 \sum_{i=0}^{N_x}\,  \left( (E_{\text{Kin}}^{n+1})_{i}+ \frac{1}{2}(E_i^{n+1})^2\right) 
=  \sum_{i=0}^{N_x}\,  \left( (E_{\text{Kin}}^{n})_{i}+ \frac{1}{2} (E_i^{n})^2\right). 
\end{equation}
\end{theorem}

\begin{proof}
  Sum over all $i$ for the spatial discretized system of the first equation
  in (\ref{moment3}) gives
\begin{equation}\label{rho-cons}\Delta x \sum_{i=0}^{N_x}\, \rho_i^{n+1}\,  = \Delta x \sum_{i=0}^{N_x}\, \rho_i^{n}. 
\end{equation}
Also the third equation of (\ref{moment3}), after using (\ref{E_Amp1}), gives
\begin{equation}\label{E_Kin} \Delta x \sum_{i=0}^{N_x} \frac{(E_{\text{Kin}})_i^{n+1} - (E_{\text{Kin}})_i^n}{\Delta t} 
  + \frac{1}{2 \Delta t}\, \sum_{i=0}^{N_x} \sum_{i=0}^{N_x}\left(E_i^{n+1}
  + E_i^n \right)(\rho u)^n_i = 0. \end{equation}
Using (\ref{E_Amp1}), one obtains (\ref{E_Kin_Dis}).
\end{proof}
Since the goal of this section is to preserve the total energy in time,
we will only conduct numerical examples to check the conservation property, and not consider other discretization issues for the system. 
\\[2pt]

{\bf Test III} 

Let the initial data be $$ f(t=0, x, v) = (1+\cos(2x))\, \frac{e^{-|v|^2/2}}{\sqrt{2\pi}}. $$
Periodic boundary condition in space is assumed for $f$, $E$ and $\phi$. 
The initial condition of the electric field $E$ can be obtained from the Poisson equation
$$ -\Delta_x \phi = c(x) - \int_{\mathbb R^d} f\, dv, $$ 
by using a second-order finite-difference Poisson solver and central difference spacial discretization for 
$E = -\nabla_x \phi$. To make the solution unique, we also set the boundary data for $\phi$, 
$$ \phi(x_L) = \phi(x_R) = 0. $$ Set $c(x)=1$. 

Let $x\in[0, \pi]$, $v\in[-2\pi, 2\pi]$, $N_x=200$, $N_v=64$ and $\Delta t 
= \Delta x/20$ in the following test. 
In Figure \ref{VP_Fig}, the first figure shows the density $\rho(x)$ at time $t=0.5$, computed from either solving the moment equations (`ME')
or from the solution $f$ (`Mf'). In the second figure, the electric field $E(x)$ is compared between using the Poisson equation or the Amp\'ere's Law. 
In the third figure, we plot mass as a function of time and compare it between using `ME' and `Mf'. One can see that the two solutions match well in the first three figures. In the fourth figure, the total energy, which is obtained from solving the Vlasov-Poisson (`Mf-Poiss'), Vlasov-Amp\'ere system (`Mf-Amp'), or the 
moment equations and the Amp\'ere's Law (`ME-Amp') respectively.
This verifies the proof shown in (\ref{E_Kin_Dis}) that the numerical total energy is perfectly conserved for `ME-Amp'. 
The other two lines of `Mf-Poiss' and `Mf-Amp', though non-conserved, has a small numerical error (in the order of numerical consistency error). 
However, it is remarkable to observe that the numerical total energy has an $O(10^{-3})$ error for long time which is 
exactly the same order of magnitude of the numerical total error in the simulation of the Vlasov--Poisson--Landau system computed by means of 
operator splitting of a DG scheme for the collisionless Vlasov--Poisson advection coupled to the collisional conservative step in Figure 12 of \cite{zhang2017conservative}, under the same boundary conditions as assumed here. 
 
\begin{figure}[H]
\centering
\includegraphics[width=0.49\linewidth]{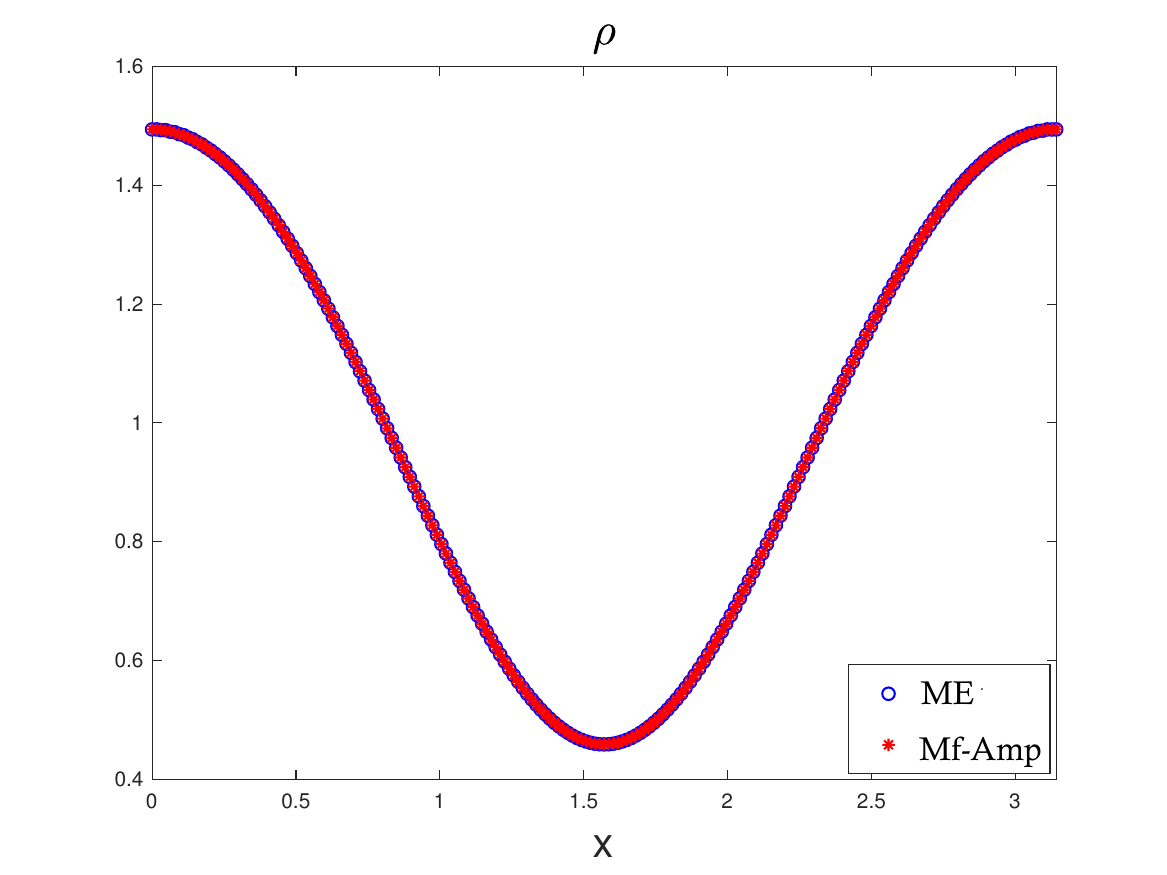}
\centering
\includegraphics[width=0.49\linewidth]{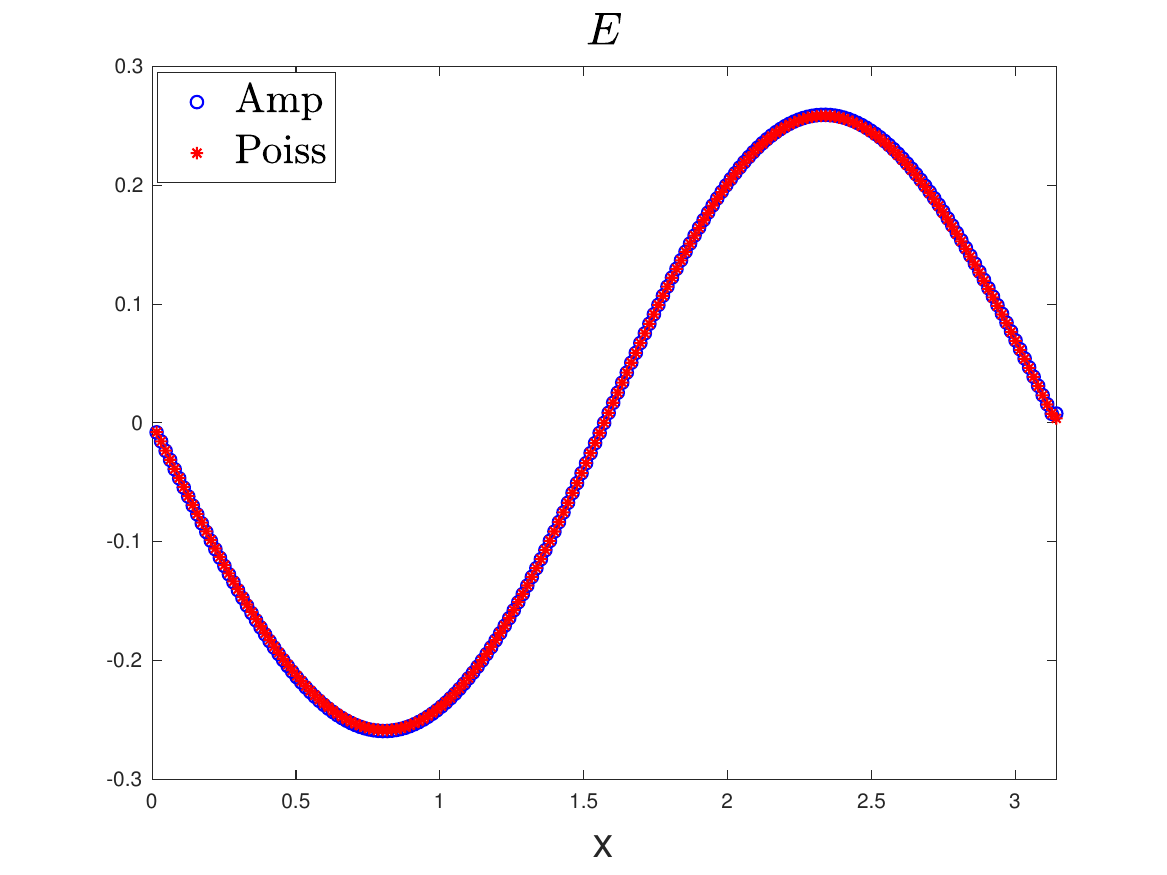}
\centering
\includegraphics[width=0.49\linewidth]{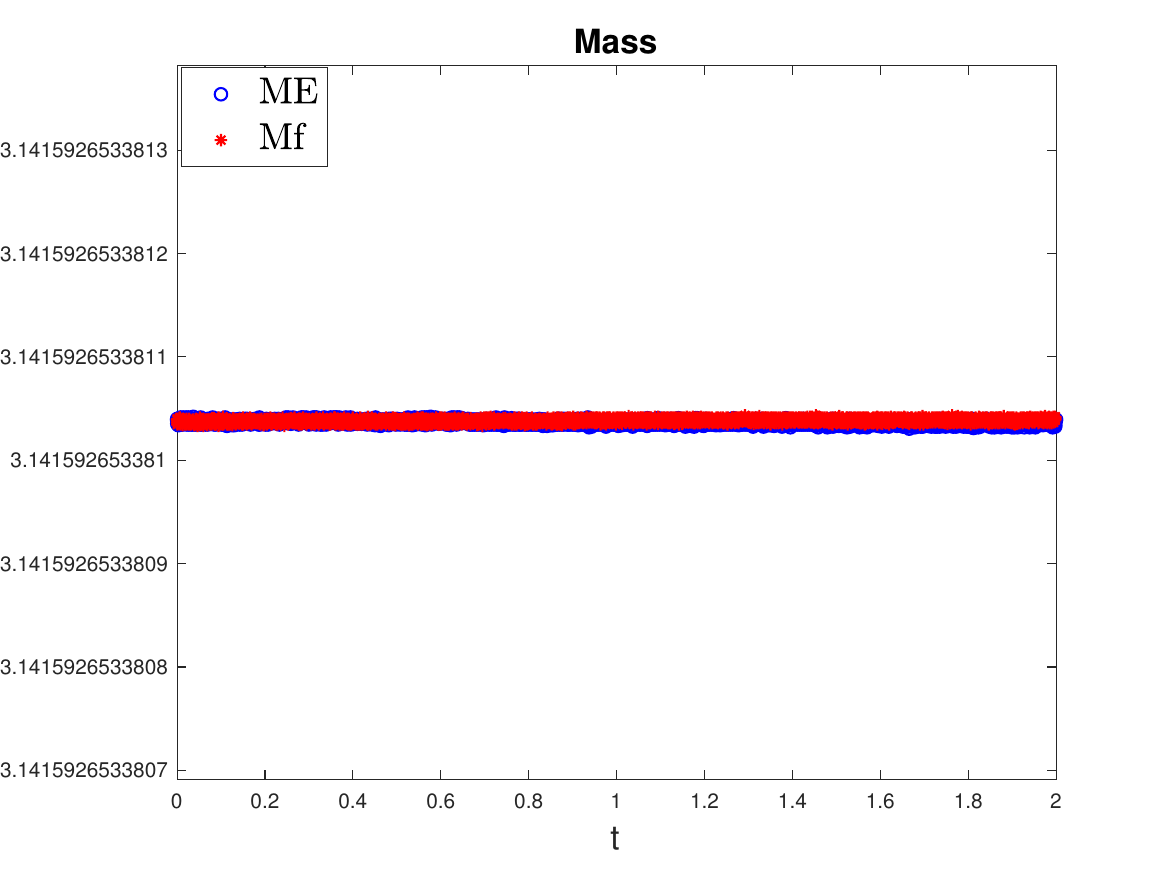}
\centering
\includegraphics[width=0.49\linewidth]{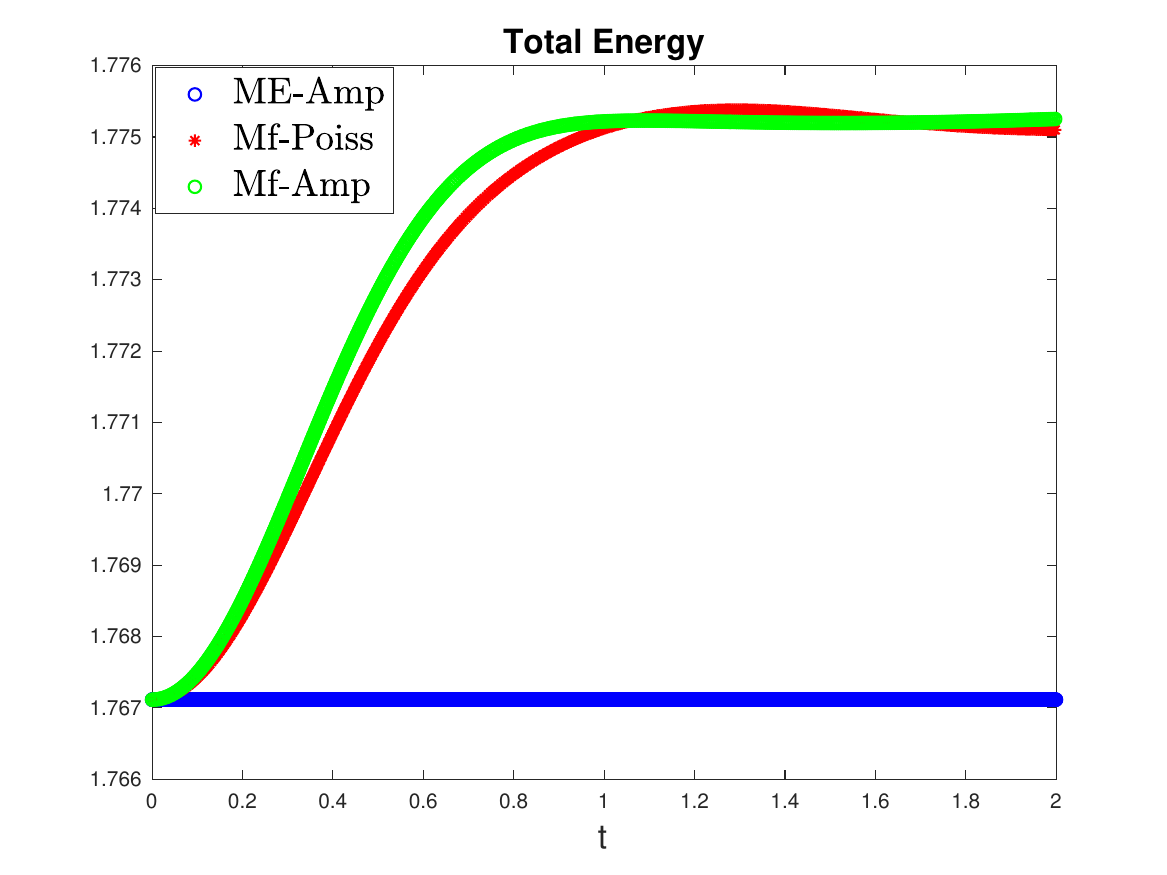}
\caption{Test III. $\rho, E$ at $t=0.5$ in the first row; mass and total energy with respect to time in the second row. }
\label{VP_Fig}
\end{figure}

\subsection{The Vlasov-Amp\'ere-Boltzmann system}

We can easily extend the scheme introduced in section \ref{Sec:VP} to the collisional problems, 
for example the Vlasov-Amp\'ere-Boltzmann system that will be studied in this section. This system models collisional plasma \cite{krall1973principles}.

Consider the Vlasov-Amp\'ere-Boltzmann system, 
\begin{align}
\begin{cases}
&\displaystyle \partial_t f + v\cdot\nabla_x f - E\cdot\nabla_v f = \frac{1}{\varepsilon}\mathcal Q_{\text{B}}(f,f)\,, \\[4pt]
&\displaystyle \label{VA}\partial_t E = \rho u\,. 
\end{cases}
\end{align}

The time-discretized scheme for the moments equations of (\ref{VA}) are the same as the Vlasov-Poisson and is given in (\ref{moment2}). 
With $\rho$, $\rho u$ and $E_{\text{Kin}}$, one can get the temperature $T$ using the relation 
$ E_{\text{Kin}}=\frac{1}{2}\rho\, u^2 + \frac{N_d}{2}\, \rho T$ and thus compute the local equilibrium
\begin{equation}
\label{moment_M} M_{\text{eq}}(x,v) = \frac{\rho(x)}{(2 \pi T(x))^{N_{d}/2}}\, \exp\left( -\frac{(v-u(x))^2}{2 T(x)} \right). 
\end{equation}

To overcome the stiffness of the collision operator in the fluid regime, we simply use the Filbet-Jin
penalty AP schemes here. 

Step 2 and Step 3 given by (\ref{E_Amp1}) and (\ref{moment3}) to update $E$ and the moments quantities are exactly the same as the scheme 
given in section \ref{Sec:VP}. With the collision term in (\ref{VA}), step 1 correspondingly becomes 
$$ \frac{f^{n+1}-f^n}{\Delta t} + v\cdot\nabla_x f^n - E^n \cdot\nabla_v f^n = \frac{\mathcal Q(f^n) - P(f^n)}{\varepsilon}
+ \frac{P(f^{n+1})}{\varepsilon}, $$ 
which gives
$$ f^{n+1} = \frac{\varepsilon}{\varepsilon + \beta \Delta t}\left(f^n - \Delta t\, v\cdot\nabla_x f^n + \Delta t\, E^n \cdot\nabla_v f^n \right) 
+ \Delta t\, \frac{\mathcal Q(f^n) - P(f^n)}{\varepsilon + \beta\Delta t} + \frac{\beta\Delta t}{\varepsilon + \beta\Delta t}\, \mathcal M^{n+1}, $$
with $\mathcal M^{n+1}$ defined through the moments quantities solved from 
(\ref{moment3}). 

In the following numerical experiments we use $\Delta x=\pi/200$, $\Delta t=\Delta x/20$. 

In Figure \ref{TestIII-a}, we show a similar set of figures as Figure \ref{VP_Fig} above. 
The first row shows the numerical solution at output time $t=0.5$, with $\varepsilon=1$. 
The numerical solutions such as $\rho$, $E$ match well no matter whether the Amp\'ere's Law or the Poisson equation is used. 
In this test, moments (mass and total energy) are perfectly conserved if obtained from 
'ME' or 'ME-Amp', as shown in the second row of Figure \ref{TestIII-b}. 
The red line in the third figure indicates that the mass obtained from $f$ is not perfectly conserved but has a spectrally small error. 
The green (`Mf-Amp') and red (`Mf-Poiss') lines in the fourth figure show that the energy, 
if obtained from $f$ coupled with the Amp\'ere's Law or the Poisson equation for $E$, 
is not perfectly conserved but still have a small error. 

For the last test, we will only use the exactly conservative scheme and 
check the penalty method for the Vlasov-Amp\'ere-Boltzmann equation, for the case of small $\varepsilon$. 
Figure \ref{TestIII-b} shows in the first row the numerical solution $\rho$, $E$ at output time $t=0.1$, with $\varepsilon=0.05$. 
In the second row, we show that mass and total energy are perfectly conserved if using the moments equations given by (\ref{moment3}). 

\begin{figure}[H]
\centering
\includegraphics[width=0.49\linewidth]{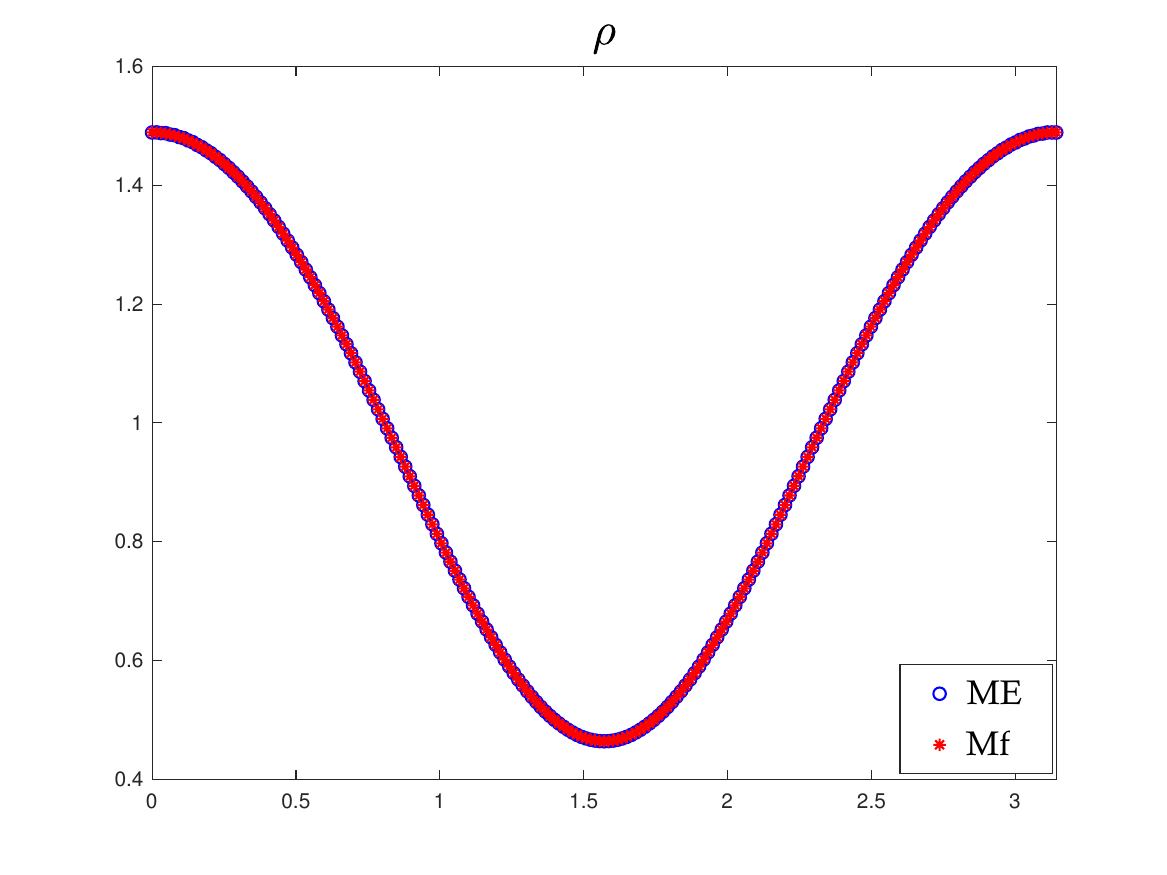}
\centering
\includegraphics[width=0.49\linewidth]{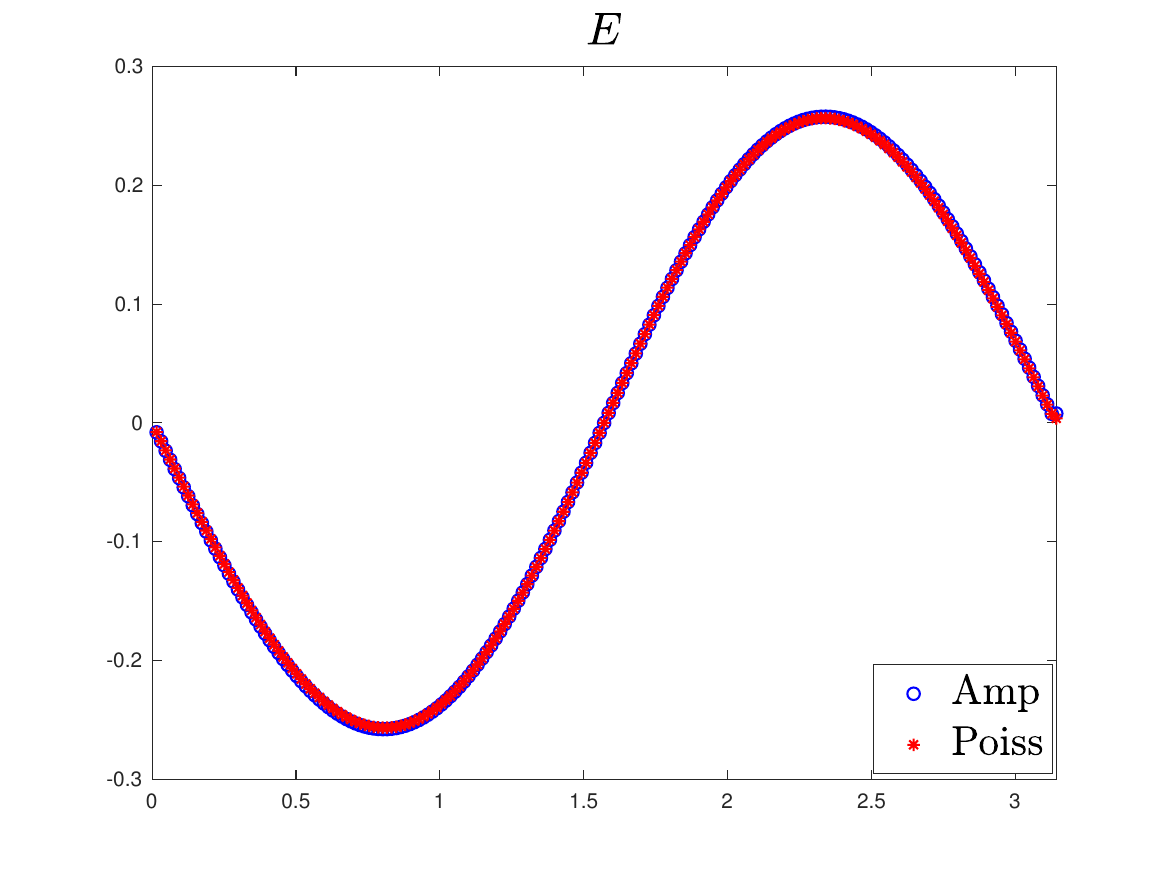}
\centering
\includegraphics[width=0.49\linewidth]{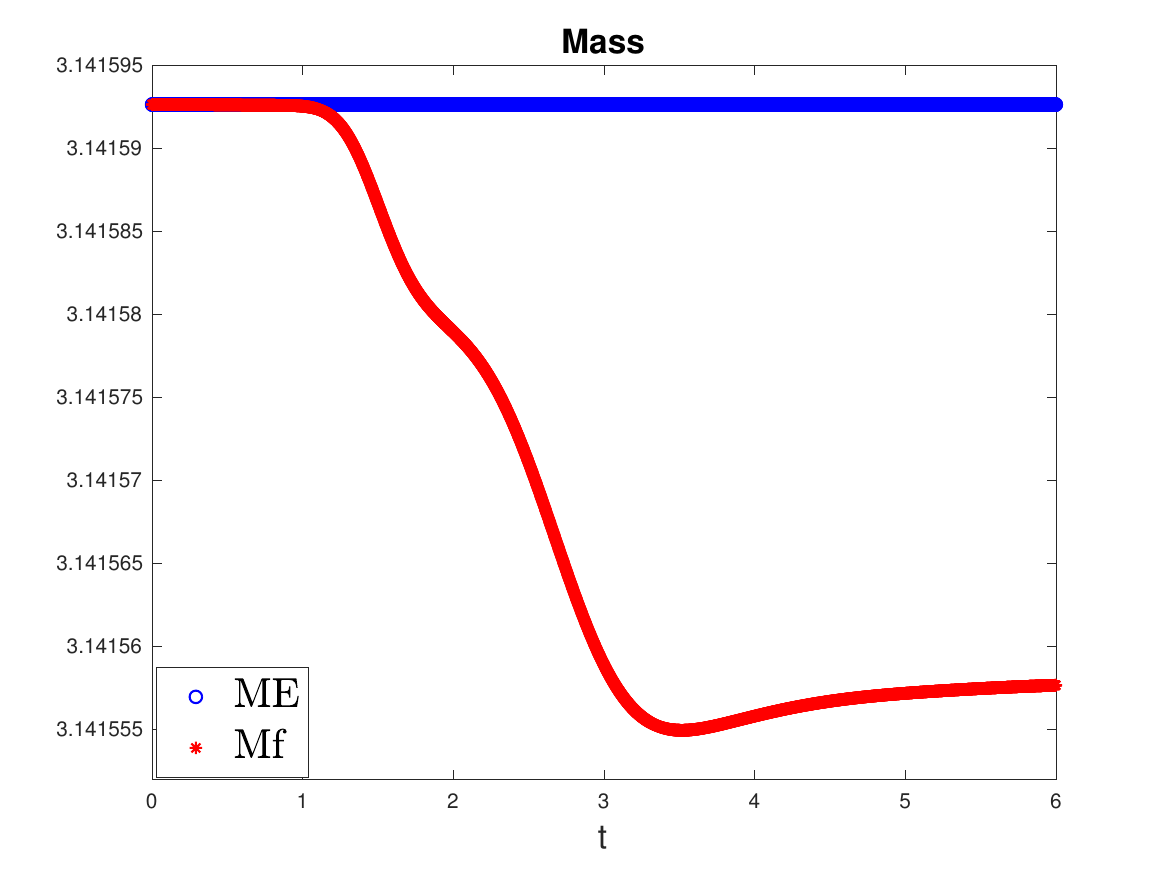}
\centering
\includegraphics[width=0.49\linewidth]{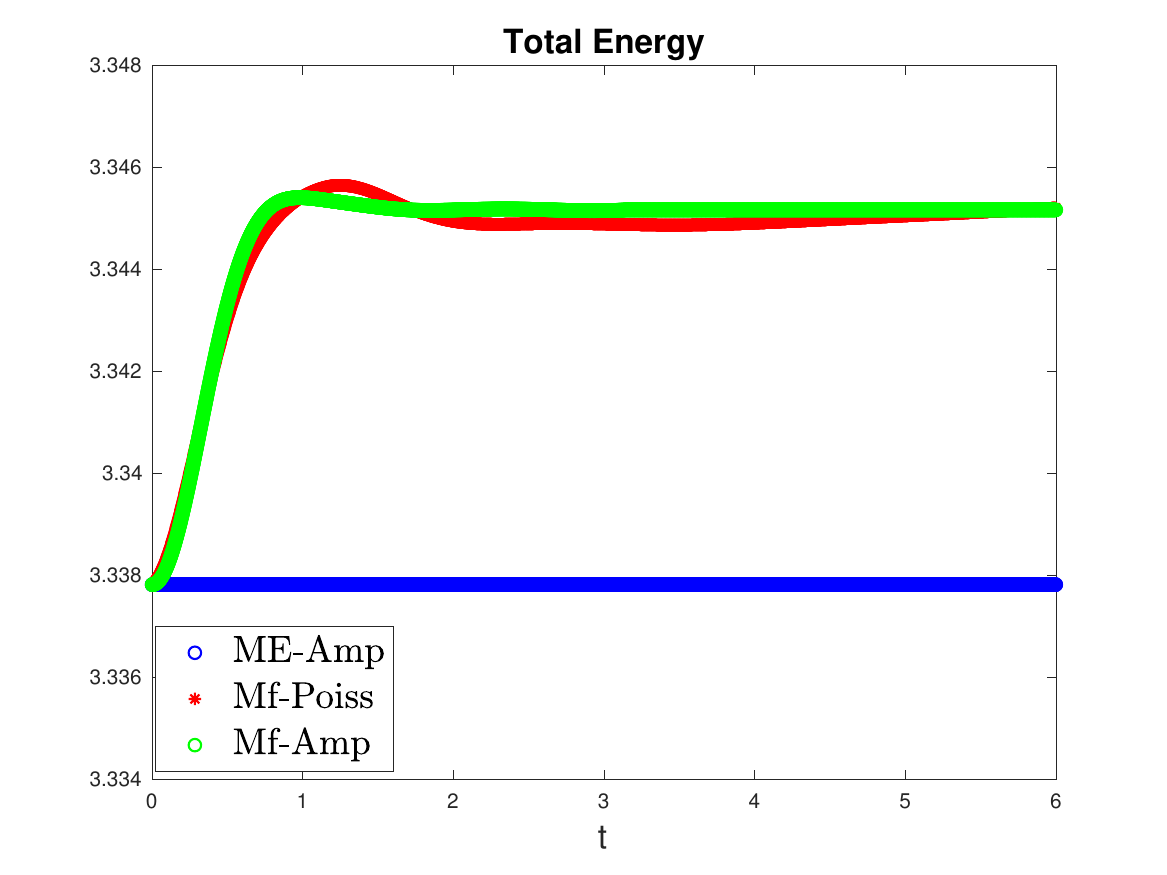}
\caption{$\varepsilon=1$.  Numerical solutions $\rho, E$ at $t=0.5$ in the first row; 
mass and total energy with respect to time in the second row. }
\label{TestIII-a}
\end{figure}

\begin{figure}[H]
\centering
\includegraphics[width=0.49\linewidth]{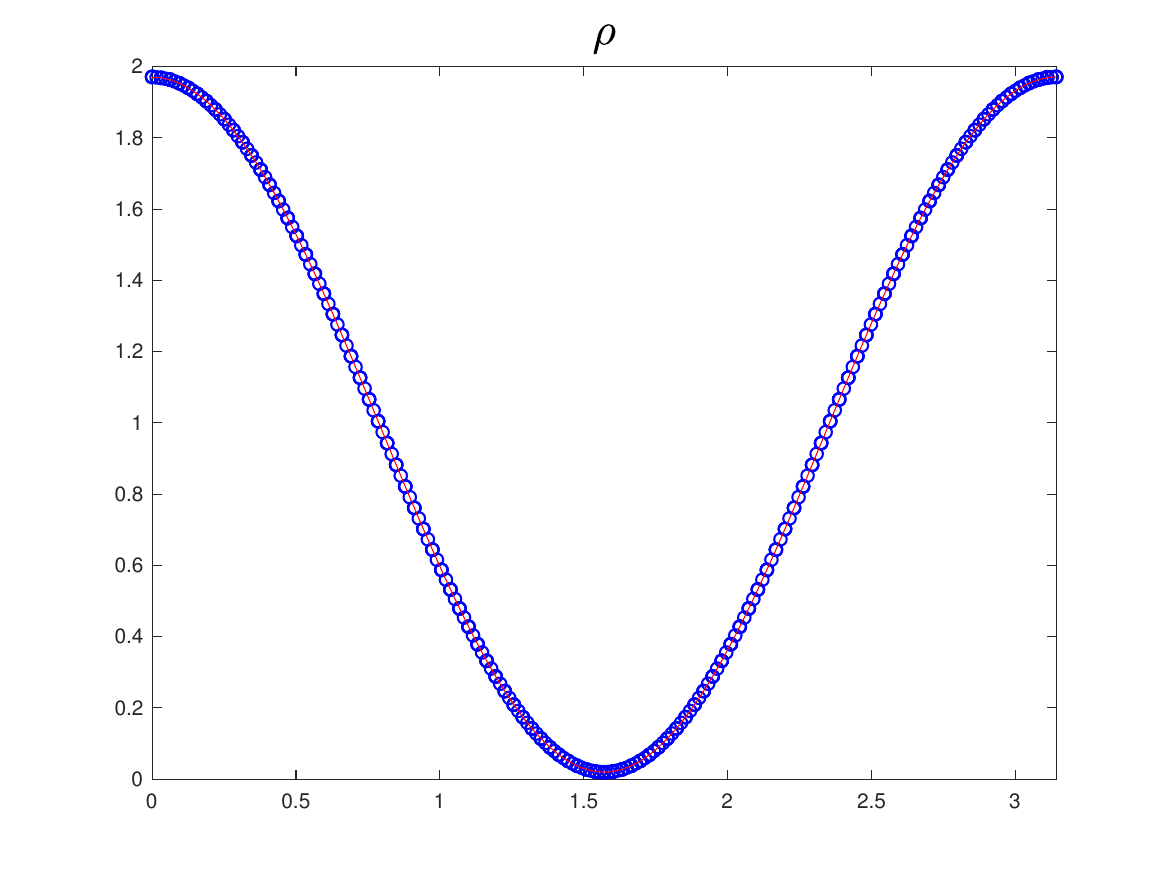}
\centering
\includegraphics[width=0.49\linewidth]{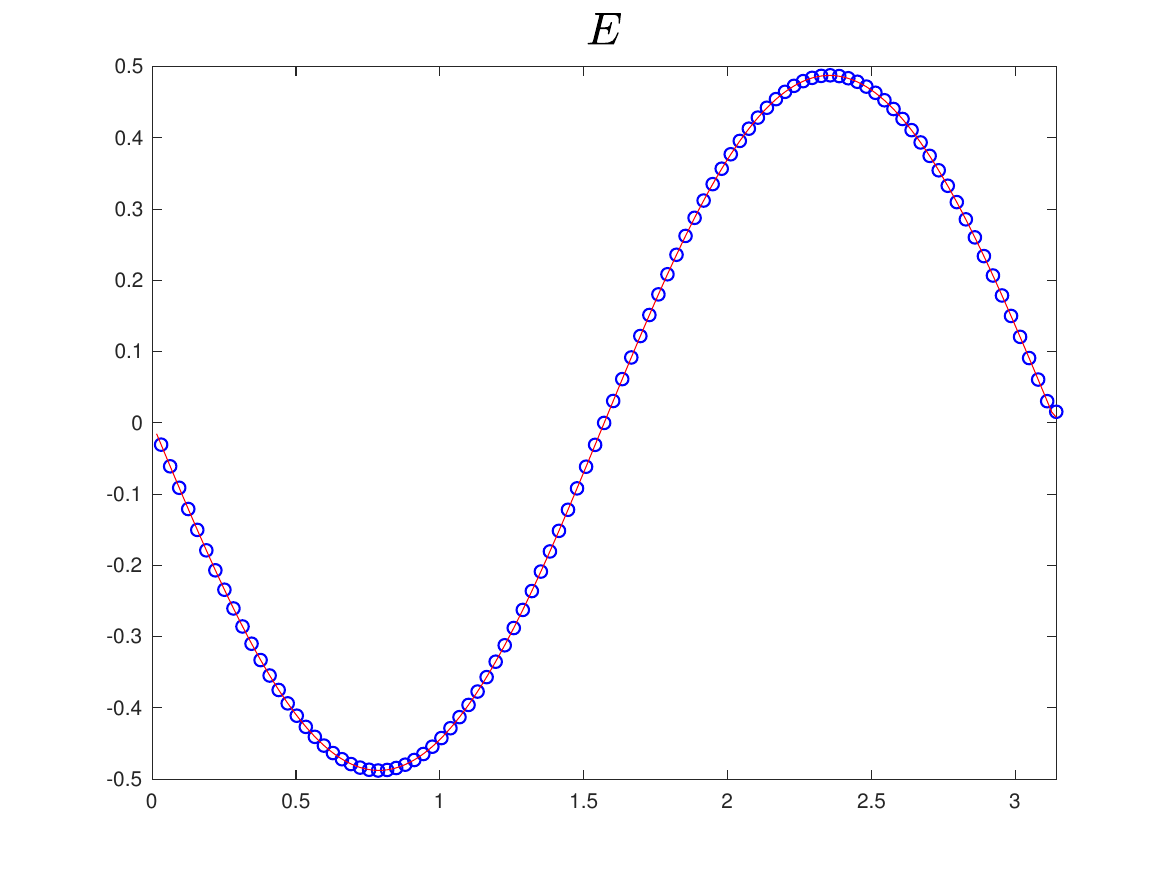}
\centering
\includegraphics[width=0.49\linewidth]{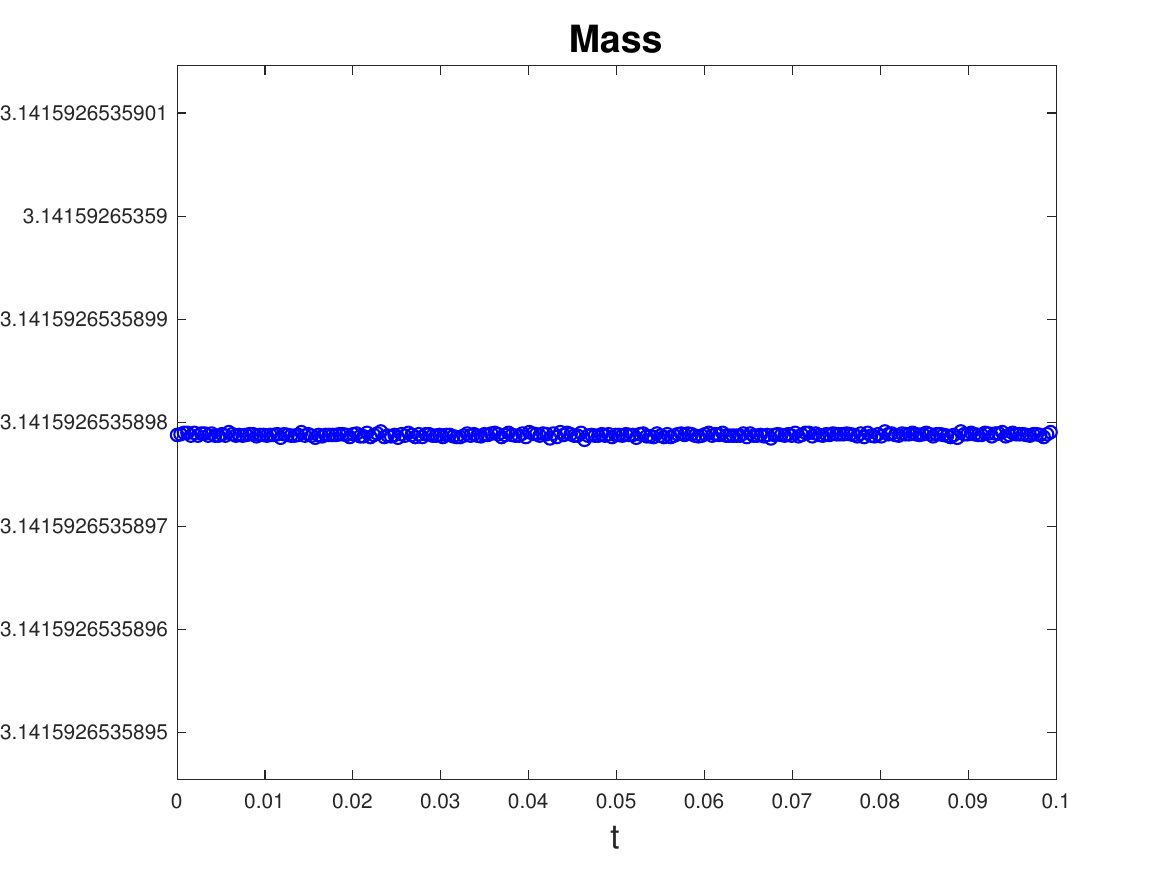}
\centering
\includegraphics[width=0.49\linewidth]{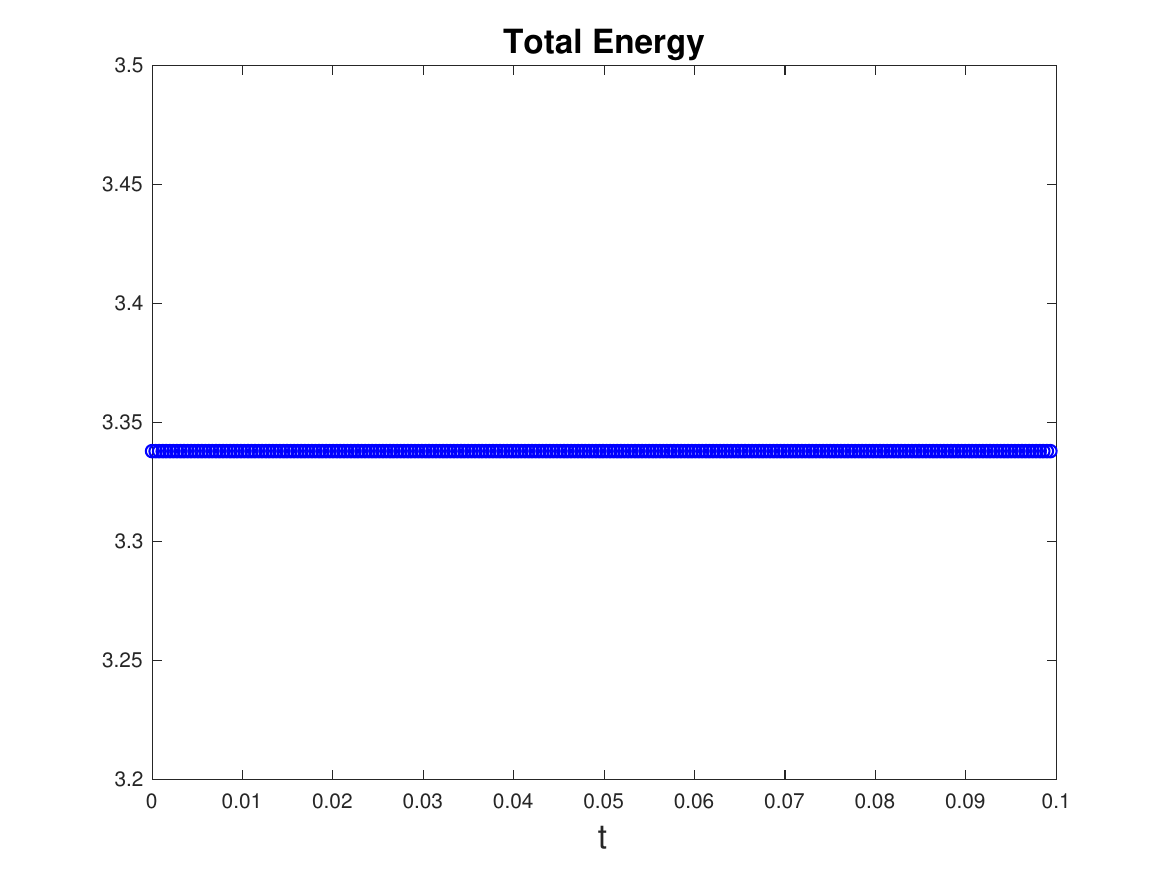}
\caption{$\varepsilon=0.05$. Numerical solutions $\rho, E$ of the Vlasov-Amp\'ere-Boltzmann equation at $t=0.1$ in the first row (blue circle uses $N_x=100$, 
$\Delta t=\Delta x/20$, red line is as reference solution using fine mesh $N_x=200$, $\Delta t= 10^{-4}$); 
mass and total energy with respect to time in the second row. }
\label{TestIII-b}
\end{figure}

\begin{remark} 
The schemes proposed in this section give the desired conservation 
property thanks to the use of moment equations. Here we obtain the moment
system first (so the right hand side vanishes) and then discretize it. 
If one obtains the moments from the discretized $f$ equation, due to the 
non-conservation of the approximate collision operator, the discrete 
moments are not necessarily conserved. This has already been addressed in \cite{JinYan} for a different purpose,
but here it serves the purpose as a generic strategy to devise conservative
schemes for general collision kinetic system. The only price paid is the
extra effort to solve the moment system. 
\end{remark}

\section{Conclusions and future work}
\label{sec:FW}

The micro-macro decomposition based method for multiscale kinetic equations has found many applications 
as an effective method to derive Asymptotic-Preserving schemes that work efficiently in all regimes, including both the kinetic and fluid regimes.  
However, so far it has been developed only for the BGK model. In this paper we extend it to general collisional kinetic equations, including the Boltzmann and the Fokker-Planck Landau equations. 
One of the difficulty in this formulation is the numerical stiff linearized collision operator, which needs to be treated implicitly thus becomes numerically difficult. Our  main idea is to use a relation between the (numerically stiff) linearized collision operator with the nonlinear quadratic ones, the latter's stiffness can be overcome using the BGK penalization method of Filbet and Jin for the Boltzmann, or the linear Fokker-Planck penalization method of Jin and Yan for the Fokker-Planck Landau equations. Such a scheme allows the computation of multiscale collisional kinetic equations efficiently in all regimes, including the fluid regime in which the fluid dynamic behavior can be correctly computed even without numerically resolving the small Knudsen number. It is implicit but can be implemented {\it explicitly}.  
 
This scheme preserves the moments (mass, momentum and energy) {\it exactly} due to the use of the macroscopic system which is naturally in a conservative form. We then utilize this conservation property for more general kinetic equations, using the Vlasov-Amp\`{e}re and Vlasov-Amp\`{e}re-Boltzmann systems as examples. The main idea is to evolve both the kinetic equation for the probability density distribution and the moment system, the later naturally induces a scheme that conserves exactly the moments numerically if they are physically conserved. This recipe is generic and applies to all kinetic equations.
 
Numerical examples demonstrate the conservation properties of our schemes, as well as it robustness in the fluid dynamic and mixed regimes. Notice that the numerical total energy exhibited an $O(10^{-3})$ error persistent for long time that coincides
with the order of magnitude of the numerical total energy error in the implementation of the Vlasov--Poisson--Landau system by 
operator splitting of a DG scheme for the collisionless Vlasov--Poisson advection coupled to the collisional conservative step in \cite{zhang2017conservative}. This observation opens an interesting problem of understanding how to diminish this computational error on obtaining the total energy evolution from the kinetic pdf that solves the Vlasov--Poisson with either Boltzmann or Landau collisional forms, 
by perhaps either imposing a conservation constraint in the kinetic simulation of our proposed scheme or to address operator splitting improvements in the approach used in \cite{zhang2017conservative}. 
 
In the numerical simulation, we use a second order space discretization and a first order IMEX temporal discretization. 
It would be nice to improve the first order time approximation and develop a fully second order scheme, for example, by adopting the method introduced in 
\cite{IMEX-HighOrder}. This will be done in a future work. 
To extend the micro-macro method for multi-dimensional problems also remain to be pursued. 
Here one needs to extend the staggered grid to higher dimension, a task that was investigated for hyperbolic systems of conservative laws \cite{JT} but yet to be studied for kinetic equations. 

\section*{Acknowledgement}
The authors would like to thank both referees for their helpful comments to improve this paper. 

\section*{Appendix: Details of Numerical Implementation}

Details of solving (\ref{U_discrete}) are shown below. In the case of $x\in\mathbb R$, $v\in\mathbb R^2$ ($d=2$), 
$$u_1 = \frac{1}{\rho}\int_{\mathbb R^d} f v_1\, dv, \qquad 
u_2 = \frac{1}{\rho}\int_{\mathbb R^d}  f v_2\, dv. $$
We then have
\begin{align*}
 &\displaystyle\frac{\partial\rho}{\partial t} + \partial_x F_1
 = - \varepsilon\, \partial_x \langle g \rangle, \\[4pt]
&\displaystyle \frac{\partial}{\partial t} (\rho u_1) + \partial_x F_2
 = - \varepsilon\, \partial_x \langle v_1^2\, g \rangle,  \\[4pt]
&\displaystyle \frac{\partial}{\partial t} (\rho u_2) + \partial_x  F_3
= - \varepsilon\, \partial_x \langle v_1 v_2\, g \rangle, \\[4pt]
&\displaystyle \frac{\partial E}{\partial t} + \partial_x F_4 = - \varepsilon\, \partial_x \langle v_1\, \frac{|v|^2}{2} g \rangle,
\end{align*}
where \begin{equation*}\label{F12} F_1 = \langle v_1 M \rangle, \qquad F_2 = \langle v_1^2\, M \rangle, \qquad F_3 = \langle v_1 v_2\, M \rangle, \qquad 
F_4 = \langle v_1\, \frac{|v|^2}{2}M \rangle. 
\end{equation*}
with $M$ associated with $\rho$, $u_1$, $u_2$, $T$ as defined in (\ref{Max}).

The kinetic formulation of the flux splitting (\ref{flux_split}) is given by
\begin{equation*}\label{FS} F_1^{\pm} = \langle v_1^{\pm} M \rangle, \qquad F_2 ^{\pm}= \langle v_1^{\pm}\, v_1 M \rangle, \qquad F_3^{\pm} = 
\langle v_1^{\pm}\, v_2 M \rangle, \qquad  F_4^{\pm} = \langle v_1^{\pm}\, \frac{|v|^2}{2} M \rangle, 
\end{equation*}
with $v_1^{\pm} = (v_1 \pm |v_1|)/2$, $v_2^{\pm} = (v_2 \pm |v_2|)/2$. 

\bibliographystyle{siam}
\bibliography{MM_Boltzmann.bib}

\begin{thebibliography}{10}

\bibitem{Bardos}
{\sc C.~Bardos, F.~Golse, and D.~Levermore}, {\em Fluid dynamic limits of
  kinetic equations. {I}. {F}ormal derivations}, J. Statist. Phys., 63 (1991),
  pp.~323--344.

\bibitem{MM-Lemou}
{\sc M.~Bennoune, M.~Lemou, and L.~Mieussens}, {\em Uniformly stable numerical
  schemes for the {B}oltzmann equation preserving the compressible
  {N}avier-{S}tokes asymptotics}, J. Comput. Phys., 227 (2008), pp.~3781--3803.

\bibitem{YD-Cheng1}
{\sc Y.~Cheng, A.~Christlieb, and X.~Zhong}, {\em Numerical study of the
  two-species {V}lasov-{A}mp\`{e}re system: {E}nergy-conserving schemes and the
  current-driven ion-acoustic instability}, Journal of Computational Physics,
  288 (2014).

\bibitem{YD-Cheng}
{\sc Y.~Cheng, A.~J. Christlieb, and X.~Zhong}, {\em Energy-conserving
  {D}iscontinuous {G}alerkin methods for the {V}lasov-{A}mp\`ere system}, J.
  Comput. Phys., 256 (2014), pp.~630--655.

\bibitem{dimarco2011exponential}
{\sc G.~Dimarco and L.~Pareschi}, {\em Exponential {R}unge-{K}utta methods for
  stiff kinetic equations}, SIAM Journal on Numerical Analysis, 49 (2011),
  pp.~2057--2077.

\bibitem{Filbet-Jin}
{\sc F.~Filbet and S.~Jin}, {\em A class of asymptotic-preserving schemes for
  kinetic equations and related problems with stiff sources}, J. Comput. Phys.,
  229 (2010), pp.~7625--7648.

\bibitem{gamba2014conservative}
{\sc I.~M. Gamba and J.~R. Haack}, {\em A conservative spectral method for the
  {B}oltzmann equation with anisotropic scattering and the grazing collisions
  limit}, Journal of Computational Physics, 270 (2014), pp.~40--57.

\bibitem{gamba2017fast}
{\sc I.~M. Gamba, J.~R. Haack, C.~D. Hauck, and J.~Hu}, {\em A fast spectral
  method for the {B}oltzmann collision operator with general collision
  kernels}, SIAM Journal on Scientific Computing, 39 (2017), pp.~B658--B674.

\bibitem{gamba2009spectral}
{\sc I.~M. Gamba and S.~H. Tharkabhushanam}, {\em {S}pectral-{L}agrangian
  methods for collisional models of non-equilibrium statistical states},
  Journal of Computational Physics, 228 (2009), pp.~2012--2036.

\bibitem{LFY-DG}
{\sc J.~Jang, F.~Li, J.-M. Qiu, and T.~Xiong}, {\em Analysis of asymptotic
  preserving {DG}-{IMEX} schemes for linear kinetic transport equations in a
  diffusive scaling}, SIAM J. Numer. Anal., 52 (2014), pp.~2048--2072.

\bibitem{JT}
{\sc G.-S. Jiang and E.~Tadmor}, {\em Non-oscillatory central schemes for
  multidimensional hyperbolic conservation laws}, SIAM Journal on Scientific
  Computing, 19 (1998), pp.~1892--1917.

\bibitem{jin1999efficient}
{\sc S.~Jin}, {\em Efficient asymptotic-preserving ({AP}) schemes for some
  multiscale kinetic equations}, SIAM Journal on Scientific Computing, 21
  (1999), pp.~441--454.

\bibitem{jin2010asymptotic}
\leavevmode\vrule height 2pt depth -1.6pt width 23pt, {\em Asymptotic
  preserving ({AP}) schemes for multiscale kinetic and hyperbolic equations: a
  review}, Lecture notes for summer school on methods and models of kinetic
  theory (M\&MKT), Porto Ercole (Grosseto, Italy),  (2010), pp.~177--216.

\bibitem{QinJin}
{\sc S.~Jin and Q.~Li}, {\em A {BGK}-penalization-based asymptotic-preserving
  scheme for the multispecies {B}oltzmann equation}, Numer. Methods Partial
  Differential Equations, 29 (2013), pp.~1056--1080.

\bibitem{JinYan}
{\sc S.~Jin and B.~Yan}, {\em A class of asymptotic-preserving schemes for the
  {F}okker-{P}lanck-{L}andau equation}, J. Comput. Phys., 230 (2011),
  pp.~6420--6437.

\bibitem{klar2001numerical}
{\sc A.~Klar and C.~Schmeiser}, {\em Numerical passage from radiative heat
  transfer to nonlinear diffusion models}, Mathematical Models and Methods in
  Applied Sciences, 11 (2001), pp.~749--767.

\bibitem{krall1973principles}
{\sc N.~Krall and A.~Trivelpiece}, {\em Principles of plasma physics.
  international series in pure and applied physics}, 1973.

\bibitem{Lemou-Note}
{\sc M.~Lemou}, {\em Relaxed micro-macro schemes for kinetic equations}, C. R.
  Math. Acad. Sci. Paris, 348 (2010), pp.~455--460.

\bibitem{Lemou-BC}
{\sc M.~Lemou and L.~Mieussens}, {\em A new asymptotic preserving scheme based
  on micro-macro formulation for linear kinetic equations in the diffusion
  limit}, SIAM J. Sci. Comput., 31 (2008), pp.~334--368.

\bibitem{LeVeque}
{\sc R.~J. LeVeque}, {\em Numerical methods for conservation laws}, Lectures in
  Mathematics ETH Z\"urich, Birkh\"auser Verlag, Basel, second~ed., 1992.

\bibitem{li2014exponential}
{\sc Q.~Li and L.~Pareschi}, {\em Exponential {R}unge-{K}utta for the
  inhomogeneous {B}oltzmann equations with high order of accuracy}, Journal of
  Computational Physics, 259 (2014), pp.~402--420.

\bibitem{Liboff}
{\sc R.~Liboff}, {\em Kinetic Theory: Classical, Quantum and Relativistic
  Descriptions}.

\bibitem{liu2016unified}
{\sc C.~Liu, K.~Xu, Q.~Sun, and Q.~Cai}, {\em A unified gas-kinetic scheme for
  continuum and rarefied flows {IV}: Full {B}oltzmann and model equations},
  Journal of Computational Physics, 314 (2016), pp.~305--340.

\bibitem{liu2010analysis}
{\sc J.-G. Liu and L.~Mieussens}, {\em Analysis of an asymptotic preserving
  scheme for linear kinetic equations in the diffusion limit}, SIAM Journal on
  Numerical Analysis, 48 (2010), pp.~1474--1491.

\bibitem{liu2006nonlinear}
{\sc T.-P. Liu, T.~Yang, S.-H. Yu, and H.-J. Zhao}, {\em Nonlinear stability of
  rarefaction waves for the {B}oltzmann equation}, Archive for rational
  mechanics and analysis, 181 (2006), p.~333.

\bibitem{liu2004boltzmann}
{\sc T.-P. Liu and S.-H. Yu}, {\em Boltzmann equation: micro-macro
  decompositions and positivity of shock profiles}, Communications in
  mathematical physics, 246 (2004), pp.~133--179.

\bibitem{mieussens2000discrete}
{\sc L.~Mieussens}, {\em Discrete-velocity models and numerical schemes for the
  {B}oltzmann-{BGK} equation in plane and axisymmetric geometries}, Journal of
  Computational Physics, 162 (2000), pp.~429--466.

\bibitem{Luc}
\leavevmode\vrule height 2pt depth -1.6pt width 23pt, {\em On the asymptotic
  preserving property of the unified gas kinetic scheme for the diffusion limit
  of linear kinetic models}, J. Comput. Phys., 253 (2013), pp.~138--156.

\bibitem{DG-BP}
{\sc J.~Morales~Escalante, I.~Gamba, A.~Majorana, Y.~Cheng, C.-W. Shu, and
  J.~Chelikowsky}, {\em Discontinuous {G}alerkin deterministic solvers for a
  {B}oltzmann-{P}oisson model of hot electron transport using an averaged
  empirical pseudopotential band}, Comput. Methods Appl. Mech. Engrg, 321
  (2017), pp.~209--234.

\bibitem{mouhot2006fast}
{\sc C.~Mouhot and L.~Pareschi}, {\em Fast algorithms for computing the
  {B}oltzmann collision operator}, Mathematics of computation, 75 (2006),
  pp.~1833--1852.

\bibitem{pareschi2000numerical}
{\sc L.~Pareschi and G.~Russo}, {\em Numerical solution of the {B}oltzmann
  equation {I}: Spectrally accurate approximation of the collision operator},
  SIAM journal on numerical analysis, 37 (2000), pp.~1217--1245.

\bibitem{IMEX-HighOrder}
\leavevmode\vrule height 2pt depth -1.6pt width 23pt, {\em Implicit-{E}xplicit
  {R}unge-{K}utta schemes and applications to hyperbolic systems with
  relaxation}, J. Sci. Comput., 25 (2005), pp.~129--155.

\bibitem{xu2010unified}
{\sc K.~Xu and J.-C. Huang}, {\em A unified gas-kinetic scheme for continuum
  and rarefied flows}, Journal of Computational Physics, 229 (2010),
  pp.~7747--7764.

\bibitem{zhang2017conservative}
{\sc C.~Zhang and I.~M. Gamba}, {\em A conservative scheme for {V}lasov
  {P}oisson {L}andau modeling collisional plasmas}, Journal of Computational
  Physics, 340 (2017), pp.~470--497.

\end{thebibliography}

\end{document}